\documentclass[letterpaper,11pt]{article}
\usepackage{ece-paper}
\graphicspath{ {./figures/} }

\newcommand{\stepa}[1]{\overset{\rm (a)}{#1}}
\newcommand{\stepb}[1]{\overset{\rm (b)}{#1}}
\newcommand{\stepc}[1]{\overset{\rm (c)}{#1}}
\newcommand{\stepd}[1]{\overset{\rm (d)}{#1}}

\newcommand{\eqlaw}{\overset{{\rm law}}{=}}

\newcommand{\termI}{(\text{I})}
\newcommand{\termII}{(\text{II})}

\newcommand{\btheta}{\boldsymbol{\theta}} 
\newcommand{\balpha}{\boldsymbol{\alpha}} 
\newcommand{\bbeta}{\boldsymbol{\beta}} 
\newcommand{\bxi}{\boldsymbol{\xi}} 

\newcommand{\fnorm}[1]{\|#1\|_{\rm F}}
\newcommand{\opnorm}[1]{\lnorm{#1}{\rm op}}
\newcommand{\Opnorm}[1]{\|#1\|_{\rm op}}

\newcommand{\polylog}{{\mathrm{polylog}}}
\newcommand{\sech}{{\mathrm{sech}}}

\newcommand{\Lip}{\mathrm{Lip}}

\newcommand{\thetanorm}{\|\theta_*\|}
\newcommand{\thetaub}{r}
\newcommand{\Tstar}{T_{\star}}

\newcommand{\thetaMLE}{\hat{\theta}_{{\rm MLE}}}

\usepackage{algorithm,algorithmic}

\newcommand{\Var}{\mathrm{Var}}
\newcommand{\Span}{\mathrm{span}}
\newcommand{\otheta}{\theta^{\texttt{+}}}
\newcommand{\utheta}{\theta^{\texttt{-}}}
\newcommand{\oepsilon}{\epsilon^{\texttt{+}}}
\newcommand{\uepsilon}{\epsilon^{\texttt{-}}}
\newcommand{\oalpha}{\alpha^{\texttt{+}}}
\newcommand{\ualpha}{\alpha^{\texttt{-}}}

\newcommand{\ttheta}{\tilde{\theta}}

\newcommand{\talpha}{\tilde{\alpha}}

\newcommand{\tb}{{\tilde{b}}}
\newcommand{\ty}{{\tilde{y}}}

\newcommand{\tY}{{\tilde{Y}}}

\renewcommand{\tilde}{\widetilde}

\usetikzlibrary{positioning}

\usepackage{soul}

\title{Randomly initialized EM algorithm for two-component Gaussian mixture achieves near optimality in $O(\sqrt{n})$ iterations}

\author{Yihong Wu and Harrison H.~Zhou\thanks{The authors are with the Department of Statistics and Data Science, Yale University, New Haven, CT, email: 
\url{{yihong.wu,huibin.zhou}@yale.edu}. The results of this paper were first presented in part at Joint Statistical Meetings (JSM), Vancouver BC, Canada, July 2018 \cite{WZ-JSM18}. Y.~Wu is supported in part by the NSF grants CCF-1527105, CCF-1900507, an NSF CAREER award CCF-1651588, and an Alfred Sloan fellowship.
H.~H.~Zhou is supported in part by NSF grants DMS 1811740, DMS 1918925, and NIH grant 1P50MH115716.
}}

\date{\today}

\begin{document}
\maketitle

\begin{abstract}


We analyze the classical EM algorithm for parameter estimation in the symmetric two-component Gaussian mixtures in $d$ dimensions. We show that, 
even in the absence of any separation between components, provided that the sample size satisfies $n=\Omega(d \log^3 d)$,
 the randomly initialized EM algorithm converges to an estimate in at most $O(\sqrt{n})$ iterations with high probability, which is at most 
$O((\frac{d \log^3 n}{n})^{1/4})$ in Euclidean distance from the true parameter and within logarithmic factors of the minimax rate of $(\frac{d}{n})^{1/4}$. 
Both the nonparametric statistical rate and the sublinear convergence rate are direct consequences of the zero Fisher information in the worst case.
Refined pointwise guarantees beyond worst-case analysis and convergence to the MLE are also shown under mild conditions.

This improves the previous result of  Balakrishnan et al \cite{BWY17} which requires strong conditions on both the separation of the components and the quality of the initialization,
and that of  Daskalakis et al \cite{DTZ17} which requires sample splitting and restarting the EM iteration.

\end{abstract}

\tableofcontents

\section{Introduction}
\label{sec:intro}

The Expectation-Maximization (EM) algorithm \cite{DLR1977} is a powerful heuristic aiming at approximating the maximal likelihood estimator (MLE) in the presence of latent variables. The general setting can be described as follows:
Let $(X,Y)$ be random variables distributed according to some parametrized joint distribution with density $p_{\theta_*}(x,y)$.
Observing $Y$ (but not the latent $X$), the goal is to estimate the true parameter $\theta_*$.
Let $p_\theta(y) = \int p_\theta(x,y) dx$ denote the marginal density of $Y$.
Given $Y=y$, the MLE for $\theta_*$ is
		\begin{equation}
		\hat \theta_{\text{MLE}} \in \arg\max_\theta \log p_\theta(y),
		\label{eq:MLE}
		\end{equation}
which is frequently expensive to compute due to the non-convexity of the likelihood and the computational cost of the marginalization.
		To this end, the EM algorithm was proposed as an iterative algorithm to approximate the MLE. 
		Given the current estimate $\theta_0$, the next estimate $\theta_1$ is obtained by executing the following two steps:
		\begin{itemize}
			\item ``E step'': compute
	\begin{equation}
	Q(\theta|\theta_t) \triangleq \int p_{\theta_t}(x|y) \log p_{\theta}(x,y) dx
	\label{eq:Estep}
	\end{equation}
				
			\item ``M step'': update
			\begin{equation}
			\theta_{t+1} = \arg\max_{\theta} Q(\theta|\theta_t).
			\label{eq:Mstep}
			\end{equation}
			\end{itemize}
			The algorithm then proceeds by iterating the following two steps and generates a sequence of estimators $\{\theta_t\colon t \geq 0\}$.			
			The interpretation of this methodology is that \prettyref{eq:Mstep} is equivalent to maximizing the following lower bound of the log-likelihood:
		\[
		\int p_{\theta_t}(x|y) \log \frac{p_{\theta}(x,y)}{p_{\theta_t}(x|y)} dx = \log p_\theta(y) - D(p_{\theta_t}(\cdot|y) \| p_{\theta}(\cdot|y))
		\]
			where $D(\cdot\|\cdot)$ denote the Kullback-Leibler (KL) divergence. Consequently,
			\[
			\log p_\theta(y)-\log p_{\theta_t}(y) \geq Q(\theta|\theta_t)-Q(\theta_t|\theta_t)
			\]
			for any $\theta$, and hence the likelihood along the EM trajectory $\{\theta_t\}$ is non-decreasing.

\subsection{Gaussian mixture model}
\label{sec:intro-model}
We consider the symmetric two-component Gaussian mixture (2-GM) model in $d$ dimensions:
\begin{equation}
P_\theta = \frac{1}{2} N(-\theta,I_d) + \frac{1}{2} N(\theta,I_d),
\label{eq:Ptheta}
\end{equation}
which corresponds to two equally weighted clusters centered at $\pm \theta$ respectively.
Recall that 
$\cosh(x)=\frac{e^x+e^{-x}}{2}$, $\sinh(x)=\frac{e^x-e^{-x}}{2}$, and  $\tanh(x)=\frac{\sinh(x)}{\cosh(x)}$.
The density function of $P_\theta$ is
\begin{equation}
p_\theta(y) \triangleq \frac{1}{2}[\varphi(y-\theta)+\varphi(y+\theta)] = \exp(-\|y\|^2/2) \varphi(\theta) \cosh \iprod{y}{\theta}.
\label{eq:ptheta}
\end{equation}
where $\varphi$ denotes the standard normal density in $\reals^d$, $\|\cdot\|$ denotes the Euclidean norm.

Let $\theta_*\in\reals^d$ denote the ground truth. Given iid samples $Y=(Y_1,\ldots,Y_n) \iiddistr P_{\theta_*}$, the goal is to estimate  
$\theta_*$ up to a global sign flip, under the following loss function:
\[
\ell(\hat\theta, \theta) = \min\{\|\hat\theta -\theta\|, \|\hat\theta +\theta\|\}.
\]
Here the latent variables $(X_1,\ldots,X_n)$ correspond to the labels of each sample, which are iid and equally likely to be $\pm 1$ (Rademacher). Then we have
\begin{equation}
Y_i = X_i \theta_* + Z_i
\label{eq:latent}
\end{equation}
where $Z_i \iiddistr N(0,I_d)$ and are independent of $X_i$'s.
Since 
\[
p_\theta(x,y) \propto e^{-\frac{1}{2} \sum_{i=1}^n \|y_i - x_i \theta\|^2} \propto e^{-\frac{1}{2} \sum_{i=1}^n \|\theta\|^2 - \iprod{x_i y_i}{\theta}},
\]
 the M-step in \prettyref{eq:Mstep} simplifies to
\begin{align*}
			\theta_{t+1} &~ = \arg\min_{\theta} \sum_{i=1}^n \sum_{x_i\in\{\pm\}} \|y_i - x_i \theta\|^2 p_{\theta_t}(x_i|y) \\
			&~ = \arg\min_{\theta} \sth{n \|\theta\|^2 - \iprod{\theta}{\sum_{i=1}^n y_i \Expect_{\theta_t}[X_i|Y_i=y_i]}} \\
			&~ = \frac{1}{n} \sum_{i=1}^n y_i \Expect_{\theta_t}[X_i|Y_i=y_i],
\end{align*}
where the conditional mean is given by 
\begin{equation}
\Expect_{\theta}[X|Y=y] = \tanh \iprod{\theta}{y}.
\label{eq:condmean}
\end{equation}
Thus, specialized to the symmetric 2-GM model, the EM algorithm takes the following form:
\begin{equation}
\theta_{t+1} = f_n(\theta_{t})
\label{eq:EMsample}
\end{equation}
where 
\begin{equation}
f_n(\theta) \triangleq \Expect_n[Y \tanh \iprod{\theta}{Y}] \triangleq 
\frac{1}{n}\sum_{i=1}^n Y_i \tanh \iprod{\theta}{Y_i}.
\label{eq:fn}
\end{equation}
In the case of infinite samples ($n\diverge$), \prettyref{eq:fn} reduces to the following
\begin{equation}
f(\theta) \triangleq \Expect[Y \tanh \iprod{\theta}{Y}], \quad Y\sim P_{\theta_*}.
\label{eq:f}
\end{equation}
We refer to \prettyref{eq:fn} and \prettyref{eq:f} as the \emph{sample version} and the \emph{population version} of the EM map, respectively.

In the special case of symmetric Gaussian mixture,\footnote{In fact, this holds for any Gaussian mixture distribution, where the center of each component has the same Euclidean norm.} EM algorithm can also be interpreted as maximizing the likelihood by means of gradient ascent with \emph{constant step size}.
Indeed, denote the average $n$-sample log likelihood by
	\begin{equation}
	\ell_n(\theta) \triangleq \frac{1}{n} \sum_{i=1}^n \log p_\theta(Y_i) = \Expect_n[\log p_\theta(Y)]
	\label{eq:LL}
	\end{equation}
	and its population version by
	\begin{equation}
\ell(\theta) \triangleq \Expect[\log p_\theta(Y)], \quad Y\sim P_{\theta_*}.
\label{eq:LL-pop}
\end{equation}
	Since $\nabla \log \ell_n(\theta) = 
	\Expect_n[\nabla_\theta \log p_\theta(Y)] = -\theta + \Expect_n[Y \tanh \iprod{\theta}{Y}],
	$
	the EM iteration \prettyref{eq:EMsample} can be written as
	in the following gradient ascent form (with step size equal to one)
	\begin{equation}
	\theta_{t+1} = \theta_{t} + \nabla \ell_n(\theta_t).
		\label{eq:EM-GD}
	\end{equation}


Recently there is a sequence of work on the performance of the EM algorithm
\cite{BWY17,XHM16,DTZ17,jin2016local},  in particular, on the global convergence of the population (infinite sample size) version. For finite samples, either strong conditions on the initializations and the separation need to be assumed, or certain variants of the algorithm (such as sample splitting or restart) need to be executed.
Despite these progress, the performance guarantee of the classical EM algorithm remains not fully understood, especially with random initializations, which are widely adopted in practice.
The main focus of this paper is to provide statistical and computational guarantees for the randomly initialized EM algorithm in high dimensions, thereby assessing the optimality of the EM estimate and the number of iterations needed to reach the statistical optimum. We do so in the simple symmetric 2-GM model.

\subsection{Main results}
\label{sec:main}

We focus on the regime of bounded $\thetanorm$. This is the most interesting case for parameter estimation, wherein consistent clustering is impossible but accurate estimation of $\theta_*$ is nevertheless possible. In fact, for the purpose of parameter estimation, it is not necessary to impose any separation between the two clusters, since the parameter $\theta_*$ is perfectly identifiable even when $\theta_*=0$ is allowed, in which case the data are simply generated from a single standard Gaussian component.

Formally, throughout the paper we assume that 
\begin{equation}
\thetanorm \leq \thetaub
\label{eq:bounded}
\end{equation}
for some constant $\thetaub$.

\begin{theorem}
\label{thm:main}
There exist constants $C,C_0$ depending only on $\thetaub$, such that 
the following holds. 
Assume that $n \geq C  d \log^3 d$.
Initialize the EM iteration \prettyref{eq:EMsample} with 
\begin{equation}
\theta_0 = C_0 \pth{\frac{d}{n} \log n}^{1/4} \eta_0,
\label{eq:random-init-main}
\end{equation}
where $\eta_0$ is drawn uniformly at random from the unit sphere $S^{d-1}$. 
For any $\theta_*\leq \thetaub$,  with probability $1-o_{n}(1)$,  
\begin{equation}
\ell(\theta_t,\theta_*) \leq C \pth{\frac{d \log^3 n}{n}}^{1/4}
\label{eq:main}
\end{equation}
for all $t \geq C \sqrt{n}$.
\end{theorem}

\prettyref{thm:main} provides a statistical and computational guarantee for the EM algorithm for all $\theta_*$, with the worst case occurring for $\theta_*$ close to zero.
In fact, if $\|\theta_*\| = O((\frac{d}{n})^{1/4})$, the 2-GM model is statistically indistinguishable from the standard normal model.
The following result is a refined version of \prettyref{thm:main} under the modest assumption that $\theta_*$ is slightly bounded away from zero, which also shows the convergence to the MLE:
\begin{theorem}
\label{thm:main-pointwise}
In the setting of \prettyref{thm:main}, assume in addition that $\thetanorm \geq (C \frac{d \log^3 n}{n})^{1/4}$.
		Then, with probability at least $1-o_n(1)$, 
		\begin{equation}
	\ell(\theta_t,\theta_*) \leq \frac{C}{\thetanorm} \sqrt{\frac{d \log n}{n}}
	\label{eq:main-pointwise}
	\end{equation}
	holds for all $t \geq \frac{C \log n}{\thetanorm^2}$ and, furthermore, $\lim_{t\diverge} \theta_t$ exists and coincides with $\thetaMLE$, the unique (up to a global sign change) global maximizer of the likelihood \prettyref{eq:LL} and 
		$\ell(\theta_t,\thetaMLE) = o(\frac{1}{n})$ for all $t \geq \frac{C \log n}{\thetanorm^2}$.
\end{theorem}

The statistical optimality of the EM estimate can be seen by comparing 
Theorems \ref{thm:main} and 
\ref{thm:main-pointwise}  with the following minimax results (which are consequences of \prettyref{thm:minimax} in \prettyref{app:minimax}): for any $r\gtrsim 1$ and $n \gtrsim d$, we have
	\begin{equation}
	\inf_{\hat\theta} \sup_{\|\theta_*\|\leq r} \Expect_{\theta^*}[\ell(\hat\theta,\theta_*)] \asymp \pth{\frac{d}{n}}^{\frac{1}{4}}.
	\label{eq:minimax-ball}
	\end{equation}	
	where the infimum is take over all estimators $\hat\theta$ as a function of $Y_1,\ldots,Y_n \iiddistr P_{\theta_*}$.
Furthermore, for any fixed $\thetanorm = s \lesssim 1$ and $n \gtrsim d$, we have
		\begin{equation}
	\inf_{\hat\theta} \sup_{\|\theta_*\| = s} \Expect_{\theta^*}[\ell(\hat\theta,\theta_*)] \asymp \min\sth{s,\frac{1}{s} \sqrt{\frac{d}{n}}}.
	\label{eq:minimax-sphere}
	\end{equation}		
Comparing \prettyref{eq:minimax-ball} with \prettyref{eq:main}, we conclude that the performance of the EM algorithm is within logarithmic factors\footnote{In the one-dimensional case, it is possible to show that the EM algorithm attains the minimax rate \prettyref{eq:minimax-ball} without logarithmic factors; see \prettyref{cor:em1d} in \prettyref{sec:onedim}.} of the minimax rate, which can be attained in at most $O(\sqrt{n})$ iterations in the worst case.
In addition, \prettyref{eq:minimax-sphere} shows that the transition from the worst-case rate $(\frac{d}{n})^{1/4}$ to the parametric rate $\frac{1}{\thetanorm} \sqrt{\frac{d}{n}}$
occurs when $\thetanorm$ exceeds $(\frac{d}{n})^{1/4}$, in which case the more refined guarantee \prettyref{eq:main-pointwise} demonstrates the near-optimality of the EM algorithm that is adaptive to $\thetanorm$.

We pause to clarify that the main objective of this paper is not to exhibit nearly minimax optimal methods, as other procedures (e.g., spectral method; cf.~\prettyref{app:minimax}) are known to achieve the minimax rate \prettyref{eq:minimax-ball} without extra logarithmic factors, but rather to show 
the popular EM algorithm with a single random initialization achieves near optimality and, furthermore, approaches the MLE. Compared to spectral methods, the statistical advantage of the EM algorithm is due to its asymptotic efficiency which is inherited from the MLE.


We conclude this subsection with a remark interpreting the results of the preceding theorems:

\begin{remark}[Statistical  and computational consequences of flat likelihood]
	In \prettyref{thm:main}, the statistical estimation rate $O((\frac{d}{n})^{1/4})$ which is slower than the typical parametric rate. Furthermore, the convergence rate is in fact $O(\frac{1}{\sqrt{t}})$ which is much slower than the typical linear convergence rate that is exponential in $t$. Both guarantees are tight in the worst case which occurs when $\|\theta_*\| = O((\frac{d}{n})^{1/4})$, and both phenomena are due to the zero curvature of log likelihood function.
	To explain this, let us consider the simple setting of one dimension and $\theta_*=0$. 




\begin{itemize}
	\item 
	\emph{Vanishing Fisher information and nonparametric rate}: 
	When $\theta_*=0$, a simple Taylor expansion shows that the population likelihood \prettyref{eq:LL-pop} satisfies $\ell_n(\theta)=\ell_n(0) - \frac{1}{4} \theta^4 + O(\theta^6)$ when $\theta \to0$, corresponding to the flat maxima at $\theta=0$ at as shown in \prettyref{fig:flat}. In particular, the Fisher information is zero, resulting in an  estimation rate 
slower than the typical rate $\sqrt{d/n}$ for parametric models.
	Furthermore, for $\theta_* \neq 0$, the Fisher information behaves as $\Theta(\theta_*^2)$ (cf.~\prettyref{rmk:rate}). Therefore \prettyref{eq:main-pointwise} shows that the EM algorithm achieves the local minimax rate within logarithmic factors.
	
	\item 
	\emph{Non-contraction and sub-linear convergence rate}: 
	In typical analysis of iterative methods, linear convergence rate is a direct consequence of contractive mapping theorem. This however fails for the case of $\theta_*=0$. Indeed, using \prettyref{eq:EM-GD} we obtain that the population EM map $f(\theta)$ satisfies $f(\theta) = \theta - \theta^3 + O(\theta^5)$ with $f'(0)=1$. Thus 
the EM iteration roughly behaves as $\theta_{t+1} \approx \theta_t - \theta_t^3$. Despite the non-strict contraction, the iteration nevertheless converges monotonically to the unique fixed point at zero (see \prettyref{fig:contraction}); however, the resulting convergence rate is $O(\frac{1}{\sqrt{t}})$ (cf.~\prettyref{lmm:rate} in \prettyref{app:lemma}). 
This gives theoretical quantification of the slow convergence rate of EM algorithm for poorly separated Gaussian mixtures, which has been widely observed in practice \cite{RW1984,KX2003}. 
\end{itemize}

\end{remark}

%

\begin{figure}[ht]
	\centering
	\subfigure[Flat minimum of the negative log likelihood.]%
	{\label{fig:flat} \includegraphics[width=0.4\columnwidth]{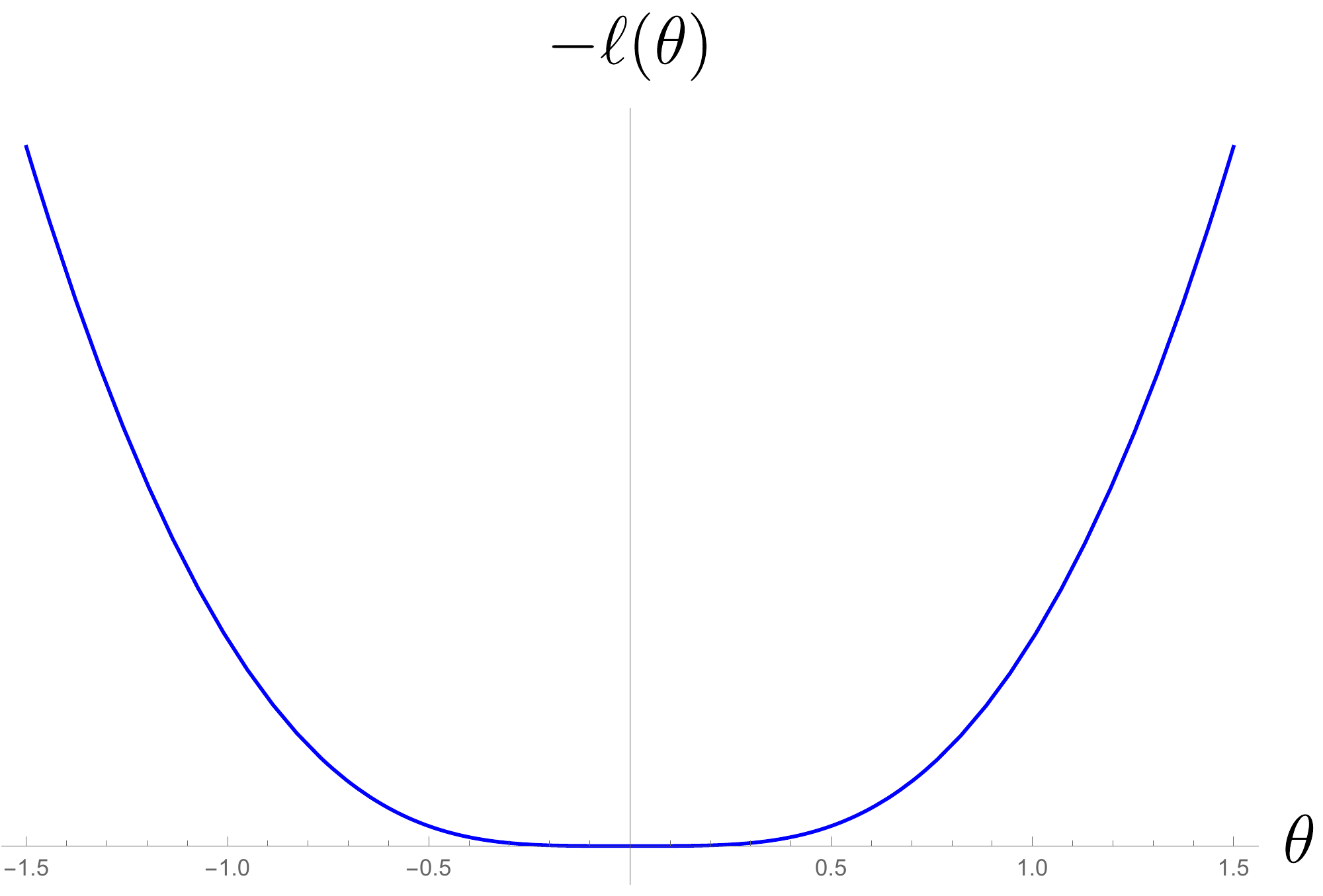}}
	\subfigure[Nonlinear contraction of $f(\theta)$ and the resulting sublinear rate of convergence.]%
	{\label{fig:contraction} \includegraphics[width=0.4\columnwidth]{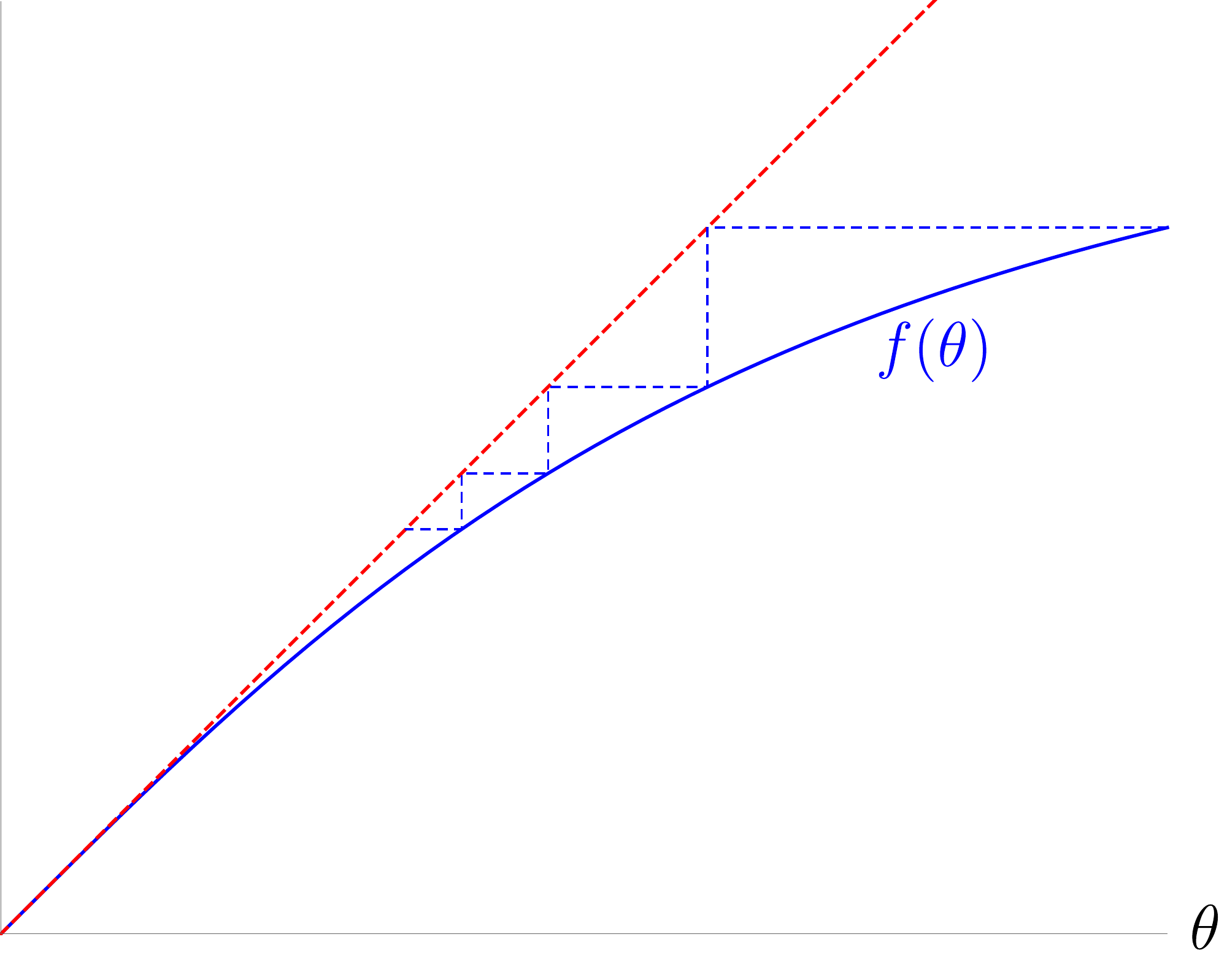}}
	\caption{Population version of the negative log likelihood and the EM map for $\theta_*=0$.}
	\label{fig:zerotheta}
\end{figure}

\subsection{Related work}
	\label{sec:related}
	
		Since the original paper \cite{DLR1977}, the EM algorithm has been widely used in Gaussian mixture models \cite{RW1984,XJ1996}. As can be seen from its gradient ascent interpretation \prettyref{eq:EM-GD}, a limiting point of the EM iteration is only guaranteed to be a critical point of the likelihood function rather than the global MLE. Various techniques for choosing the initialization has been proposed (cf.~the survey \cite{KX2003} and the references therein); 
	however, in practice random initializations are often preferred due to its simplicity over more costly approaches such as spectral methods \cite{biernacki2003choosing}.
		Furthermore, it is well-known in practice \cite{RW1984,KX2003} that the convergence of the EM iteration can be very slow when the components are not well separated, agreeing with the theoretical findings in \prettyref{thm:main} and \prettyref{thm:main-pointwise}.
	
	Recently there is a renewed interest on the EM algorithm in high dimensions from both statistical and optimization perspectives. General conditions (such as strong concavity and smoothness) are given in \cite{BWY17}	to guarantee the local convergence of the EM algorithm as well as its statistical performance. 
	Particularized to the simple 2-GM model \prettyref{eq:Ptheta}, \cite[Corollary 2]{BWY17} shows that if $\thetanorm$ exceeds some large constant and the initialization satisfies $\|\theta_0 - \theta_*\| \leq \frac{1}{4} \thetanorm$, then with probability $1-\delta$ the EM iteration converges exponentially fast to a neighborhood at $\theta_*$ of radius $\sqrt{\frac{C d }{n} \log \frac{1}{\delta}}$ for some constant $C$ depending on $\thetanorm$. 
	There are two major distinctions between \cite{BWY17}	 and the current paper:
	First, the requirement on the initialization in \cite{BWY17} is very strong, which implies that $\theta_0$ has a non-trivial angle with $\theta_*$ and clearly cannot be afforded by random initializations. 
	Second, to bound the deviation between the sample EM trajectory and its population counterpart, \cite{BWY17} proved that 
	\[
	\sup_{\|\theta\|\leq C}\|f_n(\theta)-f(\theta)\| = \tilde O\pth{\sqrt{\frac{d}{n}}}
	\]
	with high probability, where $\tilde O(\cdot)$ hides logarithmic factors. 
	Such a concentration inequality in terms \emph{absolute deviation} is too weak to yield the sharp rates in \prettyref{thm:main} and \ref{thm:main-pointwise} even in one dimension. Instead, 
in order to obtain the optimal  statistical and computational guarantees,	it is crucial to bound the \emph{relative deviation} and show that with high probability, 
	\begin{equation}
	\sup_{\|\theta\|\leq C}\frac{\|f_n(\theta)-f(\theta)\|}{\|\theta\|} = \tilde O\pth{\sqrt{\frac{d}{n}}}
	\label{eq:relative-dev}
	\end{equation}	
	i.e., $f_n-f$ is $\tilde O(\sqrt{\frac{d}{n}})$-Lipschitz, 
	the reason being that when the iterates are close to zero, the finite-sample deviation is proportionally small as well.
In addition, in \prettyref{sec:mle} we show that the EM iterations converge to the MLE under mild conditions.

The global convergence of the population EM iterates has been analyzed in \cite{XHM16,DTZ17}. 
 The following deterministic result was shown: 
Provided that the initial value $\theta_0$ is not orthogonal to $\theta_*$, 
  the population version of the EM iteration, that is, the sequence \prettyref{eq:EMsample} with $f_n$ replaced by $f$, 
converges to the global maximizer of the population log likelihood $\ell$ in \prettyref{eq:LL-pop},
namely, 
$\theta_*$ (resp.~$-\theta_*$) if $\Iprod{\theta_0}{\theta_*} > 0$ (resp.~$<0$).
If $\Iprod{\theta_0}{\theta_*} = 0$, then the population EM iteration converges to $0$, the unique saddle point of $\ell$.
For the sample EM, \cite[Theorem 7]{XHM16} showed that when the dimension and $\theta_*$ are fixed, the difference of the sample and population EM iteration vanishes in the double limit of $t\to\infty$ followed by $n\to\infty$; neither finite-sample nor finite-iteration guarantees are provided. 
As for high dimensions, a variant of the EM algorithm using sampling splitting is analyzed in \cite{DTZ17} consisting of two steps: First, run EM with a random and sufficiently small initialization for $\Theta(\frac{\log d}{\thetanorm^2})$ iterations. Next, renormalize the resulting estimate so that its norm is a large constant, and continue to run EM for another $\Theta(\frac{1}{\thetanorm^2} \log \frac{1}{\epsilon})$ iterations. The final output achieves a loss of $\epsilon$ with high probability provided that each iteration operates on a fresh batch of $\tilde \Theta(\frac{d}{\epsilon^2 \thetanorm^4})$ samples. 
The use of sampling splitting conveniently ensures independence among iterations and circumvents the major difficulty of analyzing the entire trajectory;
however, for the desired accuracy of $\epsilon = O(\frac{1}{\thetanorm} \sqrt{d/n})$, the total number of samples is $\tilde \Theta(\frac{n}{\thetanorm^4})$, which far exceeds $n$ when $\thetanorm$ is small.

Based on the population results in \cite{XHM16}, \cite{mei2018landscape} showed that if $\thetanorm$ is at least a constant, the landscape of the log likelihood $\ell_n$ is close to that of the population version (in terms of the critical points and the Hessian).
Specifically, \cite[Theorem 8]{mei2018landscape} showed the following: 
There exist constants $C,C'$ depending on $\thetanorm$ and $\delta$, such that if $n \geq C d\log d$, then with probability $1-\delta$, $\ell_n$ has two local maxima in the ball $B(0,C')$, which are within Euclidean distance $C \sqrt{\frac{d \log n}{n}}$ of $\pm \theta_*$.
As a corollary of the empirical landscape analysis, with appropriately chosen parameters and initialized from any point in $B(0,C')$, standard trust-region method (cf.~e.g.~\cite[Algorithm 6.1.1]{conn2000trust}) is guaranteed to converge to a local maximizer of $\ell_n$.
It should be noted that trust-region method is a second-order method using the Hessian information, which is more expensive than first-order methods such as gradient descent including the EM algorithm \prettyref{eq:EMsample}. Furthermore, the number of iterations needed to reach the statistical optimum is unclear.

On the technical side, the main difficulty of analyzing a sample-driven iterative scheme, such as \prettyref{eq:EMsample}, is the dependency between the iterates $\{\theta_t\}$ and the data, since each iteration takes one pass over the same set of samples. 
Of course, one can conduct a uniform analysis by taking a supremum over the realization of $\theta_t$; however, since the supremum is over a $d$-dimensional space, the resulting bound is too crude to characterize the growth of the ``signal'' $\Iprod{\theta_*}{\theta_t}$, which is very close to zero initially (that is, $O_P(\frac{1}{\sqrt{d}})$, due to random initialization). 
It is for this reason that the analysis is significantly more challenging than those using sample splitting such as \cite{BWY17,DTZ17}, which sidesteps the difficulty of dependency.
Furthermore, such \emph{trajectory analysis}, which tracks the signal growth from random initializations, does not follow from landscape analysis.

In this vein, the most related to the current paper is the recent seminal work \cite{chen2018gradient} on analyzing gradient descent for nonconvex phase retrieval with random initializations, where the goal is to recover a $d$-dimensional signal $x_*$ from noiseless quadratic measurements $\Iprod{a_i}{x_*}^2$ with iid Gaussian $a_i$. 
To overcome the aforementioned difficulties due to dependency, the main idea of \cite{chen2018gradient} is two-fold: 
In addition to the commonly used ``leave-one-sample-out'' method that analyzes the auxiliary iteration when one measurement is replaced by an independent copy, \cite{chen2018gradient} introduced a ``leave-one-coordinate-out'' auxiliary iteration where a single coordinate of each measurement vector is is replenished with a random sign. This is possible thanks to the rotational symmetry of the Gaussian measurement vectors, which allows one to assume, without loss of generality, that the ground truth is a coordinate vector. 
By comparing the auxiliary dynamics to the original one, one can effectively decouple the data and the iterates.
The idea of leave-one-coordinate-out turns out to be crucial in our analysis of randomly initialized EM, where we introduce an auxiliary sequence with a randomized label but otherwise identical to the original sequence; on the other hand, we are able to conduct the analysis without resorting to the leave-one-sample-out method.
Compared to \cite{chen2018gradient} which relies on the strong convexity of the population objective function and the resulting contraction of the iteration, for the EM algorithm since we do not assume $\theta_*$ is bounded away from zero, none of these applies which creates additional challenges for the analysis.

Finally, we note that the very recent and independent work \cite{dwivedi2018singularity,dwivedi2019challenges} obtained a tight analysis of the performance of EM algorithm when the true model is a single Gaussian and the postulated model is an over-specified Gaussian mixture. 
In particular, guarantees similar to \prettyref{thm:main} are shown for the special case of $\theta_*=0$, and both balanced and unbalanced mixture model are considered as well as the more general location-scale mixtures.

\subsection{Notations}
\label{sec:notations}
Throughout the paper, $c,C,C_0,C_1,\ldots,C',C''$ denote constants whose values vary from place to place and 
only depend on an upper bound on $\|\theta_*\|$, and the notation $\lesssim, \gtrsim, \asymp$  are within these constant factors.
Since we assume that $\thetanorm \leq \thetaub$ for some absolute constant $\thetaub$, these constant factors are absolute as well.

Let $\calL(X)$ denote the distribution (law) of a random variable $X$.
The generic notation $\Expect_n[\cdot]$ denotes the empirical average over $n$ iid samples, namely, $\Expect_n[f(X)] \triangleq \frac{1}{n}\sum_{i=1}^n f(X_i)$, where $X_i$'s are iid copies of $X$.
We say a random variable $X$ is $s$-subgaussian (resp.~$s$-subexponential) if 
$\|X\|_{\psi_2} \triangleq \inf\{t>0: \Expect e^{X^2/t^2} \leq 2\} \leq \sqrt{s}$ (resp.~$\|X\|_{\psi_1} \triangleq \inf\{t>0: \Expect e^{|X|/t} \leq 2\} \leq s$).

Let $\|x\|$ denotes the Euclidean norm of a vector $x$. 
Let $B(x,R)$ denote the ball of radius $R$ centered at $x$ and $B(0,R)$ is abbreviated as $B(R)$.
For any matrix $M$, $\opnorm{M}$ and $\fnorm{M}$  denote its operator (spectral) norm and Frobenius norm, respectively.

\subsection{Organization}
	\label{sec:org}
	The rest of the paper is organized as follows.	
	\prettyref{sec:onedim} gives the statistical and computational guarantees for EM algorithm in one dimension, showing the achievability of the optimal average risk up to constant factors.	
	\prettyref{sec:concentration} states and proves the relative concentration result \prettyref{eq:relative-dev} for the sample EM map. 
	\prettyref{sec:ddim} presents the analysis of the EM algorithm in $d$ dimensions and give near-optimal statistical and computational guarantees assuming a modest condition on the initialization. In \prettyref{sec:refined} we show that starting from a single random initialization, such a condition is fulfilled in at most $O(\frac{\log n}{\thetanorm^2})$ iterations with high probability.
	\prettyref{sec:mle} proves the convergence of the EM iteration to the MLE.
	Discussions and open problems are presented in \prettyref{sec:discuss}.
	Proofs for Sections \ref{sec:onedim}--\prettyref{sec:mle} are given in \prettyref{sec:pf-onedim}--\ref{sec:pf-mle}, respectively.

	In particular, the main result \prettyref{thm:main-pointwise} previously announced in \prettyref{sec:main} follows from Theorems \prettyref{thm:main-alpha-large} in conjunction with \prettyref{thm:phase1} (on random initialization) and \prettyref{thm:MLE} (on convergence to MLE), 
	while \prettyref{thm:main} follows from combining Theorems \ref{thm:main-pointwise} and \ref{thm:main-alpha-small}.
	
	Complementing the performance guarantee on the EM algorithm, \prettyref{thm:minimax} in \prettyref{app:minimax} determines the minimax rates for the 2-GM model in any dimension, which may be of independent interest.
	Auxiliary results are given in \prettyref{app:lemma}.

\section{EM iteration in one dimension}
\label{sec:onedim}
In this section we present the analysis for one dimension which turns out to be significantly simpler than the $d$-dimensional case; nevertheless, several proof ingredients, both statistical and computational, will re-appear in the analysis for $d$ dimensions later in \prettyref{sec:ddim}.
To bound the relative deviation between the sample and population EM trajectories, we use the concentration inequality for empirical distributions under the Wasserstein distance. Although perhaps not crucial, this method simplifies the analysis and yields the optimal rate of the average risk without unnecessary log factors in one dimension.


\subsection{Concentration via Wasserstein distance}

Recall the 1-Wasserstein distance between probability distributions $\mu$ and $\nu$ \cite{villani.topics}:
\[
W_1(\mu,\nu) = \inf \Expect|X-Y|
\]
where the infimum is over all couplings of $\mu$ and $\nu$, i.e., joint law $\calL(X,Y)$ such that $\calL(X)=\mu$ and $\calL(Y)=\nu$.

To relate the Wasserstein distance to the EM map, we start with the following simple observation:
\begin{lemma}
\label{lmm:tanh}	
For any $x,y\in\reals$,
\[
\sup_{\theta \in \reals} \frac{|x \tanh(x \theta)-y \tanh(y \theta)|}{|\theta|} = |x^2-y^2|.
\]
\end{lemma}
\begin{proof}
	Without loss of generality (WLOG), assume that $x \geq y \geq 0$. Then by symmetry,
	\begin{align}
	\sup_{\theta \in \reals} \frac{|x \tanh(x \theta)-y \tanh(y \theta)|}{|\theta|} 
	= & ~ \sup_{\theta \geq 0} \frac{|x \tanh(x \theta)-y \tanh(y \theta)|}{\theta}	\nonumber \\
	= & ~ 	\sup_{\theta \geq 0} \frac{x \tanh(x \theta)-y \tanh(y \theta)}{\theta}.
	\label{eq:tanh}
	\end{align}
	Straightforward calculation gives 
	\[
	\fracp{}{\theta} \fracp{}{x} \pth{\frac{x \tanh(x \theta)}{\theta}}= \frac{1}{\theta^2 \cosh^2(\theta x)} \pth{\theta x - \frac{1}{2}\sinh(2\theta x) - 2 (\theta x)^2 \tanh(\theta x)} \leq 0,
	\]
	where the inequality follows from $\sinh(t) \geq t$ and $\tanh(t) \geq 0$ for $t\geq 0$.
	Therefore $\theta \mapsto \fracp{}{x} (\frac{x \tanh(x \theta)}{\theta})$ is decreasing on $\reals_+$, 
	which implies that the supremum on the RHS of \prettyref{eq:tanh} is attained at $\theta=0$.	
\end{proof}

By coupling, an immediate corollary to \prettyref{lmm:tanh} is the following:
\begin{lemma}
\label{lmm:Delta-W1}	
For any random variables $X$ and $Y$,
\[
	\sup_{\theta \in \reals} \frac{|\Expect[Y \tanh(\theta Y)] - \Expect[X \tanh(\theta X)]|}{|\theta|}  \leq 
	W_1(\calL(X^2),\calL(Y^2)).
	\]	
\end{lemma}

As mentioned earlier in \prettyref{sec:related}, it is crucial to establish the relative derivation in the sense of \prettyref{eq:relative-dev} for the sample EM trajectory. 
Let $\Delta_n = f_n-f$, where $f_n$ and $f$ are the sample and population EM map defined in \prettyref{eq:fn} and \prettyref{eq:f}. As a consequence of \prettyref{lmm:Delta-W1}, we have, for all $\theta\in\reals$,
	\begin{equation}
	|\Delta_n(\theta)| \leq |\theta| W_1(\nu,\nu_n) 
	\label{eq:Delta-lip}
	\end{equation}
	where $\nu=\calL(Y^2)$ and $\nu_n$ is the empirical distribution of the squared samples $Y_1^2, \ldots, Y_n^2$.
	In other words, $\Delta_n$ is $W_1(\nu,\nu_n)$-Lipschitz. To bound the Lipschitz constant, since $\expect{\exp(Y^2)} \leq C(r)$, 
	applying the concentration inequality in \cite[Theorems 1 and 2]{FG15} (with $d=p=1$, $\alpha=2/3$, $\epsilon=1/3$ and $\gamma=1$), we have 
	\begin{equation}
	\expect{W_1(\nu,\nu_n)} \leq \frac{c_0}{\sqrt{n}}
	\label{eq:W1-concentrate1}
	\end{equation}
	 and 
	\begin{align}
	 \prob{W_1(\nu,\nu_n) \geq x} \leq & ~ c_1[\exp(-c_2 n x^2) \indc{x\leq 1}	\nonumber \\
	& ~ 	  +  \exp(-c_2 (n x)^{1/3}) \indc{x\leq 1} +  \exp(-c_2 (n x)^{2/3})], 	\quad x > 0
	\label{eq:W1-concentrate2}
	\end{align}
	where $c_0,c_1,c_2$ depend only on $r$.
	Therefore, for any $1 \lesssim a \lesssim n^{1/10}$, $\pprob{W_1(\nu,\nu_n) \geq \frac{a}{\sqrt{n}}} \leq \exp(-\Omega(a^2))$.


\subsection{Finite-sample analysis}

The population EM map defined in \prettyref{eq:f} satisfies the following properties:
\begin{lemma}
\label{lmm:fprop}	
	For any $\theta_* \geq 0$,
	\begin{enumerate}
		\item $\theta\mapsto f(\theta)$ is an increasing odd and bounded function on $\reals$, with
		\[
		-(1+\theta_*) \leq - \Expect|Y| = f(-\infty) \leq  f(\theta) \leq f(\infty) = \Expect|Y| \leq 1+\theta_*.
		\]
		\item $\theta\mapsto f(\theta)$ is concave on $\reals_+$ and convex on on $\reals_-$.
		\item $f(0)=0$, $f'(0)=1+{\theta_*}^2$, 
	$f''(0)=0$, and	$f'(\theta_*) \leq \exp(-{\theta_*}^2/2)$.
	\item Define 
	\begin{equation}
	q(\theta)\triangleq \frac{f(\theta)}{\theta}
	\label{eq:q}
	\end{equation}
	Then $q$ is decreasing on $\reals_+$.
	Furthermore, for $\theta \geq 0$,
	\begin{equation}
	q'(\theta) = - \expect{\frac{Y \sinh(2\theta Y)-2\theta Y^2}{2\theta^2 \cosh^2(\theta Y)}} \leq - \frac{2\theta}{3} \expect{\frac{Y^4}{\cosh^2(\theta Y)}}.
	\label{eq:qtheta}
	\end{equation}
	\end{enumerate}
\end{lemma}

The sample-based EM iterates are given by \prettyref{eq:EMsample}, that is,
\[
\theta_{t+1} = f_n(\theta_t).
\]
Here the samples $Y_1,\ldots,Y_n$ are iid drawn from $P_{\theta_*}= \frac{1}{2} N(-\theta_*,1)+\frac{1}{2} N(\theta_*,1)$.
By the global assumption \prettyref{eq:thetabound}, we have $0 \leq \theta_* \leq r$.
WLOG, we assume that $\theta_0>0$ for otherwise we can apply the same analysis to the sequence $\{-\theta_t\}$. 
By \prettyref{eq:Delta-lip}, $\Delta_n=f-f_n$ is $w_n$-Lipschitz, where 
$w_n \triangleq W_1(\nu,\nu_n)$ is a random variable. 
Define the high-probability event
\begin{equation}
E = \{W_1(\nu,\nu_n) \leq c_w\},
\label{eq:event}
\end{equation}
where $c_w$ is a small constant depending only on $r$ that satisfies $c_w < \frac{1}{4}$. By \prettyref{eq:W1-concentrate2}, we have 
$\prob{E} \geq 1 - \exp(-\Omega(n^{1/3}))$. 

Define the following auxiliary iterations:
\begin{align}
\begin{cases}
\otheta_{t+1} = f(\otheta_t) + w_n \otheta_t  \\
\utheta_{t+1} = f(\utheta_t) - w_n \utheta_t
\end{cases}
, \qquad \otheta_0 = \utheta_0=\theta_0.
\label{eq:thetatsandwich}
\end{align}
By \prettyref{lmm:fprop}, $q$ is decreasing and maps $\reals_+$ onto $(0,1+\theta_*^2]$. 
Define
\begin{align}
\otheta \triangleq &~ q^{-1}\pth{1-w_n} \label{eq:otheta}\\
\utheta \triangleq &~ \begin{cases}
  q^{-1}\pth{1+w_n} & |\theta_*| \geq \sqrt{w_n} \\
 0 & |\theta_*| < \sqrt{w_n}
\end{cases}.
\label{eq:utheta}
\end{align}
We will show that on the high-probability event \prettyref{eq:event}, the EM iterates $\{\theta_t\}$ is sandwiched between the two auxiliary iterates $\{\otheta_t\}$ and $\{\utheta_t\}$ (see \prettyref{fig:sandwich1d}).
This is made precisely by the following theorem, 
which gives the estimation error bound and finite-iteration guarantees for the EM algorithm in one dimension:
\begin{figure}[ht]%
\centering
\includegraphics[width=0.35\columnwidth]{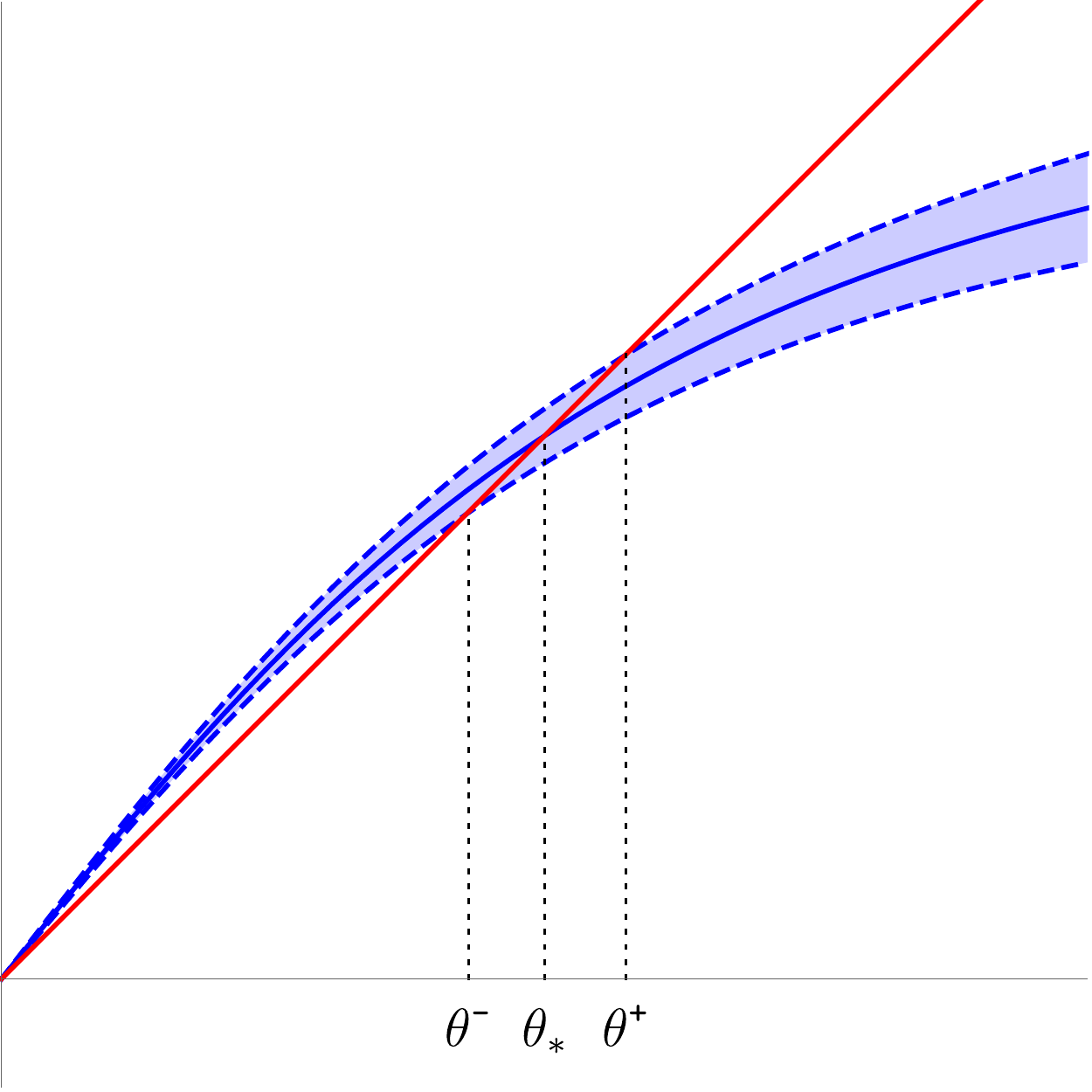}%
\caption{Perturbed EM trajectory and fixed points.}%
\label{fig:sandwich1d}%
\end{figure}

\begin{theorem}[Statistical and computational guarantees for one-dimensional EM]
\label{thm:main-1D}	
	Assume that 
	\begin{equation}
	0 \leq \theta_* \leq r
	\label{eq:thetabound}
	\end{equation}
	for some constant $r$. 
	Assume that 
	\[
	0 < \theta_0 \leq r_0.
	\]
	Then 
			there exist constants $\tau_1,\ldots,\tau_5$ depending on $r$ only, and a constant $n_0=n_0(r,r_0)$, such that for all $n \geq n_0$, on the event \prettyref{eq:event}, the following holds:
	\begin{enumerate}
		\item For all $t \geq 0$,
		\begin{equation}
		0 \leq \utheta_t\leq\theta_t\leq\otheta_t \leq \tau_1.
		\label{eq:sandwich}
		\end{equation}
		\item 
		\begin{equation}
		\ell(\theta_t,\theta_*) \leq  \tau_2 \min\sth{\frac{w_n}{\theta_* }, \sqrt{w_n}},
		\label{eq:sandwich-accuracy}
		\end{equation}
		holds for all $t \geq T=T(\theta_0,\theta_*,w_n)$, where 
		\begin{equation}
		T = \begin{cases}
		\frac{\tau_3}{w_n}	&  \theta_* \leq \tau_4 \sqrt{w_n} \\
		 \frac{\tau_3}{\theta_*^2}\log \frac{1}{\theta_0 w_n} & \theta_* \geq \tau_4 \sqrt{w_n}.
		\end{cases}
		\label{eq:sandwich-T}
		\end{equation}		
	\end{enumerate}
\end{theorem}

A corollary of \prettyref{thm:main-1D} is the following guarantee on the average risk:
\begin{coro}
\label{cor:em1d}	
				There exist constants $c_1,c_2$ depending only on $r$, such that
	\begin{equation}
		\Expect[\ell(\theta_t,\theta_*)] \leq  c_1 \min\sth{\frac{1}{\theta_* \sqrt{n}}, \frac{1}{n^{1/4}}},
	\label{eq:em1d-finitet}
	\end{equation}
		holds for all 
		\begin{equation}
		t\geq 	c_2 \min\sth{\sqrt{n}, \frac{1}{\theta_*^2}} \log \frac{n}{\theta_0}.
		\label{eq:em1d-finitet-T}
		\end{equation}		
	
\end{coro}

\begin{remark}
\label{rmk:rate}	
The rate in \prettyref{eq:em1d-finitet} is optimal in the following sense: the second term $n^{-1/4}$ matches the minimax lower bound in \prettyref{app:minimax}, while the first term corresponds to the local minimax rate since the Fisher information behaves as $\Theta(\theta_*^2)$ for small $\theta_*$.\footnote{Indeed, by Taylor expansion and the dominated convergence theorem, we have
$I(\theta) = \Expect_{\theta}[(\frac{\partial \log p_\theta(Y)}{\partial \theta})^2] = \Expect_{\theta}[(Y\tanh(\theta Y)-\theta)^2] = \theta^2 (\Expect_\theta[(Y^2-1)^2]+o(1)) = (2+o(1))\theta^2$, as $\theta\to0$.}
Indeed, we will show in \prettyref{sec:mle} that the EM iteration converges to the MLE which is asymptotic efficient.

In the special case of $\theta_*=0$, results similar to \prettyref{thm:main-1D} have been shown in \cite[Theorem 3]{dwivedi2018singularity}. 
Furthermore, \cite[Theorem 4]{dwivedi2018singularity} provided a matching lower bound showing that any limiting point of the EM iteration is $\Omega(n^{-1/4})$ with constant probability.

Computationally, suppose we initialize with $\theta_0=1$. Then regardless of the value of $\theta_*$, we have the worst-case computational guarantee:
	with high probability, the EM algorithm achieves the optimal rate \prettyref{eq:em1d-finitet} in at most $O(\sqrt{n} \log n)$ iterations.
	The number of needed iterations can be pre-determined on the basis of $n$ and $\theta_0$, without knowing $\theta_*$.
\end{remark}

\section{Concentration of the EM trajectory: relative error bound}
\label{sec:concentration}
Recall that $\Delta_n = f-f_n$ denotes the difference between the sample and the population EM maps. In one dimension, we have shown that the random function $\Delta_n: \reals \to \reals$ is $O_P(\frac{1}{\sqrt{n}})$-Lipschitz by means of the Wasserstein distance between the empirical distribution and the population. The goal of this section is to extend this result to $d$ dimensions, by showing with high probability $\Delta_n: \reals^d\to\reals^d$  is $O(\sqrt{\frac{d \log n}{n}})$-Lipschitz with respect to the Euclidean distance on a ball of radius $R = \Theta(\sqrt{d})$.\footnote{It is also possible to show that $\Delta_n$ is $O(\sqrt{\frac{d \log^3 n}{n}})$-Lipschitz on the entire space $\reals^d$.}
Since with high probability the EM map $f_n$ takes values within this radius, 
this result allows us to control the fluctuation of the EM trajectory with respect to its population counterpart
proportionally to the distance to the origin. 
This \emph{relative error bound} given next is crucial for obtaining the optimal statistical and computational guarantees.

\begin{theorem}
\label{thm:concentration}	
Assume that $\|\theta_*\| \leq r$ and
\[
n \geq C d \log d
\]
for some universal constant $C$.
There exist universal constants $c_0,C_0$, such that 
with probability at least $1-\exp(-c_0 d \log n)$, 
\begin{enumerate}
	\item 
	For all $\theta \in \reals^d$, $f_n(\theta) \in B(R)$ where $R = 10 (\sqrt{d}+r)$.
	
	\item The function $\Delta_n$ is $L$-Lipschitz on $B(R)$, where $L = C_0 (1+r) \sqrt{\frac{d}{n} \log n}$.
\end{enumerate}
\end{theorem}

The proof is given in \prettyref{sec:pf-concentration}. We note that it is straightforward to 
extend the argument in one dimension (cf.~\prettyref{eq:Delta-lip}--\prettyref{eq:W1-concentrate1}) to bound the Lipschitz constant of $\Delta_n$ by the Wasserstein (in fact, $W_2$) distance between the empirical distribution and the population. Nevertheless, it is well-known that the Wasserstein distance suffers from the curse of dimensionality; for example, the $W_1$ distance behaves as $O_P(n^{-\frac{1}{d}})$ (cf.~e.g~\cite{talagrand1994transportation,FG15}). This effect is due to the high complexity of Lipschitz functions in $d$ dimensions. In contrast, the EM map \prettyref{eq:fn} depends on the $d$-dimensional randomness only through its \emph{linear projection}, which suggests that the it is possible to obtain a rate close to $\sqrt{\frac{d}{n}}$.

\section{Analysis in $d$ dimensions}
\label{sec:ddim}

In this section we analyze for the EM algorithm in high dimensions. 
By using properties of the population EM iteration in \prettyref{sec:ddim-pop} and the relative deviation bound in \prettyref{sec:concentration}, in \prettyref{sec:ddim-sample} we prove optimal statistical and computational guarantees for the sample EM iteration, assuming a modest condition on the initialization which is much weaker than those in \cite{BWY17}. 
Although not necessarily satisfied by random initialization, later in \prettyref{sec:refined} we show that randomly initialized EM iteration will eventually fulfill such a condition with high probability.

\subsection{Properties of the population EM map}
\label{sec:ddim-pop}

Consider the population version of the EM iterates, driven by the population EM map \prettyref{eq:f}:
\[
\btheta_{t+1} = f(\btheta_{t}), \qquad \btheta_0=\theta_0.
\]
We use bold face to delineate it from the finite-sample iteration \prettyref{eq:EMsample}.
Let $\eta_* = \theta_*/\|\theta_*\|$.
Let 
\[
\theta_0 = \alpha_0 \eta_* + \beta_0 \xi_0,
\]
where $\xi_0 \perp \theta_*$ and $\|\xi_0\|=1$, so that $\Span(\theta_0,\theta_*)=\Span(\eta_*,\xi)$.
The next lemma shows that the population EM iterates cannot escape the two-dimensional subspace spanned by $\theta_*$ and $\theta_0$:
\begin{lemma}
\label{lmm:EMsubspace}	
	For each $t\geq 1$, 
	\begin{equation}
	\btheta_t\in \Span(\theta_*,\theta_0).
	\label{eq:subspace}
	\end{equation}
	Furthermore, let 
	\begin{equation*}
	\btheta_t = \balpha_t \eta_* + \bbeta_t \bxi_t
	\end{equation*}
where $\bxi_t \perp \eta_*$ and $\|\bxi_t \|=1$. Then $\{(\balpha_t,\bbeta_t)\}$ satisfies the following recursion
\begin{align}
\balpha_{t+1} = & ~  F(\balpha_t,\bbeta_t) \label{eq:alpha-pop}\\
\bbeta_{t+1}= & ~ 	G(\balpha_t,\bbeta_t)\label{eq:beta-pop}
\end{align}
	where
\begin{align}
F(\alpha,\beta)	\triangleq & ~ \Expect[V \tanh (\alpha V+\beta W)] \label{eq:F}\\
	G(\alpha,\beta) \triangleq & ~ \Expect[W \tanh (\alpha V+\beta W)] \label{eq:G}
\end{align}
with $W\sim N(0,1)$ and $V \sim \frac{1}{2} N(\pm \|\theta_*\|,1)$ being independent.
\end{lemma}
\begin{proof}
It suffices to show \prettyref{eq:subspace}, which was proved in \cite{XHM16}. To give some intuitions, we provide a simple argument below by induction on $t$. Clearly \prettyref{eq:subspace} holds for $t=0$.
Next, fix any $u \in \Span(\theta_*,\theta_0)^\perp$. By the induction hypothesis, $u\perp\btheta_t$.
Therefore
\[
\Iprod{u}{\btheta_{t+1}} = \Expect[\Iprod{u}{Y}  
\tanh(\Iprod{Y}{\btheta_{t}})] = \Expect[\Iprod{u}{Z} 
\tanh(\Iprod{\theta_*}{\btheta_{t}} X + \Iprod{Z}{\btheta_t}  )] = 0
\] 
since $\Iprod{u}{Z}, \Iprod{\btheta_t}{Z}$ and $X$ are mutually independent.
This proves \prettyref{eq:subspace} holds for $t+1$.
\end{proof}

Next, we analyze the convergence of 
$(\balpha_t,\bbeta_t)$.
Without loss of generality (otherwise we can negate $\theta_*$ and $\xi$), we assume that 
\[
\balpha_0 \geq 0,\quad \bbeta_0 \geq 0.
\]
Therefore $\btheta_t \to \theta_*$ is equivalent to $\balpha_t\to \|\theta_*\|$ and $\bbeta_t\to0$.
The convergence is easily justified by the following lemma:
\begin{lemma}[Properties of $F$ and $G$]
\label{lmm:FG}	
For any $\alpha$ and $\beta \geq 0$,
\begin{enumerate}
	\item 
	$\alpha \mapsto F(\alpha,\beta)$ is increasing, odd, concave (resp.~convex) on $\reals_+$ (resp.~$\reals_-$), with 
		$F(0,\beta)=0$, $F(\pm\|\theta_*\|,0)=\pm\|\theta_*\|$.

\item 
$F(\alpha,\beta) \geq 0$ for any $\alpha\geq 0$.
	
	\item $\beta \mapsto G(\alpha,\beta)$ is increasing and concave, with $G(\alpha,0)=0$.
	
	\item $\alpha \mapsto G(\alpha,\beta)$ is even, decreasing on $\reals_+$;
	$\beta \mapsto F(\alpha,\beta)$ is decreasing for $\alpha \geq 0$ and increasing for $\alpha \leq 0$. 
	

	\item (Boundedness) 
		\[
	|F(\alpha,\beta)| \leq \|\theta_*\|+\sqrt{2/\pi}, \quad 0\leq G(\alpha,\beta) \leq \sqrt{2/\pi}.
	\]
	\item
\begin{align}
G(\alpha,\beta) \leq & ~ G(0,\beta) 
= \Expect[W \tanh(\beta W)].
\label{eq:G0}
\end{align}

\item 
\begin{align}
f(\alpha) \geq F(\alpha,\beta) \geq f(\alpha) - (1+\|\theta_*\|^2) \alpha \beta^2, & \qquad \alpha \geq 0 \label{eq:F01}\\
f(\alpha) \leq F(\alpha,\beta) \leq f(\alpha) - (1+\|\theta_*\|^2) \alpha \beta^2, & \qquad \alpha \leq 0 \label{eq:F02}
\end{align}
where
\begin{equation}
f(\alpha)\triangleq F(\alpha,0) = \Expect[V \tanh (\alpha V)]
\label{eq:F0}
\end{equation}
coincides with the one-dimensional EM map defined in \prettyref{eq:f} with $\theta_*$ replaced by $\thetanorm$.

\item 
	\begin{equation}
	G(\alpha,\beta) \leq \beta\pth{ 1 - \frac{\alpha^2+\beta^2}{2+4(\alpha^2+\beta^2)}}.
	\label{eq:Gnew}
	\end{equation}
\end{enumerate}
\end{lemma}

From \prettyref{lmm:FG} it is clear that in the population case, the only fixed points are the desired $(\pm\thetanorm,0)$ and undesired $(0,0)$. 
As long as the initial value is not orthogonal to the ground truth (i.e., $\balpha_0\neq 0$), $\btheta_t$ converges to $\pm \theta_*$; this has been previously shown in \cite{XHM16,DTZ17}.
In fact, the orthogonal component $\bbeta_t$ converges to $0$ monotonically regardless of the signal component $\balpha_t$. 
Furthermore, if we start out with $\balpha_0>0$, then $\balpha_t>0$ remains true for all $t$, and when $\bbeta_t$ gets sufficiently close to 0, $\balpha_t$ converges to $\thetanorm$ following the one-dimensional EM dynamics (cf.~\prettyref{eq:F0}). 
However, a major distinction between the one-dimensional and $d$-dimensional case is that $\balpha_t$ need \emph{not} converge monotonically even in the infinite-sample setting.
In fact, if the initial value has little overlap with the ground truth (as is the case for random initialization in high dimensions), $\bbeta_t$ is large initially which causes $\balpha_t$ to decrease and $\btheta_t$ to move closer to the undesired fixed point at zero (see \prettyref{fig:non-monotone}). 
Therefore, in the finite-sample setting, we need to assume conditions on the initialization (namely lower bound on $|\alpha_t|$) in order to avoid being trapped near zero -- we will return to this point in the finite-sample analysis in the next subsection.
 This is in stark contrast to the one-dimensional case: even with finite samples, for any non-zero initialization, the EM iteration eventually converges to a neighborhood of the ground truth with optimal accuracy (cf.~\prettyref{thm:main-1D}).

\begin{figure}[ht]
	\centering
	\subfigure[Non-monotone convergence of $\balpha_t$ ($\balpha_0=0.1$, $\bbeta_0=0.7$).]%
	{\label{fig:non-monotone} \includegraphics[width=0.45\columnwidth]{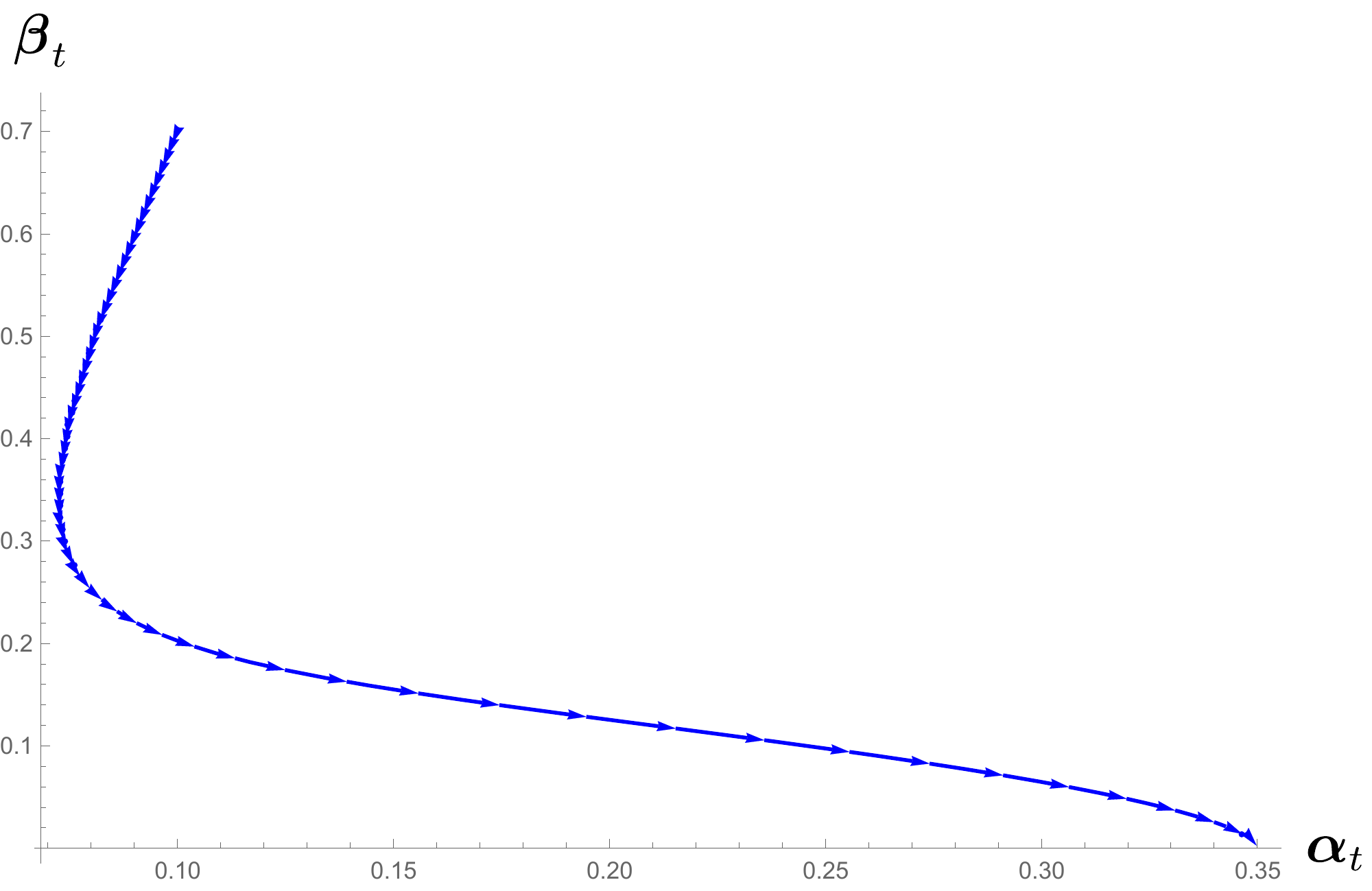}}
	\subfigure[Monotone convergence of $\balpha_t$ ($\balpha_0=\bbeta_0=0.1$).]%
	{\label{fig:monotone} \includegraphics[width=0.45\columnwidth]{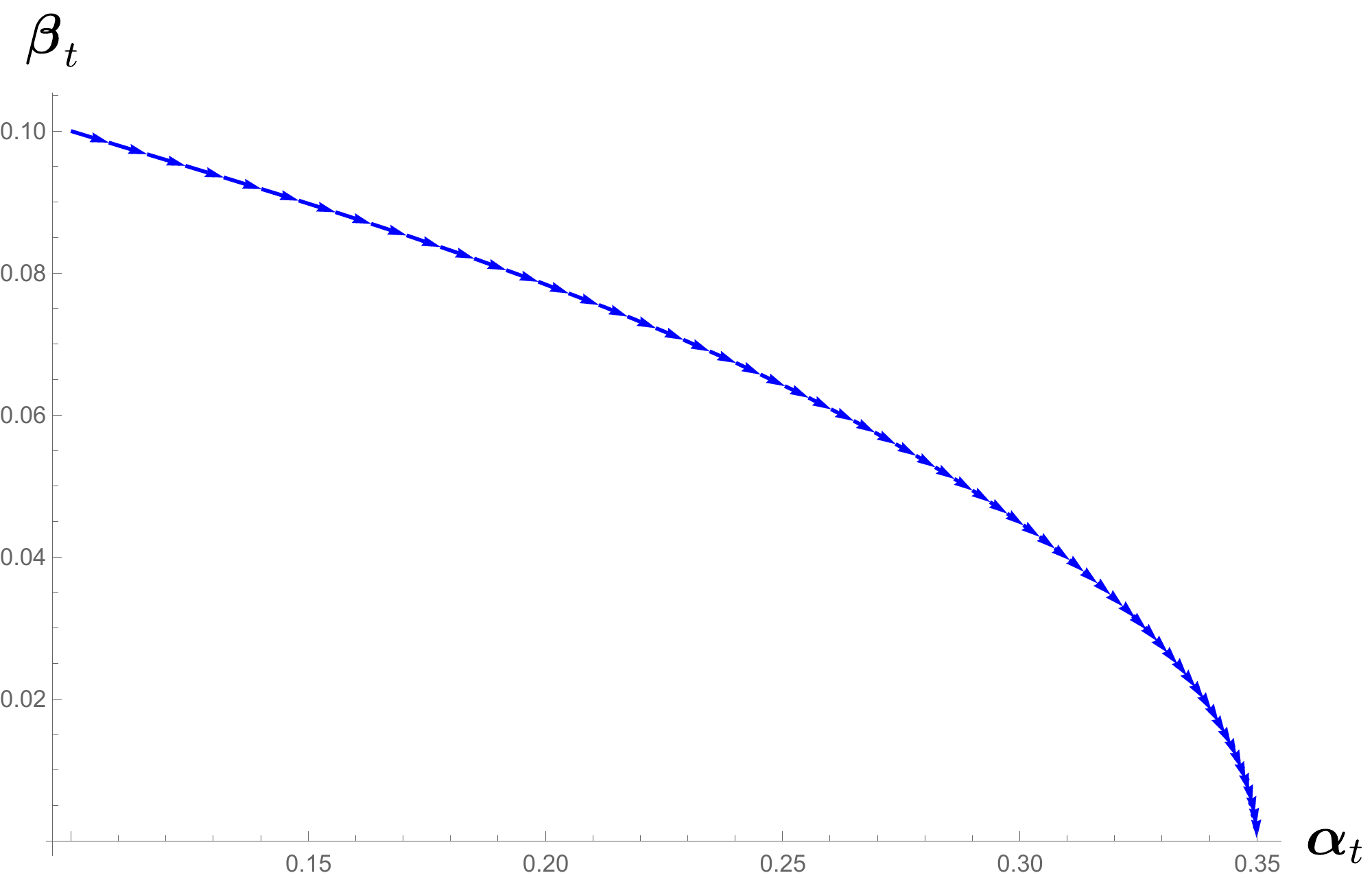}}
	\caption{Convergence of $(\balpha_t,\bbeta_t)$ in the population dynamics in $d$ dimensions with $\thetanorm=0.35$ for $60$ iterations.}
	\label{fig:EMpop}
\end{figure}





\subsection{Finite-sample analysis}
	\label{sec:ddim-sample}
We now analyze the $n$-sample EM iteration \prettyref{eq:EMsample}, that is,
\[
\theta_{t+1} = f_n(\theta_{t}).
\]
Write
\[
\theta_t = \alpha_t \eta_* + \beta_t \xi_t,
\]
where $\xi_t \perp \eta_* = \frac{\theta_*}{\thetanorm}$, $\|\xi_t\|=1$ and $\beta_t \geq 0$. Thus
$\|\theta_t\| = \sqrt{\alpha_t^2+\beta_t^2}$.

Recall that $\Delta_n=f_n-f$ denotes the difference between the sample and population EM maps. 
In view of \prettyref{thm:concentration}, with probability at least $1- \exp(-c_0 d \log n)$, the following event holds:
\begin{equation}
\begin{aligned}
\sup_{\theta\in\reals^d} \|f_n(\theta)\| & ~ \leq R	\\
\|\Delta_n(\theta)\| & ~ \leq \omega \|\theta\|, \quad \forall \theta \in B(R),	\\
\end{aligned}
\label{eq:eventd}
\end{equation}
where $R=10(r+\sqrt{d})$ and 
\begin{equation}
\omega = \sqrt{C_\omega \frac{d}{n} \log n}
\label{eq:omega}
\end{equation}
and $C_\omega $ is a constant that only depends on $\thetaub$.
We assume that $n$ is sufficiently large so that $\omega$ is at most an absolute constant.

Recall from \prettyref{lmm:EMsubspace} that $f(\theta) \in \Span(\eta_*, \theta)$ for any $\theta\in\reals^d$.
	Furthermore,
	\[
	f(\theta_t) = F(\alpha_t,\beta_t) \eta_* + G(\alpha_t,\beta_t) \xi_t,
	\]
	where $F$ and $G$ are defined in \prettyref{eq:F}--\prettyref{eq:G}.
Therefore
\[
\alpha_{t+1} = \iprod{\theta_{t+1}}{\eta_*} = F(\alpha_t,\beta_t) \eta_* + \Iprod{\Delta_n(\theta_t)}{\eta_*}.
\]
In view of \prettyref{eq:eventd}, 
 we have
\[
|\Iprod{\Delta_n(\theta_t)}{\eta_*}| \leq \|\Delta_n(\theta_t)\| \leq \omega (|\alpha_t|+\beta_t).
\]
Hence
\begin{align}
\alpha_{t+1} \leq & ~ F(\alpha_t,\beta_t)  + \omega (|\alpha_t|+\beta_t) \label{eq:alpha+}\\
\alpha_{t+1} \geq & ~ F(\alpha_t,\beta_t)  - \omega (|\alpha_t|+\beta_t) \label{eq:alpha-}
\end{align}
On the other hand, we have
\[
(I-\eta_*\eta_*^\top) \theta_{t+1} = G(\alpha_t,\beta_t) \xi_t+ (I-\eta_*\eta_*^\top) \Delta_n(\theta_t).
\]
Taking norms on both sides, we have
	\begin{align}
	\beta_{t+1}
	\leq & ~ G(\alpha_t,\beta_t) + \omega(|\alpha_t| + \beta_t). \label{eq:beta-iterate}
	\end{align}
		The equations \prettyref{eq:alpha+}--\prettyref{eq:alpha-} and \prettyref{eq:beta-iterate} should be viewed as the finite-sample perturbation of the population dynamics \prettyref{eq:alpha-pop} and \prettyref{eq:beta-pop}, respectively.

	We will show that the orthogonal component $\beta_t$ unconditionally converges to $O(\sqrt{\omega}) = O((\frac{d\log n}{n})^{\frac{1}{4}})$; however, for finite sample size we cannot expect $\beta_t$ to converge to zero. 
	To analyze $\alpha_t$, let us assume that $\beta_t$ have converged to this limiting value (in fact, by initializing near zero, we can ensure $\beta_t = O(\sqrt{\omega})$ for all $t$.)
Following the sandwich analysis in one dimension, we can define the auxiliary iterations similar to \prettyref{eq:thetatsandwich}
\begin{align}
\begin{cases}
\oalpha_{t+1} = F(\oalpha_t,\beta_t)  + \omega \oalpha_t + \omega^{3/2}   \\
\ualpha_{t+1} = F(\ualpha_t,\beta_t)  - \omega \ualpha_t - \omega^{3/2}
\end{cases}
, \qquad \oalpha_0 = \ualpha_0=\alpha_0
\label{eq:alphasandwich}
\end{align}
and show that the upper bound sequence $\{\oalpha_t\}$ converges to $\oalpha$ which is within the optimal rate of the desired $\thetanorm$. 
However, due to the additional intercept, 
	the lower bound sequence $\{\ualpha_t\}$ have two possible fixed points (see \prettyref{fig:perturb-em-ddim}): the ``good'' fixed point
	$\ualpha$ that is within the optimal rate of $\thetanorm$, and the ``bad'' fixed point	$\alpha^{\circ}$ that is close to zero (in fact, $\alpha^{\circ} = O(\sqrt{\omega})$).
\begin{figure}[ht]%
\centering
\includegraphics[width=0.4\columnwidth]{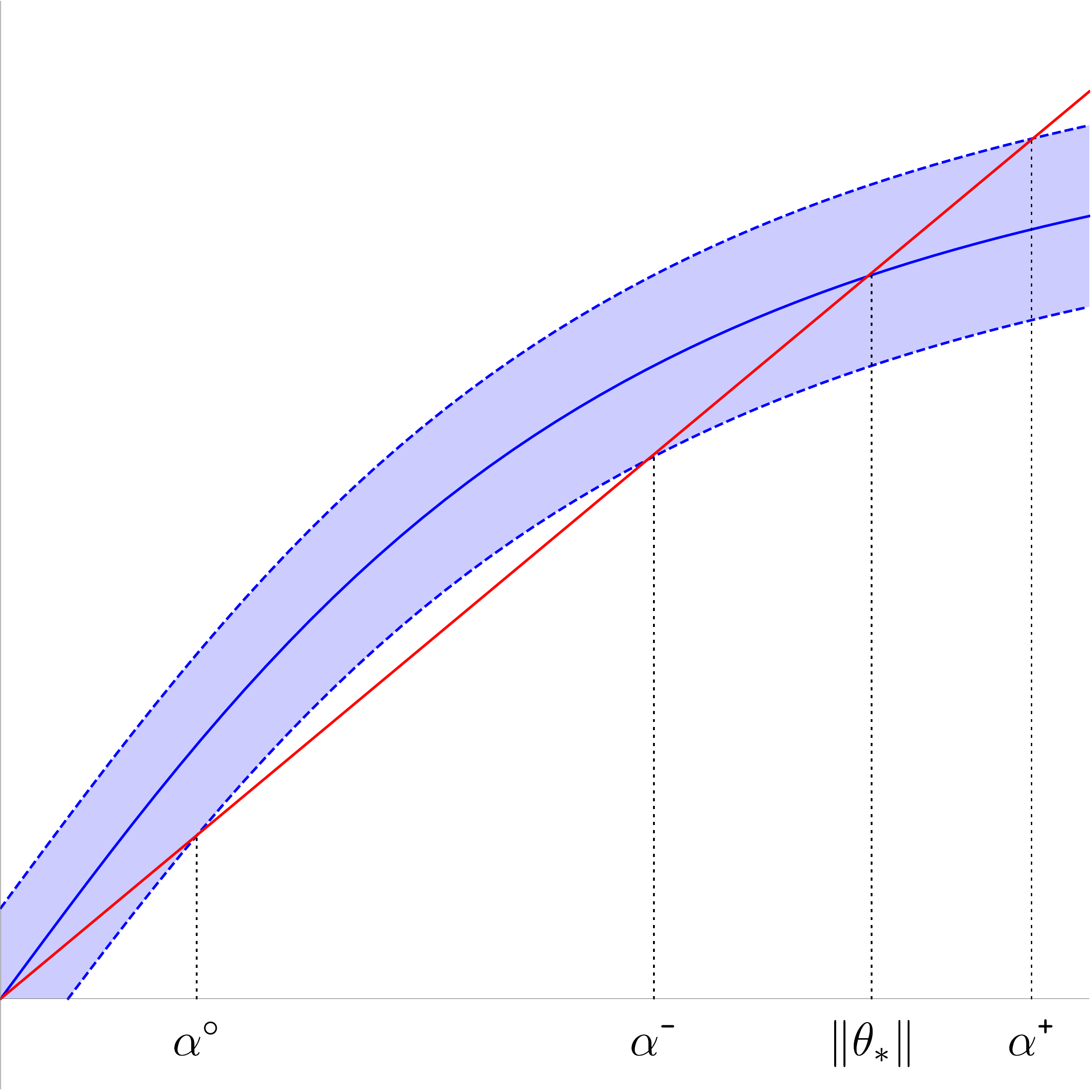}%
\caption{Perturbed EM trajectory for $\alpha_t$ and fixed points.}%
\label{fig:perturb-em-ddim}%
\end{figure}
	Consequently, if the iteration starts from from the left of the bad fixed points, i.e., $\alpha_0 < \alpha^{\circ}$, which is what happens when the initialization is nearly orthogonal to $\theta_*$, the lower bound sequence $\ualpha_t$ may be stuck at near zero and fail to converge to the desired neighborhood of $\thetanorm$.
	Thus to rule this out it requires more refined argument than the above sandwich analysis, which is carried out in the next section. 
	For this section we focus on proving the performance guarantee assuming a mild assumption on the initialization.	
	Specifically, we establish the following claims:
	\begin{enumerate}
		\item Orthogonal direction: we show that regardless of the initialization, 
			$\{\beta_t\}$ unconditionally converges to the near-optimal rate $O(\sqrt{\omega})$. 
			In particular, if we start from near zero (and we will), we can ensure that the entire sequence $\{\beta_t\}$ is $O(\sqrt{\omega})$ for \emph{all} $t$.  			
		\item Signal direction: we show that 
		\begin{itemize}
			\item For small $\theta_*$, i.e., $\thetanorm =O(\sqrt{\omega})$, 
			$\{|\alpha_t|\}$  unconditionally converges to $O(\sqrt{\omega})$, and hence so does $\|\theta_t-\theta_*\|$.
			
			\item 
			For large $\theta_*$, i.e., $\thetanorm = \Omega(\sqrt{\omega})$, 
		provided that the initialization satisfies 
		\[
		|\Iprod{\eta_0}{\eta_*}| \gtrsim \frac{1}{\thetanorm^2} \sqrt{\frac{d}{n} \log n},
		\]
		the signal part $\{\alpha_t\}$ converges to $\thetanorm + O(\frac{1}{\thetanorm}\sqrt{\frac{d}{n} \log n})$.	
			The condition on the initialization improves that of \cite{BWY17}, which requires that 
			$|\Iprod{\eta_0}{\eta}| \geq \Omega(1)$ and $\thetanorm = \Omega(1)$.			
		Note that if $\eta_0$ is drawn uniformly from the unit sphere, we have $|\iprod{\eta_0}{\eta_*}| = \Theta_P(\frac{1}{\sqrt{d}})$. Thus, in the special case of $\thetanorm$ being a constant, the above condition is fulfilled when $n = \tilde\Omega(d^2)$.		
				Nevertheless, in \prettyref{sec:refined} we will prove the refined result that as long as $n = \tilde\Omega(d)$, 
				starting from a single random initialization, the EM iterates will eventually satisfy the above condition with high probability.
				
		\end{itemize}
		
	\end{enumerate}


	In the rest of the paper, we always assume that the initialization lies in a bounded ball. To simplify the presentation, assume that
\begin{equation}
\|\theta_0\| \leq 1.
\label{eq:theta0-bdd}
\end{equation}
The following theorems are the main result of this section. We note that results similar to Theorems \ref{thm:main-beta}--\ref{thm:main-alpha-small}  have been shown in \cite[Theorem 3]{dwivedi2018singularity} in the special case of $\theta_*=0$.
	
\begin{theorem}[Unconditional convergence of $\beta_t$]
\label{thm:main-beta}	
There exist constants $\kappa_0,\kappa_1,\kappa_2$ depending only on $\thetaub$, such that
on the event \prettyref{eq:eventd}, the following holds.
\begin{enumerate}
	\item For all $t\geq 0$,
\begin{equation}
\beta_{t+1} \leq  \beta_t (1+\omega) + \omega |\alpha_t| \\
\label{eq:beta-uncond2}
\end{equation}
and
\begin{equation}
\beta_{t+1} \leq  \beta_t (1+\omega) - \frac{\beta_t^3}{2+8\Gamma^2} + \min\sth{\frac{\omega^2 (2+8\Gamma^2) }{2\beta_t}, \omega \Gamma} \\
\label{eq:beta-uncond}
\end{equation}
where $\Gamma=2+2\thetaub$.
\item Consequently, regardless of $\theta_0$, 
	\begin{equation}
\limsup_{t\to\infty} \beta_t 	\leq \kappa_1 \pth{\frac{d}{n} \log n}^{\frac{1}{4}}.	
	\label{eq:beta-rate}
	\end{equation}	
\item 	Furthermore, if
$\omega \leq \kappa_0$ and 
\begin{equation}
\|\theta_0\| \leq \kappa_2 \pth{\frac{d}{n} \log n}^{\frac{1}{4}},
\label{eq:theta0nearzero}
\end{equation}
then for all $t \geq 0$,
\begin{equation}
\beta_t \leq \kappa_2 \pth{\frac{d}{n} \log n}^{\frac{1}{4}}.
\label{eq:betanearzero}
\end{equation}		
\end{enumerate}
\end{theorem}

\begin{theorem}[Small $\thetanorm$: Unconditional convergence of $\alpha_t$]
	\label{thm:main-alpha-small}
	
	There exist absolute constants $K,L\geq 1$, such that
on the event \prettyref{eq:eventd}, the following holds.		
	Let $s_0$ be such that $K \sqrt{\omega} \leq s_0 \leq 1$. 
			Assume that $\thetanorm \leq s_0$. 
			\begin{enumerate}
				\item Regardless of $\theta_0$, 
			\begin{equation}
			\limsup_{t\to\infty} |\alpha_t| \leq 2 s_0.
			\label{eq:smalltheta}
			\end{equation}
				and hence
	\begin{equation}
	\limsup_{t\to\infty} \ell(\theta_t,\theta_*) \leq 3 s_0 +\kappa_1 \pth{\frac{d}{n} \log n}^{\frac{1}{4}}.
	\label{eq:theta-rate-smalltheta}
	\end{equation}
	
	\item Furthermore, if the initializer $\theta_0$ satisfies 
\prettyref{eq:theta0nearzero}, then 
\begin{equation}
			|\alpha_t| \leq L s_0
			\label{eq:smallthetaallt}
			\end{equation}
				and 
	\begin{equation}
	\ell(\theta_t,\theta_*) \leq 2 L s_0
	\label{eq:theta-rate-smallthetaallt}
	\end{equation}
	 hold for all $t \geq 0$.
	\end{enumerate}
			
			\end{theorem}

			\begin{theorem}[Large $\thetanorm$: Conditional convergence of $\alpha_t$]
			\label{thm:main-alpha-large}
		There exist constants $\lambda_0,\ldots,\lambda_4$ depending only on $\thetaub$, such that
		on the event \prettyref{eq:eventd}, the following holds.			
			Assume that $\thetanorm \geq \lambda_0 \sqrt{\omega}$. Let $\eta_0 \in S^{d-1}$ satisfies
	\begin{equation}
	|\iprod{\eta_0}{\eta_*}| \geq \frac{\lambda_2}{\thetanorm^2} \sqrt{\frac{d}{n} \log n}.	
	\label{eq:goodini}
	\end{equation}
	Set
	\begin{equation}
	\theta_0 = c \pth{\frac{d}{n} \log n}^{1/4}	\eta_0 
	\label{eq:ini0}
	\end{equation}
	where $c\leq \kappa_2$ and $\kappa_2$ is from \prettyref{thm:main-beta}. 
	Then 
	\begin{equation}
 \limsup_{t\to\infty} \big|\alpha_t -\thetanorm \big|	\leq \lambda_1 
\frac{1}{\thetanorm} \sqrt{\frac{d \log n}{n}}
	\label{eq:alpha-rate}
	\end{equation}
	and
	\begin{equation}
	\limsup_{t\to\infty} \ell(\theta_t,\theta_*) \leq \lambda_3 
	\frac{1}{\thetanorm} \sqrt{\frac{d \log n}{n}}.
	\label{eq:theta-rate}
	\end{equation}
	Furthermore, \prettyref{eq:alpha-rate} and \prettyref{eq:theta-rate} hold to all $t \geq \lambda_4 \frac{\log n}{\thetanorm^2}$.
	\end{theorem}
	
	\begin{remark}
	We can take $s_0=\lambda_0 \sqrt{\omega}$ in \prettyref{thm:main-alpha-small}, so that Theorems \ref{thm:main-alpha-small} and \ref{thm:main-alpha-large} gives the near-optimal rate of $O(\pth{\frac{d}{n} \log n}^{1/4})$ for the case of small and large $\thetanorm$ respectively. 
	Later in the refined analysis in \prettyref{sec:refined} we will take $s_0$ slightly larger than $\sqrt{\omega}$; cf.~\prettyref{eq:thetastar-lb}.
	\end{remark}
	
	Theorems \ref{thm:main-beta}--\ref{thm:main-alpha-large} are proved in \prettyref{sec:pf-ddim-thm}. Here we give a sketch of the proof of  \prettyref{thm:main-alpha-large}. The analysis consists of three phases:
		%
	
	\begin{description}
	\item[Phase I: $\alpha_t \lesssim \sqrt{\omega}$.]
		By using the condition \prettyref{eq:goodini} on the initialization, we show that in this phase $\alpha_t$ increases geometrically according to
		\[
		\alpha_{t+1} \geq (1+\Omega(\thetanorm^2)) \alpha_t.
		\]
	
	\item[Phase II: $\alpha_t \gtrsim \sqrt{\omega}$.]
	Now that $\alpha_t$ has escaped the undesired fixed point near zero (cf.~\prettyref{fig:perturb-em-ddim}), one can apply the ``sandwich bound''
	\prettyref{eq:alphasandwich} to show that $\alpha_t$ follows a perturbed one-dimensional EM evolution
	\[
	\alpha_{t+1} = f(\alpha_t) + O(\omega \alpha_t),
	\]
	where $f$ is defined \prettyref{eq:F0} and coincides with the one-dimensional EM map \prettyref{eq:f} with $\theta_*$ replaced by $\thetanorm$.
		
		\item[Phase III: $\alpha_t \asymp \thetanorm$.]
		Recall that \prettyref{thm:main-beta} ensures that $\beta_t$ converges to the worst-case rate $O(\sqrt{\omega})$. 
		Now that $\alpha_t$ has reached a constant fraction of the desired limit $\thetanorm$, 
		we can obtain improved estimate $\beta_t \lesssim \frac{\omega}{\thetanorm}$, leading to the optimal $\thetanorm$-dependent bound \prettyref{eq:theta-rate}.

	\end{description}

\section{Refined analysis for random initialization: the initial phase}
	\label{sec:refined}
In this section we analyze the EM iterates starting from a single random initialization. 
Since Theorems \ref{thm:main-beta} and \ref{thm:main-alpha-small} have covered the case of small $\thetanorm$, we only consider the case where 
$\thetanorm \gg (\frac{d}{n})^{1/4}$.	 
We provide a refined analysis of Phase I in the proof of \prettyref{thm:main-alpha-large}:
if the initial direction is uniformly chosen at random, then with high probability, the iterates will satisfy $\alpha_t = \Omega(\sqrt{\omega})$ for sufficiently large constant $C$ 
in at most $O(\frac{1}{\thetanorm^2} \log n)$ iterations and hence the analysis in the subsequent Phase II and III applies. 
This was previously shown in \prettyref{thm:main-alpha-large}
under the stronger assumption \prettyref{eq:goodini} 
which 
need not be fulfilled by random initializations.

Recall that 
$\eta_* = \frac{1}{\|\theta_*\|} \theta_*$ denotes the true direction
and
\[
\alpha_t = \iprod{\theta_t}{\eta_*}, \quad \beta_t = \|(I-\eta_*\eta_*^\top)\theta_t\|.
\]
%
WLOG, we assume the following:
\begin{enumerate}
	\item 
	Thanks to the rotational invariance of the Gaussian distribution, we can assume that 
	the true center is aligned with a coordinate vector, i.e., $\theta_* = \thetanorm e_1$, so that 
\[
\alpha_t = \theta_{t,1}, \quad \beta_t = \|\theta_{t,\perp}\| = \|(\theta_{t,2}, \ldots, \theta_{t,d})\|.
\]

\item The initialization satisfies $\alpha_0 > 0$. Otherwise, we can apply the same analysis to $\{-\theta_t\}$ which has the same law as $\{\theta_t\}$.
\end{enumerate}
Furthermore, we assume that the ground truth satisfies\footnote{Currently, this comes from the condition \prettyref{eq:Tstaromega}. 
The $\log d+\log\log n$ comes from the condition that 
$(1+\|\theta_*\|^2)^{\Tstar} \geq \sqrt{d \log n}$, since the random initializer satisfies $|\Iprod{\eta_0}{\eta_*}| \geq \frac{1}{\sqrt{d \log n}}$.
}
\begin{equation}
\thetaub \geq \|\theta_*\| \geq \pth{\frac{C_\star d \log^3 n }{n}  }^{1/4}
\label{eq:thetastar-lb}
\end{equation}
for some absolute constant $C_\star$. 
Otherwise, applying \prettyref{thm:main-alpha-small} (with $s_0$ being the RHS of \prettyref{eq:thetastar-lb}) 
shows that regardless of the initialization, 
we achieve the near optimal rate for all $t\geq 0$:
\begin{equation}
\|\theta_t -\theta_*\| = O\pth{\pth{\frac{d}{n} \log^3 n }^{1/4}}.
\label{eq:smalltheta-rate1}
\end{equation}

Define
\begin{equation}
T_1 \triangleq \min\sth{t\in \naturals: \alpha_t > C_* \sqrt{\omega}},
\label{eq:T1}
\end{equation}
where $C_*$ is some constant depending only on $\thetaub$; cf.~\prettyref{eq:CCC}.
The main result of this section is the following:
\begin{theorem}
\label{thm:phase1}
Assume that $\theta_*$ satisfies \prettyref{eq:thetastar-lb}.
There exists constants $C_0,C_1,C_2$ depending only on $\thetaub$,
such that the following holds:
Let 
\begin{equation}
\theta_0 = C_0 \pth{\frac{d}{n} \log n}^{1/4} \eta_0,
\label{eq:random-init}
\end{equation}
where $\eta_0$ is drawn uniformly at random from the unit sphere $S^{d-1}$. 
Assume that 
\begin{equation}
n \geq C_1 d \log^3 d.
\label{eq:ndlogd}
\end{equation}
Then with probability at least $1 - \frac{C_2 \log\log n}{\sqrt{\log n}}$,
	\begin{equation}
	T_1 \leq \Tstar \triangleq \frac{C_T (\log d + \log \log n)}{\thetanorm^2}
	\label{eq:phase1}
	\end{equation}
	where $C_T$ is some universal constant.
\end{theorem}

\prettyref{thm:phase1} shows that after $t \geq T_1$, the iteration enters Phase II and the statistical guarantee in \prettyref{thm:main-alpha-large} applies to all subsequent iterations; in particular, the optimal estimation error is achieved in another $O(\frac{\log n}{\thetanorm^2} ) = O(\sqrt{\frac{n}{d \log n})}$ iterations, proving \prettyref{thm:main-pointwise} previously announced in \prettyref{sec:main}.
Finally, since the case of $\thetanorm = O((\frac{d \log^3 n }{n})^{1/4})$ is covered by \prettyref{eq:smalltheta-rate1}, the worst-case result in \prettyref{thm:main} follows.

\subsection{Proof of \prettyref{thm:phase1}}
	\label{sec:pf-phase1}

	In this subsection we provide the main argument for proving \prettyref{thm:phase1}, with key lemmas proved in \prettyref{sec:pf-refined-mainlemmas}.
Suppose, for the sake of contradiction, that
$\alpha_t \leq \sqrt{\omega}$ for all $t \leq \Tstar$. Then in view of \prettyref{eq:betanearzero}, we conclude that for all $t \leq \Tstar$,
\begin{equation}
\|\theta_t\| \leq 2 C_1 \pth{\frac{d}{n} \log n}^{1/4}
\label{eq:thetatnorm}
\end{equation}
for some constant $C_1$.
In particular, $\theta_t$ belongs to the unit ball in view of the assumption \prettyref{eq:ndlogd}.



We now introduce an \emph{auxiliary sequence} of iterates $\{\ttheta_t\}$, which is main apparatus for analyzing the initial growth of the signal. 
Since the law of $Y_{i,1}$ is symmetric, with loss of generality, we view the $i$th sample as $Y_i=(b_i Y_{i,1}, Y_{i,2},\ldots,Y_{i,d})$, where $b_i$'s are independent Rademacher variables, and the sample-based EM iterates is
\[
\theta_{t+1} = f_n(\theta_t),
\]
where
\[
f_n(\theta)  =\Expect_n[Y \tanh \iprod{\theta}{Y}] = \frac{1}{n}\sum_{i=1}^n Y_i \tanh \iprod{\theta}{Y_i}.
\]
In comparison, the auxiliary iteration is based on the modified samples $(\tY_1,\ldots,\tY_n)$, where 
$\tY_i=(\tb_i Y_{i,1}, Y_{i,2},\ldots,Y_{i,d})$, $\tb_i$'s are independent Rademacher variables, and $\{\tb_i,b_i,Y_i\}$ are mutually independent.
Define the auxiliary iterates
\begin{equation}
\ttheta_{t+1} = \tilde f_n(\ttheta_t),
\label{eq:ttheta}
\end{equation}
where
\begin{align}
\tilde f_n(\theta) \triangleq & ~ 	\Expect_n[\tY \tanh \Iprod{\theta}{\tY}] = \frac{1}{n}\sum_{i=1}^n \tilde Y_i \tanh \Iprod{\theta}{\tilde Y_i}.
\label{eq:tfn}
\end{align}
Both the main and the auxiliary sequence starts from the same random initialization:
\[
\ttheta_0=\theta_0,
\]
as specified by \prettyref{eq:random-init}.
The angle of a random initialization satisfies the following:

\begin{lemma}[Random initialization]
\label{lmm:ini}	
There exist an absolute constant $C_0$, such that for any $a>0$, $\pprob{|\Iprod{\eta_0}{e_1}| \geq \frac{a}{\sqrt{d}}} \leq C_0 a \sqrt{\log \frac{1}{a}}$.
\end{lemma}
\begin{proof}
Note that $\iprod{\eta_0}{e_1}$ is equal in distribution to $Z_1/\|Z\|$, where $Z=(Z_1,\ldots,Z_d)$ is standard normal.
Therefore
$\pprob{|\Iprod{\eta_0}{e_1}| < \frac{a}{\sqrt{d}}} \leq \pprob{\|Z\| \geq \sqrt{C d}} + 
\pprob{|Z_1| < \sqrt{C} a} $. Take $C = 2 + 3 \log \frac{1}{a}$. By \prettyref{lmm:LM}, $\pprob{\|Z\| \geq \sqrt{C d}} \leq a^d \leq a$, 
and $\pprob{|Z_1| < \sqrt{C} a} \leq \sqrt{2C/\pi} a$.
\end{proof}

In the following, 
we conduct the analysis on the event:
\begin{equation}
\alpha_ 0 \geq \frac{1}{\sqrt{d \log n}} \|\theta_0\|,
\label{eq:ini}
\end{equation}
which holds  with probability at least $1 - O(\frac{\log\log n}{\sqrt{\log n}})$, in view of \prettyref{lmm:ini}.

The key argument is to show that the signal component $\alpha_t$ grows exponentially according to
\begin{align}
\alpha_{t+1} \geq & ~ \alpha_t (1 + \|\theta_*\|^2 - o(\|\theta_*\|^2)).
\label{eq:alphat1}
\end{align}	
More precisely, 
we prove a quantitative version of \prettyref{eq:alphat1} (cf.~\prettyref{eq:alphagrow0} below).

\begin{lemma}
\label{lmm:induction}	
With probability at least $1-O(n^{-1/2} \log n)$, for all $t =0,1,\ldots,\Tstar$, 
\begin{align}
\|\theta_t - \ttheta_t\| \leq & ~ 
\alpha_t \sqrt{\frac{K d \log^3 n}{n}} t
	\label{eq:ind1}\\
\frac{\beta_{t}}{\alpha_{t}} \leq  & ~ \sqrt{d \log n} + \omega t
\label{eq:ind4-pre}
\end{align}
and
\begin{align}
\alpha_t \geq  & ~  \frac{1}{\sqrt{K d \log n}} \|\theta_t\| 	\label{eq:ind4} \\
\alpha_{t+1} \geq & ~  \alpha_t \pth{1+\|\theta_*\|^2 - \sqrt{\frac{K d \log^3 n}{n}} }, \label{eq:alphagrow0} 
\end{align}
where $K$ is a constant depending only on $\thetaub$.
\end{lemma}

The proof of \prettyref{lmm:induction} is by induction on $t$, replying on the following results that relate the actual iterations to the auxiliary ones.

\begin{lemma}
\label{lmm:alphat1-raw}	
For each $t\geq 0$, with probability at least $1-O(n^{-1})$, we have
\begin{equation}
\alpha_{t+1} \geq \alpha_t \pth{1+\|\theta_*\|^2  - \sqrt{\frac{C \log n}{n}} - C \|\theta_t\|^2} 
- \sqrt{\frac{C \log^2 n}{n}} \|\theta_t\| - \sqrt{\frac{C d \log^2 n}{n}} \|\theta_t-\ttheta_t\|,
\label{eq:alphat1-raw}
\end{equation}
where $C$ is some constant depending only on $\thetaub$.	
\end{lemma}

\begin{lemma}
\label{lmm:ind1-raw}	
For each $t\geq 0$, with probability at least $1-O(n^{-1})$, we have
\begin{equation}
\|\ttheta_{t+1}-\theta_{t+1}\|
\leq \pth{1 + \|\theta_*\|^2+ \sqrt{\frac{C d\log^2 n}{n}}} \|\ttheta_{t}-\theta_{t}\| + \sqrt{\frac{C d\log^2 n}{n}} \alpha_t + \sqrt{\frac{C \log n}{n}} \|\theta_t\|,
\label{eq:ind1-raw}
\end{equation}
where $C$ is some constant depending only on $\thetaub$.	
\end{lemma}

Now we complete the proof of \prettyref{thm:phase1} by contradiction. 
Since \prettyref{eq:alphagrow0} holds for all $t \leq \Tstar$, in view of the assumption \prettyref{eq:thetastar-lb}, we have
\[
\alpha_{t+1} \geq 	\alpha_t \pth{1+c_0 \|\theta_*\|^2}.
\]
 Since $\alpha_0 \geq \|\theta_0\| \frac{1}{\sqrt{d \log n}} \geq \frac{C_0}{\sqrt{C_\omega}} \frac{\sqrt{\omega} }{\sqrt{d \log n}}$, when $t \geq \Tstar = \frac{C_T (\log d + \log \log n)}{\thetanorm^2}$ for sufficiently large constant $C_T$,
we have 
$\alpha_t > \sqrt{\omega} = (C_\omega \frac{d}{n} \log n)^{1/4}$, which is the needed contradiction.

\section{Approaching the MLE}
	\label{sec:mle}
	
	Despite being a heuristic of solving the maximum likelihood, in this section we show that the EM iteration converges to the MLE under minimal conditions.	
	Define the MLE as any global maximizer of the likelihood function, i.e.,
	\begin{equation}
	\thetaMLE \in \arg\max_{\theta \in \reals^n} \ell_n(\theta),	
	\label{eq:mle-n}
	\end{equation}
	where the log likelihood $\ell_n$ is given in \prettyref{eq:LL}.
	Note that from first principles it is unclear whether there exists a unique global maximizer. 
	Furthermore, our previous analysis only shows that with high probability, the EM iterates are within the optimal rate of the true mean $\theta_*$ after a certain number of iterations. Indeed, 
	for $\thetanorm \geq (\frac{C d \log^3 n}{n})^{1/4}$, \prettyref{thm:main-alpha-large} and \prettyref{thm:phase1} together imply that, with probability $1-o(1)$, 
	\begin{equation}
	\ell(\theta_t,\theta_*) \leq \pth{\frac{C d \log n}{n}}^{1/4}
	\label{eq:beforeMLE}
	\end{equation}
	for all $t \geq T \triangleq \frac{C \log n}{\thetanorm^2}$, for some constant $C$. This, however, has no direct bearing on the convergence of the sequence $\theta_t$, since it does not rule out the possibility that $\theta_t$ oscillates within the optimal rate of $\theta_*$.
Next we will address both questions by showing that the MLE is unique and coincides with the limit of the EM iteration.
	
	\begin{theorem}
	\label{thm:MLE}		
		Assume that $n \geq C_1 d \log^3 d$ and $(C_2 \frac{d \log^3 n}{n})^{1/4}\leq \thetanorm \leq r$, where 
		$C_1,C_2$ are constants depending only on $\thetaub$.		
		With probability at least $1-2 n^{-1}$, for all $t\geq 1$,
		\begin{equation}
		\|\theta_{T+t}-\thetaMLE\| \leq e^{-c t\|\theta_*\|^2} \|\theta_T-\thetaMLE\|, 	
		\label{eq:mle-converge}
		\end{equation}
		for some absolute constant $c$.
		In particular, 	$\lim_{t\diverge} \theta_t$ exists and coincides with $\thetaMLE$, the unique (up to a global sign change) global maximizer of \prettyref{eq:mle-n}.

	\end{theorem}
	
	
Next we prove \prettyref{thm:MLE}.
Note that $\thetaMLE$ is a critical point, i.e., $\nabla \ell_n(\thetaMLE) = 0$. 
Recall from \prettyref{eq:EM-GD} that the EM iteration corresponds to gradient ascent of the log likelihood $\ell_n$ with step size one.
	Applying the Taylor expansion of $\nabla \ell_n$ at $\thetaMLE$, we get from 	\prettyref{eq:EM-GD}
	\begin{align}
	\theta_{t+1}-\thetaMLE
	= & ~ \theta_{t}-\thetaMLE + \nabla \ell_n(\theta_t)\nonumber\\
	= & ~ (I + \nabla^2 \ell_n(\xi_t)) (\theta_t-\thetaMLE), 	\label{eq:EM-GD1}
	\end{align}
	where $\xi_t = \alpha \theta_t + (1-\alpha) \thetaMLE$ for some $\alpha\in[0,1]$.
	The key lemma is
	\begin{lemma}
	\label{lmm:hessMLE}	
	Under the setting of \prettyref{thm:MLE}, denote 
	$\delta \triangleq (c \frac{d \log^3 n}{n})^{1/4}$ for some constant $c$ depending only on $r$.
		With probability at least $1-2n^{-1}$, 
		for all $\theta$ such that $\ell(\theta,\theta_*) \leq \delta$.
		\[
		0 \preceq I + \nabla^2 \ell_n(\theta) 	\preceq e^{-c\|\theta_*\|^2} I.
		\]
	\end{lemma}
We now apply \prettyref{lmm:hessMLE}	to show the convergence of $\theta_t$ to $\thetaMLE$. 
To apply \prettyref{lmm:hessMLE}, we first need some crude guarantee on the MLE. 
The results of \cite{ho2016convergence} show that (cf.~\cite{DWYZ19}) with probability at least $1-\exp(-cd \log^2 n)$, 
$H(P_{\thetaMLE},P_{\theta_*}) \leq (C \frac{d \log^2 n}{n})^{1/2}$ and $\ell(\thetaMLE,\theta_*) \le (C \frac{d \log^2 n}{n})^{1/4}$ for some universal constants $c,C$.

	Since $\thetanorm > 2 \delta$ for all sufficiently large $n$, 
on the event that $\ell(\thetaMLE,\theta_*) \leq \delta$ and $\ell(\theta_T,\theta_*) \leq \delta$, 	
	$\theta_T$ and $\thetaMLE$ must both belong to exactly one of the two balls $B(\theta_*,\delta)$ and $B(-\theta_*,\delta)$. WLOG, assume the former.
	Taking norms on both sides of \prettyref{eq:EM-GD1} and applying \prettyref{lmm:hessMLE}, 
	we have
	\[
	\|\theta_{T+1}-\thetaMLE\| \leq e^{-c\|\theta_*\|^2} \|\theta_T-\thetaMLE\|, 	
	\]
	and hence \prettyref{eq:mle-converge} follows, which, 
	in particular,  implies the convergence of $\{\theta_t\}$ and the uniqueness of $\thetaMLE$.

\section{Discussions and open problems}
	\label{sec:discuss}
	
	We conclude this paper by discussing some technical aspects of the results and related or open problems:
	
	\paragraph{Small initialization}
	In this paper, we showed that the EM algorithm achieves the near-optimal rate and converges to the MLE when the direction of the initialization $\theta_0$ is uniform on the sphere and $\theta_0$ is sufficiently close to zero, specifically, $\|\theta_0\| = \Theta((\frac{d}{n} \log n)^{1/4})$ (cf.~\prettyref{thm:phase1}).
	Computationally speaking, using a small initialization does not compromise the needed number of iterations as the signal grows rapidly according to \prettyref{eq:alphagrow0} in the initial Phase I. 
	Technically speaking, the main reason for using a small initialization in the proof is to ensure the orthogonal component $\beta_t$ stays within the near-optimal rate throughout the entire trajectory, as shown in \prettyref{thm:main-beta}. An added bonus is that the signal component $\alpha_t$ converges monotonically; as demonstrated in \prettyref{fig:EMpop}, this can fail for large initialization.
	We conjecture that the same result applies to $\|\theta_0\| =\Theta(1)$. Proving such a result entails a refined analysis of the initial phase since $\alpha_t$ initially decays due to $\beta_t$ being as large as a constant (see \prettyref{fig:non-monotone}).

	%
	%
	
	\paragraph{Extensions}
	In this paper we considered the simple symmetric 2-GM model. It is of great interest to understand the performance or limitations of EM algorithms in more general Gaussian mixture models, e.g., multiple components, unknown covariance matrix, asymmetric and unknown weights, and, more generally, location-scale mixtures.
	The optimal and adaptive rates of location mixtures in one dimension were obtained in \cite{HK2015} and shown to be achieved by the generalized method of moments \cite{WY18}. It remains open whether the corresponding EM algorithm achieves competitive performance.
	One immediate hurdle is the existence of bad fixed points, which can exist for population EM for $3$-GM even in one dimension	\cite{jin2016local}.
	
	Beyond Gaussian mixture models, statistical problems with missing data, and other latent variable models such as 	mixture of regression  and 
	alignment problems in cryo-EM \cite{sigworth2010introduction} are major avenues where EM algorithm are applied. 
	Promising results have been obtained recently in \cite{BWY17,kwon2018global}, although finite-sample finite-iteration guarantees and analysis for random initializations are still lacking.

The present paper concerns analyzing EM algorithm for the purpose of parameter estimation. For the related problem of \emph{classification}, that is, recovering the labels of each sample with small error rate, we refer to	the recent work on Lloyd's algorithm \cite{lu2016statistical} and optimal rates \cite{ndaoud2018sharp}.
It remains open to understand the performance of EM algorithm for clustering and whether it achieves the optimal rates.

\section{Proofs in \prettyref{sec:onedim}}
\label{sec:pf-onedim}

\subsection{Proofs of \prettyref{thm:main-1D} and \prettyref{cor:em1d}}
	\begin{proof}[Proof of \prettyref{thm:main-1D}]
\emph{Step 1.} We show that 
\begin{equation}
\theta_t\leq\otheta_t
\label{eq:sandwich1}
\end{equation}
by induction on $t$.
The base case of $t=0$ is clearly true.
Assume that \prettyref{eq:sandwich1} holds for $t$.
Then
\begin{align*}
\theta_{t+1}
= & ~ f(\theta_t)+\Delta_n(\theta_t) \\
\leq & ~ f(\theta_t)+ w_n \theta_t \\
\leq & ~ f(\otheta_t)+ w_n \otheta_t = \otheta_{t+1},
\end{align*}
where we used the fact that $\theta \mapsto f(\theta) + w_n \theta$ is increasing on $\reals_+$.

\medskip
\emph{Step 2.} We show that $\otheta_t \leq C_1$ for all $t$ for some constant $C_1$.
This simply follows from the fact that $f$ is bounded. By \prettyref{lmm:fprop} and the assumption $\theta_*\leq r$,
\[
\otheta_{t+1} = f(\otheta_t) + w_n \otheta_t \leq 1+r + w_n \otheta_t,
\]
where $w_n \leq c_0 \leq \frac{1}{2}$ on the event \prettyref{eq:event}.
Setting $C_1=2(1+r)$ and letting $n \geq 4 C_0^2$, the proof follows from induction on $t$.

\medskip
\emph{Step 3.}
We show that 
\begin{equation}
\theta_t\geq\utheta_t \geq 0,
\label{eq:sandwich2}
\end{equation}
by induction on $t$. The base case of $t=0$ is clearly true.
Assume that \prettyref{eq:sandwich2} holds for $t$.
Then
\begin{align*}
\theta_{t+1}
\geq & ~ f(\theta_t)- w_n \theta_t \\
\geq & ~ f(\utheta_t)- w_n \utheta_t = \utheta_{t+1},
\end{align*}
where we used the fact $C_1\geq \theta_t \geq  \utheta_t$ as shown in the previous step and $\theta \mapsto f(\theta) - w_n \theta$ is increasing on $[0,C_1]$. To see this, note that $f(\theta)$ is concave on $\reals_+$. Therefore 
$f'(\theta) \geq f'(C_1) \geq w_n$ which holds on the event \prettyref{eq:event} provided that $c_w \leq f'(C_1)$.
Finally, $\utheta_{t+1} \geq 0$ follows again from monotonicity and $\utheta_t \geq 0$.
This completes the proof of \prettyref{eq:sandwich}.

\medskip
\emph{Step 4.} 
Next we prove the convergence of $\{\otheta_t\}$ to $\otheta$. 
Recall $q(\theta)=\frac{f(\theta)}{\theta}$ from \prettyref{lmm:fprop}, which is a decreasing function on $\reals_+$.
By definition, we have
\begin{equation}
q(\otheta) = 1-w_n.
\label{eq:qotheta}
\end{equation}
Furthermore, we have, crucially, $f(\theta) + w_n\theta \gtrless \theta$ if $\theta \lessgtr \otheta$. 
Therefore, $|\otheta_{t+1} - \otheta| < |\otheta_t - \otheta|$ and hence $\otheta_t \to \otheta$ as $t\diverge$.
Similarly, if $\theta_*^2 \geq w_n$, then we have $\utheta_t \to \utheta$; 
if $\theta_*^2 < w_n$, then $\utheta = 0$ by definition and we have $\liminf \utheta_t \geq \utheta$.

\medskip
\emph{Step 5.}
Finally, we show \prettyref{eq:sandwich-accuracy}.
Recall $q(\theta)=\frac{f(\theta)}{\theta}$ from \prettyref{lmm:fprop}.
If $\theta_*^2 \geq w_n$, by definition \prettyref{eq:otheta}--\prettyref{eq:utheta}, we have
\begin{align*}
q(\otheta) = &~ 1-w_n \\
q(\utheta) = &~ 1+w_n \\
q(\theta_*) = &~ 1.
\end{align*}
If $\theta_*^2 \leq w_n$, then $\utheta=0$ by definition. 
In both cases, since $q$ is decreasing on $\reals_+$ by \prettyref{lmm:fprop}, we have
\[
\utheta\leq\theta_*\leq\otheta.
\]
Furthermore, since $\theta_* \in [0,r]$, by \prettyref{eq:qtheta}, for all $\theta \in [0,C_1]$,
\begin{equation}
q'(\theta) \ \leq - \frac{2\theta}{3} \expect{\frac{Y^4}{\cosh^2(\theta Y)}} \leq - C_4 \theta 
\label{eq:qp}
\end{equation}
where $C_4$ is a constant that depends on $r$ (recall $C_1=2r+1$).

Let $\oepsilon =\otheta-\theta_*$. Then
\begin{align*}
-w_n 
= & ~  q(\theta_*+\oepsilon ) -q(\theta_*) = \int^{\theta_*+\oepsilon}_{\theta_*} q'(\tau) d\tau \\
\overset{\prettyref{eq:qp}}{\leq} & ~ - \frac{C_4}{2} ((\theta_*+\oepsilon)^2-\theta_*^2)	= - \frac{C_4}{2} (2\theta_* \oepsilon + {\oepsilon}^2).
\end{align*}
Hence
\begin{equation}
0 \leq \oepsilon \leq \min\sth{\frac{w_n}{C_4 \theta_* }, \sqrt{\frac{2w_n}{C_4}}}
\leq C_3 \min\sth{\frac{w_n}{\theta_* }, \sqrt{w_n}}.
\label{eq:eps+}
\end{equation}

Similarly, let $\uepsilon =\theta_*-\utheta$. Then $0 \leq \uepsilon \leq \theta_*$. 
Furthermore, if $\theta_*^2 \geq w_n$,
\begin{align*}
w_n 
= & ~  q(\theta_*-\uepsilon) -q(\theta_*) = \int_{\theta_*-\uepsilon}^{\theta_*} -q'(\tau) d\tau \\
\overset{\prettyref{eq:qp}}{\geq} & ~ \frac{C_4}{2} (\theta_*^2 - (\theta_*-\uepsilon)^2)	= \frac{C_4}{2} (2\theta_*- \uepsilon) \uepsilon
\geq \frac{C_4}{2} \theta_* \uepsilon.
\end{align*}
Hence
\begin{equation}
0 \leq \uepsilon \leq \min\sth{\theta_*, \frac{2w_n}{C_4 \theta_* }} \leq C_5
\min\sth{\frac{w_n}{\theta_* }, \sqrt{w_n} }.
\label{eq:eps-}
\end{equation}
If $\theta_*^2 < w_n$, since $\uepsilon\leq \theta_*$, then \prettyref{eq:eps-} holds automatically. Thus, combining \prettyref{eq:eps+} and \prettyref{eq:eps-} yields
\begin{equation}
		\theta_* - \epsilon \leq \utheta \leq \liminf_{t\to\infty} \theta_t  \leq \limsup_{t\to\infty} \theta_t \leq \otheta \leq  \theta_*  + \epsilon,
		\label{eq:sandwich-accuracy-lim}
		\end{equation}
		where 
		$\epsilon\triangleq C_6 \min\{\frac{w_n}{\theta_* }, \sqrt{w_n}\}$.

\medskip
\emph{Step 6.} Finally, we provide a finite-iteration version of \prettyref{eq:sandwich-accuracy-lim}.
In view of the sandwich inequality \prettyref{eq:sandwich}, it suffices to determine the convergence rate of $\{\otheta_t\}$ and $\{\utheta_t\}$.
Consider two cases separately.

\textbf{Case I:  $\theta_*^2 \leq 2w_n$.}
Let $\oepsilon_t = \otheta_t-\otheta$. 
If $\oepsilon_t \leq 0$, then we have 
$0 \leq \theta_t \leq \otheta_t \leq \otheta \leq \theta^*+\epsilon \lesssim n^{-1/4}$, which is already within the optimal rate of convergence.
So it suffices to consider $\oepsilon_t \geq 0$, i.e., $\otheta_t$ converging to $\otheta$ from above. Then
\begin{align}
\oepsilon_{t+1}
= & ~ \otheta_t \pth{q(\otheta_t) + w_n} -\otheta 	\nonumber \\
\overset{\prettyref{eq:qotheta}}{=} & ~ \oepsilon_t	+ \otheta_t [q(\otheta_t)-q(\otheta)] \nonumber\\
\overset{\prettyref{eq:qp}}{\leq} & ~ \oepsilon_t - C_6 (\oepsilon_t+\otheta) (\otheta \oepsilon_t + (\oepsilon_t)^2) \nonumber\\
\leq & ~ \oepsilon_t - C_6 ((\otheta)^2 \oepsilon_t + (\oepsilon_t)^3) \label{eq:epst+}\\
\leq & ~ \oepsilon_t - C_6' (\oepsilon_t)^3
\end{align}
where $C_6'=\min\{C_6,\frac{1}{r_0^2}\}$. Next we apply \prettyref{lmm:rate} with $h(x) = C_6' x^3$ to the sequence $\{\oepsilon_t\}$, which satisfies $h(x) < x$ for all $x \in (0,\epsilon_0^+)$, since $\epsilon_0^+ \leq \theta_0\leq r_0$. We have $G(x) = \int_x^{r_0} \frac{1}{h(\tau)}d\tau = C_7 (\frac{1}{x^2}-\frac{1}{r_0^2})$, we conclude that
\[
\oepsilon_t \leq \frac{1}{\sqrt{t/C_7+1/r_0^2}} \leq \sqrt{\frac{C_7}{t}}.
\]
Thus for all $t \geq C_7/w_n$, we have 
$\oepsilon_t \leq \sqrt{w_n}$ and hence $|\otheta_t - \theta_*| \lesssim \sqrt{w_n}$.

\textbf{Case II:  $\theta_*^2 \geq 2w_n$.}
Let $\oepsilon_t = \otheta_t-\otheta$.
First assume $\oepsilon_t \geq 0$, in which case $\oepsilon_t$ converges to zero from above.
Since $\theta_* \gtrsim \sqrt{w_n}$, we have $\utheta \asymp \otheta \asymp \theta_*$.
Continuing from \prettyref{eq:epst+}, we conclude that 
$\epsilon_{t+1}^+ \leq (1 - C_8 \theta_*^2)\oepsilon_t$
Therefore for all sufficiently large $n$, as soon as $t \geq C_8' \frac{\log n}{\theta_*^2}$, we have  $\theta_t - \theta_* \leq \oepsilon_t \leq \frac{1}{\theta^* \sqrt{n}}$. 
Similarly, if $\oepsilon_t \leq 0$, we have $\oepsilon_{t+1} \geq \oepsilon_t (1-C_8 \theta_*^2)$, which converges to zero from below.

Next we analyze the convergence rate of $\{\utheta_t\}$. 
 Let $\uepsilon_t = \utheta-\utheta_t$. 
We only consider the case of $\uepsilon_t \geq 0$ as the other case is entirely analogous.
Since $f(\theta) - w_n \theta > \theta$ if and only if $\theta < \utheta$, we have
$\utheta_t \to \utheta$ from below and $\uepsilon_t$ is a decreasing positive sequence.
Let $c_0= \frac{1}{200 \sqrt{3+r^4}}$.
Consider two cases:

\textbf{Case II.1:  $\utheta_t \geq c_0 \theta_*$.}
Entirely analogous to \prettyref{eq:epst+}, we have
\begin{align}
\uepsilon_{t+1}
= & ~ \uepsilon_t	- \utheta_t [q(\utheta_t)-q(\utheta)] \nonumber\\
\leq & ~ \uepsilon_t -  C_6  (\utheta)^2 \uepsilon_t \nonumber \\
\leq & ~ \uepsilon_t \pth{1 - C_9 \theta_*^2}. \label{eq:epst-}
\end{align}
Since $\epsilon_0^- = \utheta-\theta_0 \leq \theta_* \leq r$, 
for all sufficiently large $n$, as soon as $t \geq C_9' \frac{\log \frac{1}{w_n}}{\theta_*^2}$, we have  $\theta_t - \theta_* \geq - \uepsilon_t \geq - \frac{w_n}{\theta^* }$.

\textbf{Case II.2:  $0<\utheta_t \leq c_0\theta_*$.}
Recall from  \prettyref{lmm:fprop} that $f(0)=f''(0)=0$ and $f'(0)=1+\theta_*^2$.
Furthermore, $f'''(\theta) = \Expect[Y^4 \tanh'''(\theta Y)]$. Since $|\tanh'''| \leq 2$, we have for all $\theta$, 
\begin{equation}
|f'''(\theta)| \leq 2 \Expect[Y^4] \leq 16(3+r^4).
\label{eq:fppp}
\end{equation}
Therefore the Taylor expansion of $f$ at zero yields
\begin{align*}
\utheta_{t+1}
= & ~ f(\utheta_t)  - w_n \utheta_t \geq \pth{1+ \theta_*^2 - w_n - \frac{16(3+r^4)}{6} c_0^2 \theta_*^2  } \utheta_t
\geq \pth{1+ \frac{\theta_*^2 }{4} } \utheta_t,
\end{align*}
where the last inequality is by the choice of $c_0$.
Therefore in at most $ \frac{C_{11}}{\theta_*^2} \log \frac{\theta_*}{\theta_0}$ iterations, we have $\utheta_t \geq c_0\theta_*$ which enters the previous Case II.1.

In summary, for all $t \geq \frac{C_{12}}{\theta_*^2} \log \frac{\theta_*}{\theta_0 w_n}$, we have $|\theta_t - \theta_*| \lesssim \frac{w_n}{\theta^*}$.
\end{proof}

\begin{proof}[Proof of \prettyref{cor:em1d}]
An inspection of the proof of \prettyref{thm:main-1D} shows that the guarantees in \prettyref{eq:sandwich-accuracy} and \prettyref{eq:sandwich-T} apply if $w_n$ is replaced by any upper bound thereof, which we choose to be $\max\{w_n,\frac{1}{\sqrt{n}}\}$. Then on the event $E$ defined in \prettyref{eq:event}, we have
\begin{equation}
		\ell(\theta_t,\theta_*) \leq  \tau_2 \min\sth{\frac{\max\{w_n,\frac{1}{\sqrt{n}}\}}{\theta_* }, \sqrt{\max\sth{w_n,\frac{1}{\sqrt{n}}}}}
		\end{equation}
		holds for all $t$ satisfying \prettyref{eq:em1d-finitet-T}.
Taking expectation and using \prettyref{eq:W1-concentrate1} and Jensen's inequality, we have
	\[
		\Expect[\ell(\theta_t,\theta_*) \Indc_{E}] \leq  \tau_2 \min\sth{\frac{1}{\theta_* \sqrt{n}}, \frac{1}{n^{1/4}}},
	\]
	where the high-probability event $E$ is in \prettyref{eq:event}.
		Finally, by definition of the EM map, we have $|\theta_t| \leq \|f_n\|_\infty \leq \Expect_n |Y|$ and hence $|\ell(\theta_t,\theta_*)| \leq r + \Expect_n |Y|$. Therefore by the Cauchy-Schwarz inequality, we have
		\[
		\Expect[\ell(\theta_t,\theta_*) \Indc_{E^c}] \leq \sqrt{\prob{E^c}} \sqrt{\Expect[(r + \Expect_n |Y|)^2]} \overset{\prettyref{eq:W1-concentrate2}}{\leq} C \exp(-c n^{1/3})
		\]
		for some constants $c,C$ depending on $\thetaub$. 		Combining the previous two displays yields the desired \prettyref{eq:em1d-finitet}.		
\end{proof}

\subsection{Proof of \prettyref{lmm:fprop}}
\begin{proof}
	\begin{enumerate}
		\item By definition,
		\begin{align*}
		f'(\theta) 		= & ~ 	\Expect[Y^2 \tanh'(\theta Y)] = \expect{\frac{Y^2}{\cosh^2(\theta Y)}} \geq 0 \\
		f''(\theta) 		= & ~ 	\Expect[Y^3 \tanh''(\theta Y)] = -2 \expect{\frac{Y^3 \tanh(\theta Y)}{\cosh^2(\theta Y)}}.
		\end{align*}
		\item Clearly $f''(\theta)$ is negative (resp.~positive) when $\theta$ is  positive (resp.~negative).
		\item $f(0)=f''(0)=0$ by definition, $f'(0)=\Expect[Y^2]$ and 
		\begin{align*}
		f'(\theta_*)
		= & ~ \expect{\frac{Y^2}{\cosh^2(\theta_* Y)}}  \\
		= & ~ \expect{\frac{Z^2}{\cosh(\theta_* Z)}} \exp(-{\theta_*}^2/2) \qquad Z\sim N(0,1) \\
		\leq & ~ \expect{Z^2} \exp(-{\theta_*}^2/2) = \exp(-{\theta_*}^2/2),
		\end{align*}
		where the second equality follows from a change of measure from $Y$ to $Z$ (cf.~\prettyref{lmm:com}).
		
		\item The monotonicity of $q$ simply follows from the concavity of $f$ on $\reals_+$ and $f(0)=0$.		By the symmetry of the distribution of $Y$, we have
		\[
		q'(\theta) = - \expect{\frac{Y \sinh(2\theta Y)-2\theta Y^2}{2\theta^2 \cosh^2(\theta Y)} \Big| Y \geq 0} 
		\leq  - \frac{2\theta}{3} \expect{\frac{Y^4}{\cosh^2(\theta Y)} \Big| Y \geq 0}
		\]
		where we used the fact that $\sinh(x) \geq x + x^3/6$ for $x\geq 0$; (b) follows from $\cosh \geq 1$ and Jensen's inequality.
	\end{enumerate}
\end{proof}

\section{Proofs in \prettyref{sec:concentration}}
	\label{sec:pf-concentration}

\begin{proof}[Proof of \prettyref{thm:concentration}]
First of all, by definition, we have
\[
\|f_n(\theta)\| = \|\Expect_n[Y \tanh \iprod{\theta}{Y}]\| \leq \Expect_n[\|Y\|] \leq \sqrt{\Expect_n[\|Y\|^2]}.
\]
Define the event
\[
E_2 = \{\Expect_n[\|Y\|^2] \leq 2\|\theta_*\|^2 + 10d 	\}.
\]
Since $\Expect_n[\|Y\|^2] \leq 2 \|\theta_*\|^2 + 2 \Expect_n[\|Z\|^2]$, where $n\Expect_n[\|Z\|^2]\sim\chi^2_{nd}$.
By the $\chi^2$ tail bound \prettyref{eq:LM} in \prettyref{app:lemma}, 
\begin{equation}
\prob{E_2} \geq 1 - \exp(-nd).
\label{eq:PE2}
\end{equation}

Next, we show that with probability at least $1- \exp(-c_0 d \log n)$,
\[
\|\Delta_n(\theta)\| \leq  C_0 \|\theta\| (1+r) \sqrt{\frac{d}{n} \log n}
\]
for all $\theta\in B(R)$.

Let $Y,Y_1,\ldots,Y_n\iiddistr P_{\theta_*}$.
Let $\calC \subset S^{d-1}$ be an $\epsilon$-covering of $S^{d-1}$ in Euclidean distance, where $\epsilon \leq \frac{1}{2}$ is to be specified later. It is well-known (cf.~\cite{Vershynin-HDP}) that $\calC$ can be chosen so that $|\calC| \leq (1+\frac{2}{\epsilon})^d \leq (\frac{3}{\epsilon})^d$. Furthermore, for any $y\in\reals^d$, 
\[
\|y\| \leq \frac{1}{1-\epsilon}\max_{u \in \calC} \iprod{u}{y}
\]
and hence
\[
\|\Delta_n(\theta)\| \leq 2\max_{u \in \calC} \Expect [\iprod{u}{Y} \tanh\iprod{\theta}{Y}]-\Expect_n [\iprod{u}{Y} \tanh\iprod{\theta}{Y}]
\]

For each $\theta\in\reals$, there exists $v \in \calC$ such that $\|\|\theta\|v - \theta\| \leq \epsilon \|\theta\|$.
For any $u\in\calC$, using Cauchy-Schwarz and the fact that $\tanh$ is 1-Lipschitz, we have
\begin{align*}
& ~ |\Expect [\iprod{u}{Y} \tanh(\iprod{\theta}{Y})] - \Expect [\iprod{u}{Y} \tanh(\|\theta\|\iprod{v}{Y})]| \\
\leq & ~ |\Expect [|\iprod{u}{Y}| \iprod{\theta-\|\theta\|v}{Y}]	\leq \Expect[\|Y\|^2] \|u\| \|\theta-\|\theta\|v\| \leq \epsilon  \|\theta\| \Expect[\|Y\|^2].
\end{align*}
Similarly,
\begin{align*}
|\Expect_n [\iprod{u}{Y} \tanh(\iprod{\theta}{Y})] - \Expect_n [\iprod{u}{Y} \tanh(\|\theta\|\iprod{v}{Y})]| 
\leq  \epsilon \Expect_n[\|Y\|^2] \|\theta\|. 
\end{align*}
Therefore
\begin{align*}
\|\Delta_n(\theta)\| 
\leq  & ~  2\max_{u,v \in \calC} \Big|\Expect [\iprod{u}{Y} \tanh(\|\theta\|\iprod{v}{Y})]-\Expect_n [\iprod{u}{Y} \tanh(\|\theta\|\iprod{v}{Y})]\Big|\\
 & ~ + \epsilon \|\theta\| ( \Expect[\|Y\|^2] + \Expect_n[\|Y\|^2]),
\end{align*}
and hence
\begin{align*}
 \sup_{0<\|\theta\| \leq R}
\frac{\|\Delta_n(\theta)\|}{\|\theta\|}
\leq & ~  2\max_{u,v \in \calC} \sup_{0<a\leq R}
\underbrace{\frac{1}{a} \Big|\Expect [\iprod{u}{Y} \tanh(a \iprod{v}{Y})]-\Expect_n [\iprod{u}{Y} \tanh(a \iprod{v}{Y})]\Big|}_{\triangleq F(u,v,a)} \\
& ~ 	+ \epsilon ( \Expect[\|Y\|^2] + \Expect_n[\|Y\|^2]),
\end{align*}
where $\Expect[\|Y\|^2] = d + \thetanorm^2 \leq d + r^2$.
Consider two cases separately:

\textbf{Case I: $0<a\leq \epsilon$.}
Since $|\tanh'|\leq 1$ and $|\tanh''|\leq 1$ everywhere, we have
$|\frac{1}{a} \Expect [\iprod{u}{Y} \tanh(a \iprod{v}{Y})] - \Expect[\iprod{u}{Y} \iprod{v}{Y})]| \leq \epsilon \Expect[|\iprod{u}{Y}| \iprod{v}{Y}^2] \leq 
\epsilon \Expect[\|Y\|^3]$, and similarly, 
$|\frac{1}{a} \Expect_n[\iprod{u}{Y} \tanh(a \iprod{v}{Y})] - \Expect_n[\iprod{u}{Y} \iprod{v}{Y})]| \leq \epsilon \Expect_n[\|Y\|^3]$.
Therefore
\begin{align*}
 \sup_{0<a\leq \epsilon}F(u,v,a)
\leq \big|\Expect[\iprod{u}{Y} \iprod{v}{Y})] - \Expect_n[\iprod{u}{Y} \iprod{v}{Y})] \big|  + \epsilon (\Expect[\|Y\|^3]+\Expect_n[\|Y\|^3]).
\end{align*}
For any $u,v\in\calC$, note that 
\[
\iprod{u}{Y} \iprod{v}{Y} = 
\iprod{u}{\theta_*} \iprod{v}{\theta_*} + \iprod{XZ}{\iprod{u}{\theta_*}v+\iprod{v}{\theta_*}u} + \iprod{u}{Z} \iprod{v}{Z}.
\]
Since $\|\theta_*\|\leq r$ by assumption and $\|\iprod{u}{Z} \iprod{v}{Z}\|_{\psi_1} \leq \|\iprod{u}{Z}\|_{\psi_2} \|\iprod{v}{Z}\|_{\psi_2} =1$ (cf.~\cite[Lemma 2.7.7]{Vershynin-HDP}), we conclude that $\iprod{u}{Y}\iprod{v}{Y}$ is $C_2(r+1)$-subexponential 
By Bernstein's inequality (cf.~\cite[Theorem 2.8.1]{Vershynin-HDP}), for any $b$ such that $b d \log n \leq n$,
\begin{equation}
\prob{\big|\Expect[\iprod{u}{Y} \iprod{v}{Y})] - \Expect_n[\iprod{u}{Y} \iprod{v}{Y})] \big| \geq (1+r) \sqrt{\frac{b d \log n}{n}}} 
\leq \exp(- c b d \log n),
\label{eq:ber1}
\end{equation}
where $c$ is some absolute constant. Furthermore, $\Expect[\|Y\|^3] \leq C_4 (r + \sqrt{d})^3$,
and $\Expect_n[\|Y\|^3] \leq \max_{i\in[n]} \|Y_i\|^3$.
Since $n \geq d \log d$, 
$\prob{\|Y_i\| \geq \sqrt{n}} \leq \exp(- c n)$.
Therefore by the union bound, $\Expect_n[\|Y\|^3] \leq  n^{3/2}$ with probability at least $1-\exp(- c' n)$.

\textbf{Case II: $\epsilon \leq a \leq R$.}
Let $\calR$ be an $\epsilon^2$-net for the interval $[\epsilon,R]$, so that for any $a \in [\epsilon,R]$, there exists $a' \in \calR$ such that $|a-a'|\leq \epsilon^2$. Then
$|\frac{1}{a}\Expect [\iprod{u}{Y} \tanh(a \iprod{v}{Y})] - \frac{1}{a'}\Expect [\iprod{u}{Y} \tanh(a' \iprod{v}{Y})]| 
\leq 2 \frac{|a-a'|}{a} \Expect[|\iprod{u}{Y}\iprod{v}{Y}|] \leq 2 \epsilon \Expect[\|Y\|^2]$. Therefore
\[
\sup_{\epsilon \leq a \leq R} F(u,v,a) 
\leq 
\max_{a \in\calR} F(u,v,a)  + 2 \epsilon (\Expect[\|Y\|^2]+\Expect_n[\|Y\|^2]).
\]
For any $u,v \in \calC$ and $a \in \calR$, 
$|\frac{\Iprod{u}{Y} \tanh(a\Iprod{v}{Y})}{a}| \leq |\Iprod{u}{Y}||\Iprod{v}{Y}|$. 
Therefore $|\frac{\Iprod{u}{Y} \tanh(a\Iprod{v}{Y})}{a}|$ is $C_2(1+r)$-subexponential.
Again by Bernstein's inequality, we have
\begin{equation}
\prob{|F(u,v,a)| \geq (1+r) \sqrt{\frac{b d \log n}{n}}} \leq \exp(- c b d \log n).
\label{eq:bern2}
\end{equation}

Set $\epsilon = n^{-4}$ so that $|\calC| \leq (3n^4)^{d}$ and $|\calR| \leq R n^4$. 
Applying the union bound to both cases and choosing a sufficiently large constant $b$ completes the proof.
\end{proof}

\section{Proofs in \prettyref{sec:ddim}}
	\label{sec:pf-ddim}

		\subsection{Proofs of Theorems \ref{thm:main-beta}--\ref{thm:main-alpha-large}}
	\label{sec:pf-ddim-thm}

Throughout this section denote for brevity $s\triangleq \thetanorm$.

\begin{proof}[Proof of \prettyref{thm:main-beta}]
	We first show that the sequence $\{\alpha_t,\beta_t\}$ is bounded.
	By assumption, $\omega \leq \frac{1}{2}$ and $\|\theta_0\| \leq 1$ by \prettyref{eq:theta0-bdd}.
	Using the bounded property of the $F$ and $G$ maps in  \prettyref{lmm:FG} and induction on $t$, we have	
		\begin{equation}
		|\alpha_t| \leq \Gamma, \quad 0  \leq \beta_t  \leq \Gamma		
		\label{eq:ab-bounded}
		\end{equation}
		where $\Gamma = 2(\|\theta_*\|+\sqrt{2/\pi}) \leq 2r+2$.

		Combining \prettyref{eq:Gnew} and \prettyref{eq:beta-iterate}, we have
		\begin{align}
		\beta_{t+1}
		\leq & ~ 	\beta_t\pth{ 1 - \frac{\alpha_t^2+\beta_t^2}{2+4(\alpha_t^2+\beta_t^2)}}
+ \omega(|\alpha_t| + \beta_t) 
\end{align}
from which \prettyref{eq:beta-uncond2} follows.
To show \prettyref{eq:beta-uncond}, note that, in view of \prettyref{eq:ab-bounded}, we have
\begin{align}
		\beta_{t+1} \leq & ~ \beta_t\pth{ 1 - \frac{\alpha_t^2+\beta_t^2}{2+8\Gamma^2}}
+ \omega(|\alpha_t| + \beta_t)	\label{eq:betaraw1} \\
\leq  & ~ \beta_t (1+\omega) - \frac{\beta_t^3}{2+8\Gamma^2}
+ \sup_{0 \leq \alpha \leq \Gamma}\pth{\omega \alpha - \frac{\alpha^2 \beta_t}{2+8\Gamma^2}} \\
\leq  & ~ \beta_t (1+\omega) - \frac{\beta_t^3}{2+8\Gamma^2}
+ \min\sth{\frac{\omega^2 (2+8\Gamma^2) }{4\beta_t}, \omega \Gamma}. \label{eq:beta-uncond3}
		\end{align}		
Let $C_1=2+8\Gamma^2$. 
Let $\beta$ be any limiting point of the sequence $\{\beta_{t}\}$.
Taking limits on both sides we have
\[
\frac{\beta^3}{C_1}
\leq 
\omega \beta  + \frac{\omega^2 C_1}{4\beta}
\leq 
2 \pth{ \omega \beta  \vee \frac{\omega^2 C_1}{4\beta} }
\]
which implies that either $\beta \leq \sqrt{2C_1\omega}$ or $\beta \leq (\omega^2 C_1^2/2)^{1/4}$. So we conclude \prettyref{eq:beta-rate}.

	Finally, we prove \prettyref{eq:betanearzero}. We show by induction that there exists 
some constant $a$ depending only on $r$, such that $\beta_t \leq a \sqrt{\omega}$ for all $t \geq 0$. The base case is the assumption \prettyref{eq:theta0nearzero}.
Next, fix some constant $b$ to be specified and consider two cases:

\textbf{Case I: $\beta_t \leq b \omega$.}
From \prettyref{eq:beta-uncond}, we get
\begin{align*}
\beta_{t+1}
\leq \beta_t(1+\omega) - \frac{\beta_t^3}{C_1} + \omega \Gamma \leq \omega \pth{b+\omega + \Gamma} \leq a \sqrt{\omega},
\end{align*}
provided that $\sqrt{\omega} \leq \frac{a}{b+\omega + \Gamma}$.

\textbf{Case II: $b \omega \leq \beta_t  \leq a \sqrt{\omega}$.}
Again from \prettyref{eq:beta-uncond}, we get $\beta_{t+1} \leq h(\beta_t)$, where 
$h(\beta) \triangleq \beta(1+\omega) - \frac{\beta^3}{C_1} + \frac{\omega^2 C_1}{2\beta}$.
Note that 
$\frac{d}{d\beta} h(\beta) = 
1+\omega - \frac{3\beta^2}{C_1} - \frac{\omega^2 C_1}{2\beta^2} \geq 
 1 - \frac{C_1}{3b^2} +\omega (1- \frac{3a^2}{C_1}) \geq 0$, provided that
$\frac{C_1}{3b^2} \leq \frac{1}{2}$ and $\omega (1- \frac{3a^2}{C_1}) \geq -\frac{1}{2}$.
Therefore
\begin{align*}
\beta_{t+1} \leq \sup_{b\omega \leq \beta \leq a\sqrt{\omega}} h(\beta) \leq h(a\sqrt{\omega}) = 
a \sqrt{\omega}  + \omega^{3/2} \pth{ a - \frac{a^3}{C_1} + \frac{C_1}{2a}} \leq a \sqrt{\omega},
\end{align*}
provided that $\frac{a^3}{C_1} \geq 2a $ and $\frac{a^3}{C_1} \geq \frac{C_1}{a}$. Finally, choosing $a=2C_1$ and $b=C_1$, then the above conditions hold simultaneously 
as long as $\omega \leq c_0=c_0(r)$ for some small constant $c_0$.
\end{proof}

	\begin{proof}[Proof of \prettyref{thm:main-alpha-small}]

		It suffices to show \prettyref{eq:smalltheta} which, together with \prettyref{eq:beta-rate}, implies \prettyref{eq:theta-rate-smalltheta}.
	Combining \prettyref{eq:alpha+} with \prettyref{eq:F01} and \prettyref{eq:alpha-} with \prettyref{eq:F02}, we have		
	\begin{align}
\alpha_{t+1} \leq & ~ f(\alpha_t) + \Gamma |\alpha_t| \beta_t^2  + \omega (|\alpha_t|+\beta_t) \label{eq:alpha1} \\ 
\alpha_{t+1} \geq & ~ f(\alpha_t) - \Gamma |\alpha_t| \beta_t^2  - \omega (|\alpha_t|+\beta_t) \label{eq:alpha2}
\end{align}
with $\Gamma=1+s^2$.
Since $\thetanorm=s \leq s_0 \leq 1$, we have $\Gamma \leq 2$. 
Furthermore, in this case the constant $\kappa_2$ in \prettyref{eq:betanearzero} is also absolute.
Let $\alpha$ be any limiting point of $\{\alpha_t\}$. 
We show that $|\alpha| \leq 2 s_0$.
Assume for the sake of contradiction that $\alpha \geq 2 s_0$.
Sending $t\diverge$ in \prettyref{eq:alpha1} and in view of \prettyref{eq:betanearzero}, we have
	\begin{align}
\alpha
\leq & ~ f(\alpha) + C_3 (\alpha \omega  + \omega^{3/2}),
\label{eq:alpha3}
\end{align}
for some absolute constant $C_3$.
Let $q(\alpha) = \frac{f(\alpha)}{\alpha}$ be defined in \prettyref{eq:q} with $\theta_*$ replaced by $s$. As shown in  \prettyref{lmm:fprop}, $q$ is a decreasing function on $\reals_+$ with $q(s)=1$. 
Dividing both sides of \prettyref{eq:alpha3} by $\alpha$ leads to 
\begin{equation*}
1 \leq 
q(\alpha) + C_3 \pth{\omega + \frac{\omega^{3/2}}{\alpha}} 
\leq 
q(2 s_0) +  \frac{3 C_3}{2} \omega,
\end{equation*}
where the last inequality holds because of the assumption $s_0\geq \sqrt{\omega}$.
Furthermore, 
for all $\alpha \in[0,2]$, 
we have $q'(\alpha) \leq -C_4 \alpha$ for some absolute constant $C_4$.
Thus, $q(2 s_0) - 1 = \int_s^{2 s_0} q'(\alpha)d\alpha \leq - C_4 (4s_0^2- s^2) \leq -  3 C_4  s^2$.
Therefore we reach the desired contradiction that 
$q(2 s_0) + \frac{3C_3}{2} \omega \leq 1 -  C_4  s^2 + \frac{3C_3}{2} \omega < 1$, provided that $s^2  \geq \frac{3C_3}{2C_4}\omega$.
The proof is completed by taking $K= \max\{1, \sqrt{\frac{3C_3}{2C_4}}\}$.

For the other direction, if $\alpha < - 2 s_0$, then the above proof applies to \prettyref{eq:alpha2} with $\alpha$ replaced by $-\alpha$ and in view of the fact that $f(-\alpha)=-f(\alpha)$.	This completes the proof of \prettyref{eq:smalltheta}.

Finally, we show the second part for small initialization satisfying \prettyref{eq:theta0nearzero}.
We prove \prettyref{eq:smallthetaallt} by induction on $t$. The base case of $t=0$ follows from $\alpha_0 \leq \|\theta_0\| \leq \kappa_2 (\frac{d\log n}{n})^{1/4} \leq L K \sqrt{\omega} \leq L s_0$, 
provided that $L \geq \frac{\kappa_2}{K C_\omega^{1/4}}$, where both $\kappa_2$ and $C_\omega$ in \prettyref{eq:omega} are absolute constants since $\thetanorm \leq 1$ by assumption. 
Next, using \prettyref{eq:alpha1} and the argument that leads to \prettyref{eq:alpha3}, we have 
\[
\alpha_{t+1} \leq f(\alpha_t) + C_3 (\alpha_t \omega  + \omega^{3/2}).
\]
By the monotonicity of $f$, it suffices to show that $f(L s_0) + C_3 (L s_0 \omega  + \omega^{3/2}) \leq L s_0$.
To this end, recalling from \prettyref{eq:qp} and the fact that $q(0)=f'(0)=1+s^2 \leq 1+s_0^2$, we have
$q(\alpha) \leq 1+s_0^2 - C_4 \alpha^2/2$, where $C_4$ is absolute since $\thetanorm \leq 1$.
Thus
$f(\alpha)=\alpha q(\alpha) \leq \alpha (1+s_0^2) - C_4 \alpha^3/2$. Therefore using the assumption that $s_0 \geq K \sqrt{\omega}$, we have
$f(L s_0) + C_3 (L s_0 \omega  + \omega^{3/2}) = L s_0 + s_0(L - C_4 L^3/2 + C_3 (L/K^2+1/K^3)) \leq L s_0$, provided that $L$ exceeds some large absolute constant.
This completes the proof of \prettyref{eq:smallthetaallt}, which implies \prettyref{eq:theta-rate-smallthetaallt} in view of \prettyref{eq:betanearzero} provided that $L \geq \kappa_2$.
	\end{proof}

	\begin{proof}[Proof of \prettyref{thm:main-alpha-large}]
By assumption, $s \geq C_0 \sqrt{\omega}$.
	WLOG, we assume that $\alpha_0 \geq 0$ (otherwise we the same argument applies with $\alpha_t$ replaced by $-\alpha_t$ and $s$ by $-s$). 
		%
		%
	By design, $\alpha_t$ is close to zero at $t=0$. The argument entails proving that initially $\alpha_t$ increases geometrically approximately as 
	$\alpha_{t+1} = (1+\Omega(s^2)) \alpha_t$, until $\alpha_t$ exceeds $\Omega(\sqrt{\omega})$. After this point, the sandwich bound \prettyref{eq:alpha1}--\prettyref{eq:alpha2} behave as the linear perturbation of the one-dimensional EM iteration in \prettyref{eq:thetatsandwich}, and consequently the one-dimensional analysis in \prettyref{thm:main-1D} applies, yielding both error bound and speed of convergence.


By the assumption \prettyref{eq:ini0}, $\|\theta_0\| = c \pth{\frac{d}{n} \log n}^{1/4} \leq c' \sqrt{\omega}$ for some small constant $c'$ proportional to $c$.
	Since $c \leq \kappa_2$, \prettyref{eq:betanearzero} in 
	\prettyref{thm:main-beta} ensures that 
	$\beta_t \leq\kappa_2 \pth{\frac{d}{n} \log n}^{1/4}$ for \emph{all} $t\geq 0$.	Then
	\prettyref{eq:alpha1}--\prettyref{eq:alpha2} imply the following
		\begin{align}
\alpha_{t+1} \leq & ~ f(\alpha_t) + C_5 (\alpha_t \omega + \omega^{3/2}) \label{eq:alpha40} \\ 
\alpha_{t+1} \geq & ~ f(\alpha_t) - C_5 (\alpha_t \omega + \omega^{3/2}). \label{eq:alpha4}
\end{align}
	Let $C_*$ be a constant to be specified.
	Consider the following phases:

	\paragraph{Phase I: $\alpha_t \leq C_* \sqrt{\omega}$} 
	We will show that throughout Phase I, for some sufficiently large constant $C_4$,
		\begin{equation}
	\alpha_t \geq \frac{C_4}{s^2} \omega^{3/2}.
	\label{eq:alphaomega}
	\end{equation}
	In view of the choice of the initialization \prettyref{eq:ini0},  the assumption \prettyref{eq:ini0} ensures that
	\prettyref{eq:alphaomega} holds for the base case of $t=0$, 
	where $C_4$ is proportional to $\frac{\lambda_2}{c}$ and can be made sufficiently large.
	Assume \prettyref{eq:alphaomega} holds at time $t$. By \prettyref{lmm:fprop} and using \prettyref{eq:fppp}, the Taylor expansion of $f$ at $0$ gives
	$f(\alpha_t) \geq (1+s^2) \alpha_t - C_6 \alpha_t^3$.
	So \prettyref{eq:alpha4} implies
	\begin{equation}
	\alpha_{t+1} \geq (1+s^2) \alpha_t - C_6 \alpha_t^3 - C_5 (\alpha_t \omega + \omega^{3/2})
	\geq (1+s^2/4) \alpha_t
	\end{equation}
	where, since $s \geq C_0\sqrt{\omega}$ and $
	\frac{C_4}{s^2} \omega^{3/2}\leq \alpha_t \leq C_* \sqrt{\omega}$ by assumption, 
	the last inequality holds provided that 
	\begin{equation}
	C_0 \geq \frac{C_*}{\sqrt{4C_6}}, \quad C_0 \geq \sqrt{4C_5}, \quad C_4 \geq 4C_5.	
	\label{eq:CCC}
	\end{equation}
	Therefore \prettyref{eq:alphaomega} holds at time $t+1$. Furthermore, $\alpha_t$ grows exponentially and in $T_1 = O(\frac{1}{s^2} \log \frac{s}{\omega}) = O(\frac{\log n}{s^2})$ iterations enters the next phase.

\paragraph{Phase II: $\alpha_t \geq C_* \sqrt{\omega}$}
Then \prettyref{eq:alpha40}--\prettyref{eq:alpha4} imply
		\begin{align}
\alpha_{t+1} \leq & ~ f(\alpha_t) + C_5'  \omega\alpha_t \label{eq:alpha5} \\ 
\alpha_{t+1} \geq & ~ f(\alpha_t) - C_t'  \omega \alpha_t, \label{eq:alpha6}
\end{align}
where $C_5'=C_5(1+\frac{1}{C_*})$.
Comparing \prettyref{eq:alpha5}--\prettyref{eq:alpha6} with \prettyref{eq:thetatsandwich}, 
by replacing $w_n$ with $\omega$, $\theta_*$ with $s=\thetanorm$, and the initial value $\theta_0$ by $\alpha_{T_1} \geq C_*\sqrt{\omega}$, 
we see that \prettyref{thm:main-1D} applies to the convergence of $\{\alpha_t: t \geq T_1\}$. 
In particular, \prettyref{eq:sandwich-accuracy} and \prettyref{eq:sandwich-T} yield
\begin{equation}
|\alpha_t - s| \leq C_7 \min\sth{\frac{\omega}{s}, \sqrt{\omega}},
\label{eq:alpha-phase2}
\end{equation}
for all $t-T_1 \geq T_2 \triangleq \frac{C_8 }{s^2} \log \frac{ns}{\sqrt{\omega}} = O(\frac{1}{s^2} \log n)$.
This completes the proof of \prettyref{eq:alpha-rate}. 

\paragraph{Phase III: improved estimate on $\beta_t$}
Since $s \geq C_0 \sqrt{\omega}$ by assumption, from \prettyref{eq:alpha-phase2}, we conclude that for all $t \geq T_1+T_2$, we have $\alpha_t\in [s/2,2s]$. 
Recall that the prior unconditional analysis in \prettyref{thm:main-beta} treats $\alpha_t$ as zero (which is the worst case) and shows that $\beta_t = O(\sqrt{\omega})$.
Now that $\alpha_t = \Theta(s)$, we will use the $\alpha$-dependent bound \prettyref{eq:Gnew} to upgrade the error bound to $\beta_t = O(\frac{\omega}{s})$.
Continuing from \prettyref{eq:betaraw1}, for all $t\geq T_1+T_2$, we have
\begin{align}
		\beta_{t+1} 
		\leq & ~ \beta_t\pth{ 1 - \frac{\alpha_t^2+\beta_t^2}{2+8\Gamma^2}}
+ \omega(|\alpha_t| + \beta_t)	\nonumber \\
\stepa{\leq} & ~ \beta_t\pth{ 1 - \frac{s^2}{4(2+8\Gamma^2)}} + \omega(2s + \beta_t)	\nonumber\\
\stepb{\leq} & ~ \beta_t\pth{ 1 - C_9 s^2}  + 2 \omega s 
		\end{align}		
where (a) follows from $s/2\leq \alpha_t \leq 2s$ and (b) follows from the assumption $s \geq C_0\sqrt{\omega}$ for sufficiently large $C_0$, where $C_9$ is a constant depending only on $\Gamma$ (hence on $r$).
Thus $\beta_t \leq \frac{4\omega}{s}$ for all $t-(T_1+T_2) \geq T_3 \triangleq \frac{C_{10} }{s^2} \log \frac{s}{\omega} = O(\frac{1}{s^2} \log n)$.
This completes the proof of \prettyref{eq:theta-rate}. 
	\end{proof}

\subsection{Proof of \prettyref{lmm:FG}}
\begin{proof}
	Let $s=\|\theta_*\|$. 
Let $W = \Iprod{\xi}{Z}$ and $U = \Iprod{\eta}{Z}$, which are independent standard normals. Then
\begin{equation}
\Iprod{\theta}{Y} = \alpha\|\theta_*\| X+\alpha  U + \beta W = \alpha V +\beta W
\label{eq:thetaY}
\end{equation}
where $V \sim \frac{1}{2} N(\pm s,1)$ is independent of $W$.
\begin{enumerate}
	\item 
	The function $\alpha\mapsto\Expect[V \tanh (\alpha V+\beta W)]$ is because of the symmetry of the distribution of $W$.
	Furthermore,
\begin{align*}
		\fracp{F}{\alpha}	
= & ~ \expect{\frac{V^2}{\cosh^2(\alpha V+\beta W)}} \geq 0 \\
\frac{\partial^2 F}{\partial \alpha^2}	= & ~ 
\expect{ V^3 \tanh''(\alpha V+\beta W)} = 
\expect{ Z^3 \tanh''(\alpha Z+\beta Z) \cosh(s Z)}e^{-s^2/2},
\end{align*}
where the last equality follows from a change of measure (\prettyref{lmm:com}) with $Z \sim N(0,1)$ independent of $W$.
	Consider $\alpha\geq 0$.	By symmetry, $\expect{Z^3 \cosh(s Z)|\alpha Z+ \beta W=y}$ is an odd function which is nonnegative if and only if $y \geq 0$. 
	Since $\tanh''=-2\tanh \sech^2$, we have $\expect{Z^3 \cosh(s Z)|\alpha Z+ \beta W} \tanh''(\alpha Z+\beta W) \leq 0$ almost surely. Therefore $\alpha\mapsto F(\alpha,\beta)$ is concave on $\reals_+$, and convex on $\reals_-$ by symmetry.

\item This is simply because $F(\cdot,\beta)$ is an odd function and increasing on $\reals_+$.
	
	\item 
	Entirely analogously,
	\begin{align*}
		\fracp{G}{\beta}	
= & ~ \expect{\frac{W^2}{\cosh^2(\alpha V+\beta W)}} \geq 0 \\
\frac{\partial^2 G}{\partial \beta^2}	= & ~ -2\expect{\frac{W^3 \tanh(\alpha V+\beta W)}{\cosh^2(\alpha V+\beta W)}} \leq 0.
\end{align*}

\item
For $\alpha \geq 0$,
	\begin{align}
\fracp{F}{\beta}	=
\fracp{G}{\alpha}	
= & ~ \expect{WV\tanh'(\alpha V+ \beta W)} \nonumber \\
= &~ \beta \expect{V\tanh''(\alpha V+ \beta W)} \label{eq:Fb1}\\
= & ~ -2 \beta  \expect{\frac{V\tanh(\alpha V+ \beta W)}{\cosh^2(\alpha V+ \beta W)}} \nonumber\\
= & ~ -2 \beta \Expect\Bigg[\underbrace{\frac{\expect{V|\alpha V+ \beta W}\tanh(\alpha V+ \beta W)}{\cosh^2(\alpha V+ \beta W)}}_{\geq 0}\Bigg] \leq 0 \label{eq:Fb2},
\end{align}
where \prettyref{eq:Fb1} follows from Stein's lemma, and 
\prettyref{eq:Fb2} follows from the fact that, in view of \prettyref{lmm:estimator} and the symmetry of the distribution of $V$, $\hat V(y) \triangleq \expect{V|\alpha V+ \beta W=y}$ is an odd and increasing function such that $\hat V(y) \gtrless 0$ when $y \gtrless 0$. 

The case for $\alpha \leq 0$ follows the fact that $G(-\alpha,\beta)=G(\alpha,\beta)$ and $F(-\alpha,\beta)=-F(\alpha,\beta)$.

\item $|F(\alpha,\beta)| = |\Expect[V\tanh(\alpha V+\beta W)]| \leq \Expect[|V|] \leq \|\theta_*\| + \Expect|U|$, and similarly, 
$|G(\alpha,\beta)| \leq \Expect[|W|]$.

\item By the third property, $\alpha\mapsto G(\alpha,\beta)$ is maximized at $\alpha=0$.

\item We only prove \prettyref{eq:F01} for $\alpha \geq 0$; \prettyref{eq:F02} follows from the fact that $F(-\alpha,\beta)=-F(\alpha,\beta)$. 
The left inequality follows from \prettyref{eq:Fb2}. To show the right inequality, note that 
since $\expect{V|\alpha V+ \beta W}\tanh(\alpha V+ \beta W) \geq 0$ almost surely, using the fact that $\cosh(x) \geq 1$ and $\tanh(x) \lessgtr x$ for $x \gtrless 0$, we have
\begin{align*}
\expect{\frac{\expect{V|\alpha V+ \beta W}\tanh(\alpha V+ \beta W)}{\cosh^2(\alpha V+ \beta W)}}
\leq & ~  \expect{ \expect{V|\alpha V+ \beta W} (\alpha V+ \beta W)} \\
= & ~  \expect{ V(\alpha V+ \beta W)}\\
= & ~  \alpha \expect{V^2} = \alpha(1+\|\theta_*\|^2).
\end{align*}
Consequently, 
\[
\fracp{F}{\beta}	\geq -2 \beta \alpha(1+\|\theta_*\|^2).
\]
Integrating over $\beta$ yields the right inequality in \prettyref{eq:F01}.

\item 
	By symmetry, without loss of generality we assume $\alpha \geq 0$.
	By Stein's identity,
	\[
	G(\alpha,\beta) = \Expect[W \tanh(\alpha V+\beta W)] = \beta \Expect[\tanh'(\alpha V+\beta W)].
	\]
	Recall that $V=sX+U$, where $X$ is Rademacher and $U\sim N(0,1)$. 
		Let $T=\alpha (sX+U)+\beta W = \alpha s X + (\alpha U + \beta W)$.
	Then
	\[
	\frac{G(\alpha,\beta)}{\beta} = \expect{\frac{1}{\cosh^2(T)}}.
	\]
	Since
	$\Expect[X|T=t]=\tanh(\frac{\alpha s}{\alpha^2+\beta^2} t)$, we have
	\begin{align*}
	\fracp{}{s}\pth{\frac{G(\alpha,\beta)}{\beta}}
	= & ~ \alpha \Expect[X \tanh''(T)] = -2\alpha \expect{\frac{X\tanh(T)}{\cosh^2(T)}}\\
	= & ~ 	-2\alpha \Expect\Bigg[\underbrace{\frac{\tanh(\frac{\alpha s}{\alpha^2+\beta^2} T)\tanh(\alpha T)}{\cosh^2(T)}}_{\geq 0}\Bigg] \leq 0.
	\end{align*}
	Therefore $\frac{G(\alpha,\beta)}{\beta}$ is decreasing in $s$, and it suffices to consider $s=0$.
		Next we show for any $\sigma \geq 0$ and $Z\sim N(0,1)$,
	\begin{equation}
	\expect{\frac{1}{\cosh^2(\sigma Z)}} \leq 1 - \frac{\sigma^2}{2(1+2\sigma^2)},
	\label{eq:Gp0-goal}
	\end{equation}
	which applied to $\sigma^2=\alpha^2+\beta^2$ implies the desired result.
	
Using the inequality $\cosh(x) \geq 1+x^2/2$ and hence $\cosh^2(x) \geq 1+x^2$, we have\footnote{The last inequality in \prettyref{eq:mills} is due to the following integral representation of Mill's ratio \cite[3.466.1]{GR}:
\begin{equation}
\expect{\frac{1}{t^2+Z^2}} = \frac{\bar{\Phi}(t)}{t\varphi(t)}.
\label{eq:mills-int}
\end{equation}
To see this, let $f(t)=\Expect[\frac{t}{t^2+Z^2}] $. By Stein's identity, one can verify that $f$ satisfies the differential equation $f'(t) = tf(t)-1$.
Thus $g(t)=f(t)\varphi(t)$ satisfies $g'(t) = -\varphi(t)$, which implies that $g(t)=\bar\Phi(t)$ since $g(\infty)=0$.}
\begin{equation}
\expect{\frac{1}{\cosh^2(\sigma Z)}} \leq \expect{\frac{1}{1+\sigma^2 Z^2}} = \frac{\bar{\Phi}(1/\sigma)}{\sigma\varphi(1/\sigma)}.
\label{eq:mills}
\end{equation}
	Using \prettyref{lmm:ito}, we have
	\[
		\expect{\frac{1}{\cosh^2(\sigma Z)}} \leq 1 - \frac{2\sigma^2}{(\sqrt{1+2\sigma^2}+1)^2} \leq 1 - \frac{\sigma^2}{2(1+2\sigma^2)} 
	\]	
	This proves \prettyref{eq:Gp0-goal} and the desired \prettyref{eq:Gnew}.
\end{enumerate}	
\end{proof}

\section{Proofs in \prettyref{sec:refined}}
	\label{sec:pf-refined}

\subsection{Proof of Lemmas \ref{lmm:induction}, \ref{lmm:alphat1-raw} and \ref{lmm:ind1-raw}}
	\label{sec:pf-refined-mainlemmas}

	%

We start by defining a few typical events which will be used subsequently for several times.
\begin{lemma}
\label{lmm:typical}	
Define
\begin{align}
H_2 = & ~  \sth{\frac{1}{n} \sum_{i=1}^n Y_{i,1}^2 \geq 1 + \thetanorm^2 - \sqrt{\frac{\kappa \log n}{n}}} \label{eq:H2} \\
H_4 = & ~  \sth{\frac{1}{n} \sum_{i=1}^n Y_{i,1}^4 \leq \kappa } \label{eq:H4} \\
H_3 = &~ \sth{ \sum_{i=1}^n \|Y_i\|^3 \leq \kappa d^{3/2} } \label{eq:H3} \\
H_\infty = & ~  \sth{\max_{i\in[n]} |Y_{i,1}| \leq \sqrt{\kappa \log n}}. \label{eq:Hinfty} 
\end{align}
Then there exists some $\kappa = \kappa(\|\theta_*\|)$ such that $\prob{H_i} \geq 1-n^{-4}$ for $i=2,3,4,\infty$. 
\end{lemma}

Next we provide the supporting lemmas:

\begin{lemma}[Smoothness of the sample-EM map]
\label{lmm:fsmooth}	
	
	Let $f_n$ be defined in \prettyref{eq:fn}. Then 
	$f_n$ is $\opnorm{\Sigma_n}$-Lipschitz continuous on $\reals^d$, 
	where $\Sigma_n \triangleq \Expect_n[YY^\top]$ is the sample covariance matrix.
	In particular, with probability at least $1 - e^{-C'd\log n}$, 		
	\begin{equation}
	\opnorm{\Sigma_n} \leq 1+\|\theta_*\|^2 + \sqrt{\frac{C d}{n}},
	\label{eq:fsmooth}
	\end{equation}
	where the constants $C,C'$ depend only on $\thetaub$.
\end{lemma}

	%

\begin{lemma}
\label{lmm:concentration2}	
Assume that $n \geq d$. 
Let $Y_\perp = [Y_{1,\perp}, \ldots, Y_{1,\perp}]$.
Then
\begin{equation}
\prob{\opnorm{Y_\perp} \geq 4 \sqrt{n}} \leq e^{-n}.
\label{eq:ds}
\end{equation}
Furthermore, there exists some constant $C$ depending only on $\thetaub$, such that 
with probability at least $1-n^{-3}$,
	\begin{equation}
	 \frac{1}{n} \sum_{i=1}^n Y_{i,1}^2 |\iprod{Y_{i,\perp}}{\theta}|^2 \leq C\|\theta\|^2 \log n,
		\label{eq:concentration2}
	\end{equation}
	for all $\theta \in \reals^{d-1}$.
\end{lemma}

\begin{lemma}
\label{lmm:rad}	
	Let $b=(b_1,\ldots,b_n)$ consist of independent Rademacher random variables and let $x=(x_1,\ldots,x_n)$ be independent of $b$. 
	Then for any $a,t>0$,
	\begin{equation}
	\prob{\frac{1}{n}\left|\sum_{i=1}^n x_i b_i \right| \geq \sqrt{\frac{a s}{n}} }
		\leq 2\exp(-s/8)  + \prob{\frac{1}{n} \sum_{i=1}^n x_i^2 \geq a}.
	\label{eq:rad}
	\end{equation}
	Furthermore, given a finite collection $\{x^{\theta}:\theta\in\Theta\}$ independent of $b$, 
	\begin{equation}
	\prob{\sup_{\theta \in \Theta} \frac{1}{n} \left|\sum_{i=1}^n x_i b_i \right| \geq \sqrt{\frac{a s}{n}} }
		\leq 2\exp(-s/8)  |\Theta| + \prob{\frac{1}{n} \sup_{\theta \in \Theta} \sum_{i=1}^n x_i^2 \geq a}.
	\label{eq:rad2}
	\end{equation}
\end{lemma}

\begin{lemma}
\label{lmm:K}	
Assume that $n \geq C d$ for some absolute constant $C$.
Let $q:\reals\to\reals$ be a function with bounded first two derivatives, such that 
\begin{equation}
\max\sth{\norm{q'}_\infty, \norm{q''}_\infty }  \leq L_0,
\label{eq:L0}
\end{equation}
for some constant $L_0$. 
Define a (random) function $D:\reals^d\to\reals$  by
\begin{equation}
D(\theta)
\triangleq \frac{1}{n} \sum_{i=1}^n Y_{i,1} b_i  q(\Iprod{\theta}{Y_{i}}).
\label{eq:D}
\end{equation}
where $\{b_i\}$ are independent Rademacher variables and independent of $\{Y_i\}$. 
Let $R>0$.
Then there exists a constant $L_1$ depending only on $L_0$, $\thetaub$ and $R$, such that with probability at least $1-10 n^{-1}$, 
$D$ is $\sqrt{\frac{L_1  d \log^2 n }{n}}$-Lipschitz on the ball $B_R = \{\theta\in\reals^{d}: \|\theta\| \leq R\}$.
\end{lemma}

\begin{lemma}
\label{lmm:M}	
	For $\theta = (\theta_1,\theta_\perp) \in \reals^d$, define 
	\begin{equation}
	M(\theta) \triangleq \frac{1}{n} \sum_{i=1}^n b_i Y_{i,\perp} Y_{i,1} Q(\theta_1 Y_{i,1}, \Iprod{\theta_\perp}{Y_{i,\perp}} ).
	\label{eq:Mtheta}
	\end{equation}
	where $Q:\reals^2\to\reals$ satisfies $\max\{\norm{Q}_\infty, \norm{\partial_x Q}_\infty, \norm{\partial_y Q}_\infty\} \leq L_0$ for some constants $L_0$.
	Let $R>0$.
	Then there exist constant $L_1$ 	depending only on $L_0$, $\thetaub$ and $R$, such that with probability at least $1-10 n^{-1}$, 
	\begin{equation}
	\sup_{\|\theta\|\leq R} \|M(\theta)\|  \leq \sqrt{\frac{L_1 d \log^2 n}{n}}.
	\label{eq:M}
	\end{equation}
\end{lemma}

We now prove the main lemmas:

\begin{proof}[Proof of \prettyref{lmm:induction}]
By the definition in \prettyref{eq:phase1}, we have $\Tstar = O(\sqrt{n} \log n)$. 
By the union bound, with probability at least $1-O(\Tstar n^{-1})=1-O(n^{-1/2} \log n)$, \prettyref{eq:alphat1-raw} and \prettyref{eq:ind1-raw} hold for all $t \leq \Tstar$. On this event, we proceed by induction on $t$. 

For the base case of $t=0$, \prettyref{eq:ind1} is trivially true, and \prettyref{eq:ind4-pre}--\prettyref{eq:ind4} hold by virtue of the random initialization on the event \prettyref{eq:ini}.

Next, assume that \prettyref{eq:ind1} and \prettyref{eq:ind4-pre} hold at time $t$. 
In particular, thanks to the assumption \prettyref{eq:thetastar-lb} and \prettyref{eq:ndlogd}, we have
\begin{equation}
\Tstar \omega \asymp \Tstar \sqrt{\frac{d \log n}{n}} \lesssim  \sqrt{\frac{d \log n (\log d + \log\log n)^2}{n \thetanorm^4}} \lesssim 1.
\label{eq:Tstaromega}
\end{equation}
Thus, \prettyref{eq:ind1} implies that
\begin{equation}
\|\theta_t - \ttheta_t\| \leq \alpha_t C_3.
\label{eq:ind1-weak}
\end{equation}
Similarly, by \prettyref{eq:Tstaromega},
\prettyref{eq:ind4-pre} implies that 
$\alpha_t \geq  \frac{1}{\sqrt{d \log n} + C_2 } \beta_t$, which further implies the desired \prettyref{eq:ind4}, since $\|\theta_t\|^2 =\alpha_t^2+\beta_t^2$.


To show that \prettyref{eq:ind4-pre} holds at time $t+1$, by \prettyref{eq:alphat1-raw} in \prettyref{lmm:alphat1-raw}, we have
\begin{align}
\alpha_{t+1} 
\geq & ~  \alpha_t \pth{1+\|\theta_*\|^2  - \sqrt{\frac{C \log n}{n}} - C \|\theta_t\|^2} - \sqrt{\frac{C \log^2 n}{n}} \|\theta_t\| - \sqrt{\frac{C d \log^2 n}{n}} \|\theta_t-\ttheta_t\| \nonumber \\
\geq & ~ \alpha_t \pth{1+\|\theta_*\|^2 - C_4 \sqrt{\frac{d \log^3 n}{n}} },
\label{eq:alphagrow}
\end{align}
where the last step follows from 
\prettyref{eq:thetatnorm}, \prettyref{eq:ind1-weak}, and \prettyref{eq:ind4}.
Combined with \prettyref{eq:beta-uncond2}, we have
\[
\frac{\beta_{t+1}}{\alpha_{t+1}}
\leq \frac{\beta_{t}}{\alpha_{t}}  \frac{1+\omega}{1+\|\theta_*\|^2 - C_4 \sqrt{\frac{d \log^3 n}{n}}} + \frac{\omega}{1+\|\theta_*\|^2 - C_4 \sqrt{\frac{d \log^3 n}{n}}}
\leq 
\frac{\beta_{t}}{\alpha_{t}} + \omega
\]
where the last step follows from the assumption \prettyref{eq:thetastar-lb} with the constant $C_\star$ chosen to be sufficiently large.
Thus, the ratio $\frac{\beta_t}{\alpha_t}$ grows at most linearly and satisfies 
$\frac{\beta_{t}}{\alpha_{t}} \leq \frac{\beta_0}{\alpha_0} + \omega t \leq \sqrt{d \log n} + \omega t$, on the event \prettyref{eq:ini}.
This is the desired \prettyref{eq:ind4-pre}.

It remains to show \prettyref{eq:ind1} holds at time $t+1$.
To this end, we write abstractly
\begin{equation}
\|\theta_t - \ttheta_t\| \leq \alpha_t K_t
\label{eq:thetatKt}
\end{equation}
and we will show that
\begin{equation}
K_t \leq C_5 \sth{\pth{1+ \sqrt{\frac{C_5 d \log^3 n}{n}}}^t -1},
\label{eq:Kt}
\end{equation}
which, in view of \prettyref{eq:Tstaromega},
 implies the desired 
\begin{equation}
K_t \leq C_5' \sqrt{\frac{d \log^3 n}{n}} t
\end{equation}
 for all $t\leq \Tstar$.
%
%

Next we apply the induction hypothesis to \prettyref{eq:ind1-raw} in \prettyref{lmm:ind1-raw}:
\begin{align*}
\|\ttheta_{t+1}-\theta_{t+1}\|
\leq  & ~ \pth{1 + \|\theta_*\|^2+ \sqrt{\frac{C d\log^2 n}{n}}} \|\ttheta_{t}-\theta_{t}\| + \sqrt{\frac{C d\log^2 n}{n}} \alpha_t + \sqrt{\frac{C \log n}{n}} \|\theta_t\|  \\
\stepa{\leq}  & ~ \alpha_t \sth{ K_t \pth{1 + \|\theta_*\|^2+ \sqrt{\frac{C d\log^2 n}{n}}} + \sqrt{\frac{C d\log^2 n}{n}} } + \sqrt{\frac{C \log n}{n}} \|\theta_t\|  \\
\stepb{\leq}  & ~ \alpha_t \sth{ K_t \pth{1 + \|\theta_*\|^2+ \sqrt{\frac{C d\log^2 n}{n}}} + \sqrt{\frac{C_6 d\log^2 n}{n}} }  \\
\stepc{\leq}  & ~ \alpha_{t+1} \frac{ K_t \pth{1 + \|\theta_*\|^2+ \sqrt{\frac{C d\log^2 n}{n}}} + \sqrt{\frac{C_6 d\log^2 n}{n}} }{
1+\|\theta_*\|^2 - C_4 \sqrt{\frac{d \log^3 n}{n}}}  \\
\stepd{\leq}  & ~ \alpha_{t+1} K_{t+1},
\end{align*}
where
(a) follows from \prettyref{eq:thetatKt};
(b) follows from \prettyref{eq:ind4};
(c) follows from \prettyref{eq:alphagrow};
(d) follows from \prettyref{lmm:Kt}.
This proves \prettyref{eq:Kt}, i.e., the desired \prettyref{eq:ind1}, at time $t+1$.
\end{proof}


\begin{proof}[Proof of \prettyref{lmm:alphat1-raw}]
First of all, in view of \prettyref{eq:ab-bounded} and \prettyref{eq:eventd}, with probability at least $1-2\exp(-2c_0 d \log n)$, both the main and the auxiliary sequences are bounded, i.e., 
\begin{equation}
\sup_{t \geq 0} \|\theta_t\| \leq 4(r+1), \quad \sup_{t \geq 0} \|\ttheta_t\| \leq 4(r+1).
\label{eq:thetat-bdd}
\end{equation}

Write
\[
f_n(\theta_t) = \Expect_n[YY^\top] \theta_t + \Expect_n[Y (\tanh\Iprod{\theta_t}{Y} - \Iprod{\theta_t}{Y})].
\]
Then
\[
\alpha_{t+1} = \underbrace{\Expect_n[Y_1 \Iprod{Y}{\theta_t }]}_{R_1} - \underbrace{\Expect_n[Y_1 (\Iprod{Y}{\theta_t}-\tanh\Iprod{\theta_t}{Y})]}_{R_2}.
\]
We first show that with probability at least $1-O(n^{-1})$,
\begin{equation}
R_1 \geq \pth{1+\|\theta_*\|^2 - \sqrt{\frac{C \log n}{n}} } \alpha_t - \sqrt{\frac{C \log^2 n}{n}}  \|\theta_{t,\perp}\| - 
 \sqrt{\frac{C d \log^2 n}{n}}  \|\ttheta_t-\theta_t\|
\label{eq:R1}
\end{equation}
and
\begin{equation}
|R_2| \leq C \alpha_t \|\theta_t\|^2 + \sqrt{\frac{C \log^2 n}{n}} \|\theta_{t,\perp}\| + \sqrt{\frac{C d \log^2 n}{n}}  \|\theta_t-\ttheta_t\|.
\label{eq:R2}
\end{equation}
Then the desired \prettyref{eq:alphat1-raw} follows from \prettyref{eq:R1} and \prettyref{eq:R2}.

For the linear term $R_1$, we have
\begin{align}
R_1
= & ~ \Expect_n[Y_1 \Iprod{Y}{\theta_t }] \nonumber \\
= & ~ \frac{1}{n} \sum_{i=1}^n b_i Y_{i1} (\alpha_t b_i Y_{i1} + \Iprod{Y_{i\perp}}{\theta_{t,\perp} })	\nonumber \\
= & ~ \pth{\frac{1}{n} \sum_{i=1}^n Y_{i1}^2 } \alpha_t +  \frac{1}{n} \sum_{i=1}^n b_i Y_{i1} \Iprod{Y_{i\perp}}{\theta_{t,\perp} }. \label{eq:R1a}
\end{align}
Here the first term (signal) satisfies $\frac{1}{n} \sum_{i=1}^n Y_{i1}^2 \geq 1+\|\theta_*\|^2 - O(\sqrt{\frac{\log n}{n}})$, in view of \prettyref{eq:H2}.   
For the second term, we cannot afford to take union bound over the $d$-dimensional sphere. Instead, we resort to the auxiliary iterates $\{\ttheta_t\}$.
Write
\begin{equation}
\frac{1}{n} \sum_{i=1}^n b_i Y_{i1} \Iprod{Y_{i\perp}}{\theta_{t,\perp} }
= \frac{1}{n} \sum_{i=1}^n b_i Y_{i1} \Iprod{Y_{i\perp}}{\ttheta_{t,\perp} }
+ \frac{1}{n} \sum_{i=1}^n b_i Y_{i1} \Iprod{Y_{i\perp}}{\theta_{t,\perp} - \ttheta_{t,\perp}  }.
\label{eq:R1b}
\end{equation}
Using the independence between $(\ttheta_t, \{Y_{i,1}\})$ and $\{b_i\}$, we have, for some constants $C,C'$,
\begin{align}
& ~ \prob{\left|\frac{1}{n} \sum_{i=1}^n b_i Y_{i1} \Iprod{Y_{i\perp}}{\ttheta_{t,\perp} }\right| \geq 
\sqrt{\frac{C \log^2 n }{n}} \|\ttheta_{t,\perp}\| } \nonumber \\
\stepa{\leq} & ~ 2 n^{-1}  + \prob{\frac{1}{n} \sum_{i=1}^n Y_{i1}^2 \Iprod{Y_{i\perp}}{\ttheta_{t,\perp} }^2  \geq C' \log n \|\ttheta_{t,\perp}\|^2} 	\nonumber \\
\stepb{\leq} & ~ 3 n^{-1}, 	\label{eq:R1c}
\end{align}
where 
(a) follows from \prettyref{lmm:rad}; (b) follows from \prettyref{lmm:concentration2}.
Furthermore, on the event \prettyref{eq:thetat-bdd}, applying \prettyref{lmm:K} to $q$ being the identity function, we conclude that, with probability at least $1-O(n^{-1})$, 
\begin{equation}
\left|\frac{1}{n} \sum_{i=1}^n b_i Y_{i1} \Iprod{Y_{i\perp}}{\theta_{t,\perp} - \ttheta_{t,\perp}  } \right|
\leq \sqrt{\frac{C d \log^2 n}{n}} \|\theta_t - \ttheta_t\|. 
\label{eq:R1d}
\end{equation}
Combining \prettyref{eq:R1a}--\prettyref{eq:R1d} and using the triangle inequality yield \prettyref{eq:R1}.


For the nonlinear term $R_2$, define
\begin{align}
g(x) \triangleq & ~ x - \tanh(x)	\label{eq:g}\\
T(x,y) \triangleq  & ~ \frac{1}{2} (g(y+x)+g(y-x))	\label{eq:T}\\
H(x,y) \triangleq  & ~ \frac{1}{2} (g(y+x)-g(y-x)).	\label{eq:H}
\end{align}
Then for any $x,y$ and any $b\in\{\pm 1\}$, we have
\begin{equation}
g(y + b x) = T(x,y) + b H(x,y).
\label{eq:gb}
\end{equation}
Furthermore, we have
\begin{lemma}
\label{lmm:gg}	
		For any $x,y\in\reals$,
	\begin{equation}
	0 \leq y \cdot T(x,y) \leq x^2 y^2 + y^4
	\label{eq:gg}
	\end{equation}
	and
	\begin{equation}
	|H(x,y)| \leq |x| .
	\label{eq:hh}
	\end{equation}
\end{lemma}
Then
\begin{align*}
R_2
= & ~ \Expect_n[Y_1 g(\Iprod{Y}{\theta_t})] \\
= & ~ \frac{1}{n} \sum_{i=1}^n b_i Y_{i,1} g(\Iprod{Y_i}{\theta_t}) = \frac{1}{n} \sum_{i=1}^n b_i Y_{i,1} g(b_i \alpha_t Y_{i,1}  + \Iprod{Y_{i,\perp}}{\theta_{t,\perp}}) \\
\stepa{=} & ~ \frac{1}{n} \sum_{i=1}^n Y_{i,1} g(\alpha_t Y_{i,1}  + b_i \Iprod{Y_{i,\perp}}{\theta_{t,\perp}}) \\
\stepb{=} & ~ \underbrace{\frac{1}{n} \sum_{i=1}^n T(\Iprod{Y_{i,\perp}}{\theta_{t,\perp}},\alpha_t Y_{i,1}) Y_{i,1}  }_{R_3}
+ \underbrace{\frac{1}{n} \sum_{i=1}^n H(\Iprod{Y_{i,\perp}}{\theta_{t,\perp}},\alpha_t Y_{i,1})  Y_{i,1} b_i}_{R_4},
\end{align*}
where (a) is due to $g(\pm x) = \pm g(x)$;
(b) follows from \prettyref{eq:gb}.
Next we show \prettyref{eq:R2} by proving that, with probability at least $1-O(n^{-1})$,
\begin{align}
|R_3| \leq  & ~ C \alpha_t \|\theta_{t}\|^2  \label{eq:Ibound}  \\
|R_4| \leq & ~ \sqrt{\frac{C \log^2 n}{n}} \|\theta_{t,\perp}\| + \sqrt{\frac{C d \log^2 n}{n}} \|\theta_t-\ttheta_t\|.  \label{eq:IIbound}
\end{align}
To prove \prettyref{eq:Ibound}, recall that $\alpha_t > 0$ by assumption.
Then with probability at least $1-O(n^{-1})$,
\begin{align*}
0 \stepa{\leq} R_3
= & ~ \frac{1}{n} \sum_{i=1}^n T(\Iprod{Y_{i,\perp}}{\theta_{t,\perp}}, \alpha_t Y_{i,1}) Y_{i,1}  \\
\stepb{\leq} & ~ \alpha_t  \pth{\frac{1}{n} \sum_{i=1}^n Y_{i,1}^2 \Iprod{Y_{i,\perp}}{\theta_{t,\perp}}^2}  +  \alpha_t^3 \pth{\frac{1}{n} \sum_{i=1}^n  Y_{i,1}^4}   \\
\stepc{\leq} & ~ C\alpha_t \|\theta_{t,\perp}\|^2  +  C \alpha_t^3 \\
\stepd{=} & ~ C \alpha_t \|\theta_t\|^2,
\end{align*}
where (a) and (b) follow from \prettyref{eq:gg} in \prettyref{lmm:gg}; (c) follows from \prettyref{lmm:concentration2} and \prettyref{eq:H4}; (d) is due to $\|\theta_{t,\perp}\|^2 + |\theta_{t,1}|^2 = \|\theta_t\|^2$.
This completes the proof of \prettyref{eq:Ibound}.

To show \prettyref{eq:IIbound}, we will again use the auxiliary iterates $\{\ttheta_t\}$.
For any $\theta=(\theta_1,\theta_\perp)\in\reals^d$, define
\begin{equation}
\xi(\theta)
\triangleq \frac{1}{n} \sum_{i=1}^n H(\Iprod{Y_{i,\perp}}{\theta_{\perp}},\theta_1 Y_{i,1})  Y_{i,1} b_i.
\label{eq:xi}
\end{equation}
Then
\begin{equation}
R_4 = \xi(\theta_{t}) = \xi(\ttheta_{t})  + \xi(\theta_{t}) - \xi(\ttheta_{t}),
\label{eq:R4pre}
\end{equation}
Define
\begin{equation}
\theta_t' \triangleq  (-\theta_{t,1},\theta_{t,\perp}), \quad \ttheta_t' \triangleq  (-\ttheta_{t,1},\ttheta_{t,\perp}),
\label{eq:thetaminus}
\end{equation}
which satisfies $\|\theta_t'-\theta_t'\|=\|\theta_t-\ttheta_t\|$.
Then
\[
\xi(\theta_{t}) - \xi(\ttheta_{t})
= 
\frac{1}{2n} \sum_{i=1}^n Y_{i,1} b_i \sth{ g(\Iprod{\ttheta_t}{Y_{i}}) - g(\Iprod{\theta_t}{Y_{i}})} -
\frac{1}{2n} \sum_{i=1}^n Y_{i,1} b_i \sth{ g(\Iprod{\ttheta_t'}{Y_{i}}) - g(\Iprod{\theta_t'}{Y_{i}})}.
\]
On the event \prettyref{eq:thetat-bdd}, applying \prettyref{lmm:K} to $q = g$ whose first two derivatives are bounded by absolute constants, we conclude that, with probability at least $1-O(n^{-1})$, 
\begin{equation}
|\xi(\theta_t) - \xi(\ttheta_t)| 
\leq \sqrt{\frac{C d \log^2 n}{n}} (\|\theta_t - \ttheta_t\| + \|\utheta_t - \ttheta_t^-\|)
= 2 \sqrt{\frac{C d \log^2 n}{n}} \|\theta_t - \ttheta_t\|.
\label{eq:R41}
\end{equation}
To bound $\xi(\ttheta_t)$, let $\tx_i \triangleq H( \Iprod{Y_{i,\perp}}{\ttheta_{t,\perp}}, \talpha_t Y_{i,1})  Y_{i,1} $, which are independent of $\{b_i\}$. Then
\begin{align*}
\prob{|\xi(\ttheta_t)| \geq \sqrt{\frac{C s \log n}{n}} \|\ttheta_{t,\perp}\|}
=  & ~ \prob{\frac{1}{n} \left|\sum_{i=1}^n \tx_i b_i\right| \geq \sqrt{\frac{C s \log n}{n}} \|\ttheta_{t,\perp}\|  }  \\
\stepa{\leq} & ~ 2 \exp(-s/8)  + \prob{\frac{1}{n} \sum_{i=1}^n \tx_i^2 \geq C \log n \|\ttheta_{t,\perp}\|^2} \\
\stepb{\leq} & ~ 2 \exp(-s/8)  + \prob{\frac{1}{n} \sum_{i=1}^n Y_{i,1}^2 \Iprod{Y_{i,\perp}}{\ttheta_{t,\perp}}^2 \geq C \log n \|\ttheta_{t,\perp}\|^2 } \\
\stepc{\leq} & ~ 2 \exp(-s/8)  + n^{-3},
\end{align*}
where (a) follows from \prettyref{lmm:rad}; (b) is due to \prettyref{eq:hh} in \prettyref{lmm:gg};
(c) is due to \prettyref{lmm:concentration2}.
Setting $s = 8 \log n$ yields
 with probability at least $1-O(n^{-1})$, 
\begin{equation}
|\xi(\ttheta_t)| 
\leq \sqrt{\frac{C \log^2 n}{n}} \|\ttheta_{t,\perp}\| \leq 
\sqrt{\frac{C \log^2 n}{n}} (\|\theta_{t,\perp}\| +  \|\theta_t - \ttheta_t\|).
\label{eq:R42}
\end{equation}
Combining \prettyref{eq:R4pre} with \prettyref{eq:R41} and \prettyref{eq:R42} completes the proof of \prettyref{eq:IIbound} and hence the lemma.	
\end{proof}

\begin{proof}[Proof of \prettyref{lmm:ind1-raw}]
Write
\[
\ttheta_{t+1} - \theta_{t+1} 
= \underbrace{f_n(\ttheta_{t}) - f_n(\theta_{t})}_{\triangleq \calE_1} + \underbrace{\tilde f_n(\ttheta_{t}) - f_n(\ttheta_{t})}_{\triangleq \calE_2}.
\]
For the first term, applying \prettyref{lmm:fsmooth} yields that with probability at least $1-\exp(-C'd\log n)$,
\begin{equation}
\|\calE_1\|=\| f_n(\ttheta_{t}) - f_n(\theta_{t})\| \leq 
\pth{1 + \|\theta_*\|^2 + \sqrt{\frac{C d}{n}}} \|\ttheta_{t} - \theta_{t}\|.
\label{eq:calE1}
\end{equation}

	%

Next we proceed to the second term. 
A trivial yet useful lemma is the following:
\begin{lemma}
\label{lmm:decomp}	
Assume that $b_i,\tb_i \in \{\pm1\}$. Then
\[
	\frac{1}{n} \sum_{i=1}^n h( y_i + \tb_i x_i ) - h(y_i + b_i x_i) 
	= 	\frac{1}{n} \sum_{i=1}^n (\tilde b_i - b_i) B(x_i,y_i) 	
	\]
	where  $B(x,y) \triangleq \frac{h(y+x)-h(y-x)}{2}$.	
\end{lemma}
\begin{proof}
	This simply follows from the fact that whenever $b=\pm 1$, we can write $h(x + b y)  = s + b \delta$, where $s \triangleq \frac{h(x+y)+h(x-y)}{2}$ and $\delta=\frac{h(x+y)-h(x-y)}{2}$.
\end{proof}
To bound the orthogonal component of $\calE_2$, note that $\tY_{i,\perp}=Y_{i,\perp}$. 
To apply \prettyref{lmm:decomp} with $h=\tanh$, we define
\begin{align}
B(x,y) \triangleq & ~ \frac{\tanh(y+x) - \tanh(y-x)}{2}  \label{eq:B}\\
Q(x,y) \triangleq & ~ \frac{B(x,y)}{x},	 \label{eq:Q}
\end{align}
with $Q(0,y)$ understood as $ \lim_{x\to0}Q(x,y) = \sech^2(y)$.
The function 	$Q$ satisfies the following smoothness property:
\begin{lemma}
\label{lmm:Q}	
Then	for all $x,y\in\reals$,
$|Q(x,y)| \leq 1$, 
$|\partial_x Q(x,y)| \leq 1/3$, 
$|\partial_y Q(x,y)| \leq 1$.
\end{lemma}

In view of \prettyref{eq:gb}, we have
\begin{align*}
\calE_{2,\perp}
= & ~ \frac{1}{n}\sum_{i=1}^n Y_{i,\perp} \tanh \Iprod{\ttheta_t}{\tilde Y_i} - \frac{1}{n}\sum_{i=1}^n Y_{i,\perp} \tanh \Iprod{\ttheta_t}{Y_i} \\
= & ~ \frac{1}{n}\sum_{i=1}^n Y_{i,\perp} (\tanh (\Iprod{\ttheta_{t,\perp}}{Y_{i,\perp}} + \tb_i \ttheta_{t,1} Y_{i,1} ) - \tanh (\Iprod{\ttheta_{t,\perp}}{Y_{i,\perp}} + b_i \ttheta_{t,1} Y_{i,1})) \\
= & ~  \frac{1}{n} \sum_{i=1}^n (\tb_i-b_i) Y_{i,\perp}  B(\ttheta_{t,1} Y_{i,1}, \Iprod{\ttheta_{t,\perp}}{Y_{i,\perp}}) \\
= & ~  \ttheta_{t,1} \sth{\frac{1}{n} \sum_{i=1}^n (\tb_i-b_i) Y_{i,1}  Y_{i,\perp} Q(\ttheta_{t,1} Y_{i,1}, \Iprod{\ttheta_{t,\perp}}{Y_{i,\perp}}) }
\end{align*}
where the penultimate step follows from applying \prettyref{lmm:decomp} to $h=\tanh$.
To apply \prettyref{lmm:M}, first note that the function $Q$ defined in \prettyref{eq:Q} fulfills the bounded derivative condition thanks to \prettyref{lmm:Q}.
Thus with probability at least $1-O(n^{-1})$,
it holds that
\[
\norm{\frac{1}{n} \sum_{i=1}^n (\tb_i-b_i) Y_{i,\perp} Y_{i,1}  Q(\ttheta_{t,1} Y_{i,1}, \Iprod{\ttheta_{t,\perp}}{Y_{i,\perp}}) } \leq \sqrt{\frac{C d \log^2 n}{n}}
\]
and hence
\begin{equation}
\|\calE_{2,\perp}\| \leq |\ttheta_{t,1}| \sqrt{\frac{C d \log^2 n}{n}} 
\leq (\alpha_t + \|\ttheta_t-\theta_t\|) \sqrt{\frac{C d \log^2 n}{n}}.
\label{eq:R2perp}
\end{equation}

To bound the first coordinate of $\calE_2$, let 
$\tx_i = \ttheta_{t,1} Y_{i,1} $, $\ty_i = \Iprod{\ttheta_{t,\perp}}{Y_{i,\perp}}$ and similarly
$x_i = \theta_{t,1} Y_{i,1} $, $y_i = \Iprod{\theta_{t,\perp}}{Y_{i,\perp}}$. Then
\begin{align*}
\calE_{2,1}
= & ~ \frac{1}{n} \sum_{i=1}^n  b_i Y_{i,1} \tanh (\ty_i + b_i \tx_i)  - \tb_i Y_{i,1} \tanh (\ty_i + \tb_i \tx_i)\\
= & ~ \frac{1}{n} \sum_{i=1}^n  Y_{i,1} \sth{ \tanh (\tx_i+ b_i \ty_i)  - \tanh (\tx_i + \tb_i \ty_i) }\\
= & ~  \frac{1}{n} \sum_{i=1}^n (\tb_i-b_i) Y_{i,1}  B(\ty_i,\tx_i) \\
= & ~  \underbrace{\frac{1}{n} \sum_{i=1}^n \tb_i Y_{i,1} B(y_i,x_i)}_{\calE_3} 
- \underbrace{\frac{1}{n} \sum_{i=1}^n b_i Y_{i,1}  B(\ty_i,\tx_i)}_{\calE_4}  + \underbrace{\frac{1}{n} \sum_{i=1}^n \tb_i Y_{i,1} \sth{
B(\ty_i,\tx_i) - 
B(y_i,x_i)}}_{\calE_5}.
\end{align*}
The first two terms can be dealt with using the same technology:
For $\calE_3$, we have
\begin{align}
\prob{|\calE_3| \geq 4 \|\theta_{t,\perp}\| \sqrt{\frac{s}{n}} }
\stepa{\leq} & ~ 2\exp(-s/8)  + \prob{\frac{1}{n} \sum_{i=1}^n B(y_i,x_i)^2 \geq 16 \|\theta_{t,\perp}\|^2}	\nonumber \\
\stepb{=} & ~ 2\exp(-s/8)  + \prob{\frac{1}{n} \sum_{i=1}^n \Iprod{\theta_{t,\perp}}{Y_{i,\perp}}^2 \geq 16 \|\theta_{t,\perp}\|^2}.	\nonumber \\	
\stepc{\leq} & ~ 2\exp(-s/8)  + \exp(-n), \label{eq:ind1I}
\end{align}
where (a) follows from \prettyref{lmm:rad};
(b) follows from the fact that $|B(y,x)| = \frac{|\tanh(x+y) - \tanh(x-y)|}{2} \leq |y|$, since $\tanh$ is $1$-Lipschitz;
(c) follows from \prettyref{eq:ds} in \prettyref{lmm:concentration2}. 
Choosing $s=8\log n$ yields
\begin{equation}
|\calE_3| \leq \|\theta_t\| \sqrt{\frac{C \log n}{n}}
\label{eq:calE3}
\end{equation}
with probability at least $1-O(n^{-1})$.

Entirely analogously,  we have
\begin{align}
\prob{|\calE_4| \geq 4 \|\ttheta_{t,\perp}\| \sqrt{\frac{s}{n}} }
\leq 2\exp(-s/8)  + \exp(-n), \label{eq:ind1II}
\end{align}
Choosing $s=8\log n$ yields
\begin{equation}
|\calE_4| \leq (\|\theta_t\|+\|\ttheta_t-\theta_t\|) \sqrt{\frac{C \log n}{n}}
\label{eq:calE4}
\end{equation}
with probability at least $1-O(n^{-1})$.

To bound $\calE_5$, 
recall from \prettyref{eq:thetaminus} the notations
$\theta_t' = (-\theta_{t,1},\theta_{t,\perp})$ and 
$\ttheta_t' = (-\ttheta_{t,1},\ttheta_{t,\perp})$, which satisfies 
$\|\theta_t'-\ttheta_t'\|=\|\theta_t-\ttheta_t\|$.
Then we have
\begin{equation}
\calE_5 = 
\frac{1}{2n} \sum_{i=1}^n \tb_i Y_{i,1}  (\tanh \Iprod{\theta_t}{Y_i} - \tanh \Iprod{\ttheta_t}{Y_i}) +
\frac{1}{2n} \sum_{i=1}^n \tb_i Y_{i,1}  (\tanh \Iprod{\theta_t'}{Y_i} - \tanh \Iprod{\ttheta_t'}{Y_i}).
\label{eq:calE5}
\end{equation}
By \prettyref{lmm:K} (applied to $q=\tanh$), the first term satisfies, with probability at least $1-O(n^{-1})$, 
\begin{equation}
\left|\frac{1}{2n} \sum_{i=1}^n \tb_i Y_{i,1}  (\tanh \Iprod{\theta_t}{Y_i} - \tanh \Iprod{\ttheta_t}{Y_i})\right|
\leq \sqrt{\frac{C_1  d \log^2 n }{n}} \|\theta_t-\ttheta_t\|.
\label{eq:calE5-bound}
\end{equation}
Entirely analogously, the second term (and hence $|\calE_5|$ itself) in \prettyref{eq:calE5} satisfies the same bound since $\|\theta_t'-\ttheta_t'\|=\|\theta_t-\ttheta_t\|$.
Finally, since 
$\|\ttheta_{t+1}-\theta_{t+1}\| \leq \|\calE_1\| + \|\calE_{2,\perp}\| + |\calE_3|+|\calE_4|+|\calE_5|$,
the desired \prettyref{eq:ind1-raw} follows from combining \prettyref{eq:R1}, \prettyref{eq:R2perp}, \prettyref{eq:calE3}, \prettyref{eq:calE4}, \prettyref{eq:calE5}, and \prettyref{eq:calE5-bound}.
\end{proof}



	\subsection{Proof of supporting lemmas}
	\label{sec:random-aux}

\begin{proof}[Proof of \prettyref{lmm:typical}]
Note that $\frac{1}{n} \sum_{i=1}^n Y_{i,1}^2$ is equal in distribution to $1 + \thetanorm^2 + \frac{\chi^2_n}{n}-1 +  N(0,\frac{4 \thetanorm^2}{n})$.
Then \prettyref{eq:H2} follows from the $\chi^2$-distribution tail bound \prettyref{eq:LM-lower} and the Gaussian tail bound.	
Next, since $Y_{i,1}\iiddistr \frac{1}{2} N(\pm \thetanorm,1)$ have finite moments, \prettyref{eq:H4} follows from the Chebyshev inequality.
Also, since  $\|Y_i\| \leq \|Z_i\| + \thetanorm$, where $\|Z_i\|\sim \chi_d$, \prettyref{eq:H3} follows similarly from the Chebyshev inequality. 
Finally, \prettyref{eq:Hinfty} follows simply from the union bound.	
\end{proof}

\begin{proof}[Proof of \prettyref{lmm:fsmooth}]
The Jacobian of $f_n$ is the following:
\begin{equation}
J_n(\theta) \triangleq \Expect_n\qth{YY^\top  \sech^2(\iprod{\theta}{Y})},
\label{eq:jac}
\end{equation}
which is a (random) PSD matrix. Since $0 \leq \sech\leq 1$, for any $u$, we have 
$0 \leq u^\top J_n(\theta) u = \Expect_n\qth{ \iprod{u}{Y}^2 \sech^2(\iprod{\theta}{Y})} \leq 
\Expect_n\qth{ \iprod{u}{Y}^2 } \leq u^\top J_n(0) u = u^\top \Sigma_n u$. Thus $J_n(\theta) \preceq  \Sigma_n$ for any $\theta$.
For $\tau \in [0,1]$, define 
$a_\tau \triangleq (1-\tau) a_0 + \tau a_1$.
Then
\[
f_n(a_1) - f_n(a_0)
= \frac{1}{n}\sum_{i=1}^n Y_i \int_0^1 d\tau \sech \iprod{a_\tau}{Y_i} \iprod{Y_i}{a_1- a_0}
= \sth{\int_0^1 d\tau J_n(a_\tau) } (a_1- a_0).
\]
Therefore
\begin{align*}
\| f_n(a_1) - f_n(a_0)\| 
\leq & ~  \opnorm{\int_0^1 d\tau J_n(a_\tau)} \|a_1- a_0\|\\
\leq & ~ \sup_\theta \opnorm{ J_n(\theta) } \|a_1- a_0\|	\\
\leq & ~ \opnorm{\Sigma_n} \|a_1- a_0\|.
\end{align*}
Finally, 
$\opnorm{\Sigma_n} \leq \opnorm{\Sigma} + \opnorm{\Sigma_n-\Sigma\|}$, where $\opnorm{\Sigma}=1 + \|\theta_*\|^2$.
Furthermore, 
since the entries of $Y_i$ are independent and subgaussian with parameter depending only on $\thetanorm\leq\thetaub$, by concentration of the sample covariance matrix (cf.~\cite[Exercise 4.7.3]{Vershynin-HDP}), we have
$\opnorm{\Sigma_n-\Sigma} \leq \sqrt{\frac{C d \log n}{n}}$ with probability at least $1 - \exp(-C' d \log n)$ for some constants $C$ and $C'$.
\end{proof}

\begin{proof}[Proof of \prettyref{lmm:concentration2}]
Note that	$Y_\perp$ is a $(d-1)\times n$ matrix with iid $N(0,1)$ entries. 
By the Davidson-Szarek bound \cite[Theorem II.7]{DS2001}, 
\begin{equation}
\prob{\opnorm{Y_\perp} \geq \sqrt{n}+ \sqrt{d-1} + t} \leq e^{-t^2/2},
\label{eq:ds1}
\end{equation}
which implies \prettyref{eq:ds} since $n\geq d$.

Next, note that
	\[
	\sup_{\theta\in S^{d-1}} \frac{1}{n} \sum_{i=1}^n Y_{i,1}^2 |\iprod{Y_{i,\perp}}{\theta}|^2 
	\leq \opnorm{\frac{1}{n} \sum_{i=1}^n Y_{i,\perp}Y_{i,\perp}^\top} \max_{i\in[n]} Y_{i,1}^2
	= \frac{1}{n} \opnorm{Y_{\perp}}^2 \max_{i\in[n]} Y_{i,1}^2.
	\]
The proof is completed in view of the high-probability event \prettyref{eq:Hinfty}.
\end{proof}

\begin{proof}[Proof of \prettyref{lmm:rad}]
Note that each $b_i$ is Rademacher and hence $4$-subgaussian. Thus conditioned on any realization of $x$, $\iprod{x}{b}$ is $4\|x\|^2$-subgaussian and hence
$\prob{\left|\iprod{x}{b} \right| \geq \sqrt{s} \|x\| \mid  x}  \leq 2\exp(-s/8)$ for any $t$.
	The desired \prettyref{eq:rad}
	then follows from 
	$\prob{\left|\iprod{x}{b} \right| \geq \sqrt{as} } \leq 
	\prob{\left|\iprod{x}{b} \right| \geq \sqrt{s} \|x\| }   + \prob{\|x\| \geq \sqrt{a} }$.
	Finally, \prettyref{eq:rad2} follows analogously from the union bound.	
\end{proof}

\begin{proof}[Proof of \prettyref{lmm:K}]
By dilating $q$, we can assume WLOG that $R=1$.
Recall the global assumption $\|\theta_*\|\leq \thetaub$.
Throughout the proof, unless stated to be absolute, all constants depend only on $\thetaub$ and $L_0$.
Since $B_1$ is convex, the Lipschitz constant of $D$ is given by
\[
L = 
\sup_{\theta\in B_1} \norm{\nabla D(\theta)}.
\]
It remains to bound $L$ from above with high probability, i.e.,
\begin{equation}
\sup_{\theta\in B_1} \norm{\nabla D(\theta)} \leq \sqrt{\frac{L_2 d\log ^2 n}{n}}
\label{eq:Labove}
\end{equation}
for some constant $L_2$. Furthermore,
the Hessian of $D$ is given by
\[
\nabla^2 D(\theta) = \frac{1}{n} \sum_{i=1}^n b_i Y_{i,1} Y_i Y_i^\top q''(\Iprod{\theta}{Y_{i}}).
\]
Since $|q''| \leq L_0$, we have
\begin{equation}
\sup_{\theta \in B_1} \opnorm{\nabla^2 D(\theta)} \leq L_0 \max_{i\in [n]} |Y_{i,1}| \|Y_i\|^2.
\end{equation}
In view of \prettyref{eq:Hinfty}, 
$\max_{i\in [n]} |Y_{i,1}| \leq \sqrt{\kappa \log n}$ with probability at least $1-n^{-2}$.
Furthermore, $\|Y_i\|^2 \leq 2 \|\theta_*\|^2+2 \|Z_i\|^2$. By \prettyref{lmm:LM}, for each $i$,
\begin{equation*}
\prob{\|Z_i\|^2 \geq d + 2 \sqrt{d x} + 2x} \leq \exp(-x).
\end{equation*}
Since $n/d$ is at least some absolute constant by assumption, $\prob{\|Z_i\|^2 \geq C_2 d \log n} \leq n^{-2}$ for some absolute constant $C_2$.
Therefore,
with probability at least $1-2n^{-1}$, 
\begin{equation}
\sup_{\theta \in B_1} \opnorm{\nabla^2 D(\theta)} \leq L_2 \max_{i\in [n]} |Y_{i,1}| \|Y_i\|^2
\label{eq:HessD}
\end{equation}
for some constant $L_2$, i.e.,
$\theta \mapsto \nabla D(\theta)$ is $L_2 d (\log n)^{3/2}$-Lipschitz.
Let $\Theta$ be a $\frac{1}{dn}$-net of the unit ball $B_1$, with cardinality  \cite[Corollary 4.2.13]{Vershynin-HDP}
\begin{equation}
|\Theta| \leq (1+ 2 dn)^d \leq (1+ 2 n^2)^d.
\label{eq:covering-Theta}
\end{equation}
Then on the event of \prettyref{eq:HessD}, 
\begin{equation}
\sup_{\theta \in B_1} \norm{\nabla D(\theta)} \leq 
\max_{\theta \in \Theta} \norm{\nabla D(\theta)}  + \frac{L_2 (\log n)^{3/2}}{n}.
\label{eq:HessD2}
\end{equation}


Note that
\begin{align}
\nabla D(\theta) = & ~ \frac{1}{n} \sum_{i=1}^n b_i Y_{i,1}Y_{i} q'(\Iprod{\theta}{Y_i}).		
\end{align}
Let $\calU$ be a $\frac{1}{2}$-net of $S^{d-1}$ with cardinality at most 
\begin{equation}
|\calU| \leq 5^d.
\label{eq:covering-U}
\end{equation}
Then
$\norm{\nabla D(\theta)} \leq 2 \max_{u \in \calU} \Iprod{u}{\nabla D(\theta)}$.
Recall the high-probability event $H_\infty = \{\max_{i\in[n]}  |Y_{i,1}| \leq\sqrt{\kappa \log n} \}$ defined in \prettyref{eq:Hinfty}. 
On this event, we have
\begin{align}
& ~\prob{\max_{\theta\in\Theta} 
\norm{\nabla D(\theta)} 		\geq 2 C_1 \sqrt{\frac{d \log^2 n }{n}}, H_\infty}\nonumber \\
\leq			 & ~  \prob{\max_{u\in\calU, \theta\in\Theta} \frac{1}{n}  \left| \sum_{i=1}^n b_i Y_{i,1} \Iprod{Y_i}{u} q'(\Iprod{\theta}{Y_i}) \right|
		\geq C_1 \sqrt{\frac{d \log^2 n }{n}} , H_\infty}\nonumber \\
\leq		 & ~  \prob{\max_{u\in\calU, \theta\in\Theta} \frac{1}{n}  \left| \sum_{i=1}^n b_i Y_{i,1} \Iprod{Y_i}{u} q'(\Iprod{\theta}{Y_i}) \right|
		\geq C_1 \sqrt{\frac{d \log^2 n }{n}} , H_\infty}\nonumber \\
\stepa{\leq} & ~  2\exp(-C_1 d \log n/8) |\Theta| |\calU| + \prob{\max_{u\in\calU, \theta\in\Theta} \frac{1}{n} \sum_{i=1}^n Y_{i,1}^2 \Iprod{Y_i}{u}^2 q'(\Iprod{\theta}{Y_i})^2 \geq C_1 \log n, H_\infty} \nonumber \\
\stepb{\leq} & ~  \exp(-C_3 d \log n)  + \prob{\max_{u\in\calU} \frac{1}{n} \sum_{i=1}^n \Iprod{Y_i}{u}^2  \geq \frac{C_1}{\kappa}}  \nonumber \\
\stepc{\leq} & ~  \exp(-C_3 d \log n)  + 5^d \exp(-n), \label{eq:maxD} 
\end{align}
where (a) follows from \prettyref{eq:rad2} in \prettyref{lmm:rad} with $s=\sqrt{C_1 d \log n}$; 
(b) follows from \prettyref{eq:covering-Theta}, \prettyref{eq:covering-U}, the assumption \prettyref{eq:L0}, and the event $H_\infty$;
(c) follows provided that $C_1$ is sufficiently large, in view of the fact that 
$\Iprod{Y_i}{u}^2 \leq 2 \Iprod{\theta_*}{u}^2 +2\Iprod{Z_i }{u}^2$ where 
$\Iprod{Z_i }{u} \iiddistr N(0,1)$, and the $\chi^2$-tail bound (cf.~\prettyref{eq:LM}):
\begin{equation}
	\prob{\chi^2_n \geq 5n} \leq e^{-n}.
	\label{eq:chi2}
\end{equation}
The proof of \prettyref{eq:Labove} is completed in view of \prettyref{eq:HessD2}.
\end{proof}

\begin{proof}[Proof of \prettyref{lmm:M}]
The proof is almost identical to that of \prettyref{lmm:K}, so we only mention the part that is different.
WLOG, assume that $R=1$.
First note that the Lipschitz constant of $M:\reals^d \to \reals^{d-1}$ (with respect to the Euclidean norm) is bounded by
\begin{equation}
\Lip(M) \leq L_0 \frac{1}{n} \sum_{i=1}^n \|Y_{i,\perp}\| |Y_{i,1}| (\|Y_{i,\perp}\| + |Y_{i,1}|).
\label{eq:LipM}
\end{equation}
Similar to the argument that leads to \prettyref{eq:HessD2}, we conclude that with probability at least $1-n^{-1}$ 
$\Lip(M)  \leq L_2 d \log n$ for some constant $L_2$.

Next let $\Theta$ be a $\frac{1}{dn}$-net of the unit ball in $\reals^d$ and 
let $\calU$ be a $\frac{1}{2}$-net of the unit sphere in $\reals^{d-1}$.
It suffices to bound $\max_{u\in\calU,\theta\in\Theta} \Iprod{u}{M(\theta)}$.
The rest of the proof is identical to that of \prettyref{eq:maxD}.	
\end{proof}

\begin{proof}[Proof of \prettyref{lmm:gg}]
	Note that $y \mapsto T(x,y)$ is an odd function and $T(x,y) \geq 0$ for $y \geq 0$.
	For the upper bound, note that
	$\partial_y T(x,y)|_{y=0} = \tanh^2(x)$ and 
	$\partial_y^3 T(x,y) = 3 (\sech(x+y)^4 + \sech(x-y)^4) -2(\sech(x+y)^2 + \sech(x-y)^2)$.
	Since $\sup_{0 \leq r \leq 1} (3r^4-2r^2) = 1$, we have $\partial_y^3 T(x,y) \leq 2$ for all $x,y$.
	Thus \prettyref{eq:gg} follows from the Taylor expansion of $T(x,y)$ at $y=0$ and the fact that $\tanh(x)^2\leq x^2$.	
	
Finally, \prettyref{eq:hh} follows from the $1$-Lipschitz continuity of $g$, since $g'(z) = 1-\sech^2(z)$.
\end{proof}

\begin{proof}[Proof of \prettyref{lmm:Q}]
	Recall that $Q(x,y) = \frac{1}{2x} (\tanh(y+x)-\tanh(y-x))$. Then
	$|Q(x,y)| \leq 1$ and 
	$|\partial_y Q(x,y)| \leq 1$ and 
	follows from the $1$-Lipschitz continuity of $\tanh$ and $\tanh'$, respectively.
	Finally,
	by Taylor's theorem, we have
	$\tanh(y+x)-\tanh(y-x) = 2 x \tanh'(y) + x \int_0^1 dz (1-z) \{\tanh''(y+xz)+\tanh''(y-xz)\}$.
	Therefore
	$\partial_x Q(x,y) = 
	\frac{1}{2} \frac{\partial}{\partial x} \int_0^1 dz (1-z) \{\tanh''(y+xz)+\tanh''(y-xz)\}
	= \frac{1}{2} \int_0^1 dz z(1-z) \{\tanh'''(y+xz)-\tanh'''(y-xz)\}	$.
	Since $|\tanh'''| \leq 2$, we have
	$|\partial_x Q(x,y)| \leq 2\int_0^1 dz z(1-z)= \frac{1}{3}$.
\end{proof}

\section{Proofs in \prettyref{sec:mle}}
	\label{sec:pf-mle}

	
	
	\begin{proof}[Proof of \prettyref{lmm:hessMLE}]
	Since $\ell(\theta,\theta_*) \leq \delta$, WLOG, assume that $\|\theta-\theta_*\| \leq \delta$. 
Note that 
$\nabla^2 \ell_n(\theta) = - I + J_n(\theta)$, where $J_n(\theta)$ is the Jacobian of $f_n$ given in \prettyref{eq:jac}. Then 
\[
I + \nabla^2 \ell_n(\theta) = J_n(\theta) = \Expect_n\qth{YY^\top  \sech^2\iprod{\theta}{Y}},
\]
which is PSD with probability one. Therefore it remains to bound the maximum eigenvalue of $J_n$ from above uniformly in a neighborhood of $\theta_*$. We do so in two steps.

\medskip
\emph{Step 1: Population version.}
By assumption, $\|\theta_*\|\geq 100 \delta$ for sufficiently large $n$ and hence $\iprod{\theta}{\theta_*} \geq 0$. 
Consider the expectation of $J_n$:
\[
J(\theta) \triangleq \Expect[J_n(\theta)] = \Expect\qth{YY^\top  \sech^2 \Iprod{\theta}{Y}},
\]
which is a PSD matrix. We show that
\begin{equation}
\sup_{\|\theta-\theta_*\|\leq\delta} \sup_{\|u\|=1} u^\top J(\theta) u \leq e^{-c \|\theta_*\|^2}.
\label{eq:hess-pop}
\end{equation}
Consider two cases:

\textbf{Case 1: $u \perp \theta$.} Then 
$|\Iprod{u}{\theta_*}| = |\Iprod{u}{\theta_*-\theta}| \leq \|\theta-\theta_*\| \leq \delta$.
By the independence of 
$\Iprod{u}{Z}$ and $\Iprod{\theta}{Z}$, we have
\begin{equation}
u^\top J(\theta) u = \Expect[\Iprod{u}{Y}^2 \sech^2 \Iprod{\theta}{Y}] = 
\Expect[\Iprod{u}{Y}^2] \Expect[\sech^2 \Iprod{\theta}{Y}].
\label{eq:uJu}
\end{equation}
Here $\Expect[\Iprod{u}{Y}^2] = \Iprod{u}{\theta_*}^2 + 1 \leq 1 +\delta^2$.
Furthermore,
let $\eta \triangleq \theta/\|\theta\|$. Then 
$U \triangleq \Iprod{\eta}{Y} \sim \frac{1}{2} N(\pm s,1)$, where 
$s  = \Iprod{\eta}{\theta_*}$ satisfies $|s -\|\theta\|| = |\Iprod{\eta}{\theta_*-\theta}| \leq \delta$ and hence $s \geq \|\theta_*\| - 2 \delta$.
By a change of measure (\prettyref{lmm:com}), we have
\begin{align}
\Expect[\sech^2 \Iprod{\theta}{Y}]
= & ~ 	\Expect[\sech^2 (\|\theta\|U)] \nonumber \\
= & ~ 	\Expect[\cosh(s W) \sech^2 (\|\theta\| W)] e^{-s^2/2}, \quad W \sim N(0,1) \nonumber \\
\leq & ~ 	F(s,\|\theta\|) e^{-\|\theta_*\|^2/4}, \quad W \sim N(0,1). \label{eq:Fsb}
\end{align}
Put 
$F(a,b) \triangleq \Expect[\cosh(a W) \sech^2 (b W)]$. 
Straightforward calculation shows that $\frac{\partial F(a,b)}{\partial b} \leq 0$ and $\frac{\partial F(a,b)}{\partial a} \geq 0$, i.e., $F(a,b)$ is increasing in $a$ and decreasing in $b$. 
Write $b=\|\theta\|$. Since $|s -b|  \leq \delta$, we have
\[
F(s,b) \leq F(b+\delta,b) = \underbrace{\Expect[\cosh(\delta W) \sech (b W)]}_{\termI} + \underbrace{\Expect[\sinh(\delta W) \sinh(b W) \sech^2 (b W)]}_{\termII}.
\]
The first term satisfies 
$\termI \leq \Expect[\cosh(\delta W)] = e^{\delta^2/2}$. For the second term, 
using the fact that $\tanh(x) \leq x$ when $x \geq 0$, we get the following bound that is, crucially, proportional to $\thetanorm$:
\[
\termII \leq b~\Expect[W \sinh(\delta W)] = b \delta e^{\delta^2/2} \leq 2 \|\theta_*\| \delta e^{\delta^2/2}.
\]
Combining the above with \prettyref{eq:Fsb} and \prettyref{eq:uJu}, we get
\begin{align*}
u^\top J(\theta) u 
\leq  & ~  (1+\delta^2) (1+2 \|\theta_*\| \delta) e^{\delta^2/2-\|\theta_*\|^2/4}\\
\leq & ~ 	e^{3\delta^2/2 + 2 \|\theta_*\| \delta -\|\theta_*\|^2/4} \leq
e^{3\delta^2/2 + 2 \|\theta_*\| \delta -\|\theta_*\|^2/4} \leq e^{ -\|\theta_*\|^2/16}.
\end{align*}


\textbf{Case 2: $u \mathrel{/\mkern-5mu/} \theta$.} WLOG, assume $u=\eta$. Entirely analogously to the previous case, we have
\[
u^\top J(\theta) u \leq \Expect[W^2 \cosh(s W) \sech^2 (\|\theta\| W)] e^{-\|\theta_*\|^2/4},
\]
and 
\begin{align*}
& ~ \Expect[W^2 \cosh(s W) \sech^2 (\|\theta\| W)] \\
\leq  & ~  \Expect[W^2 \cosh((\|\theta\|+\delta) W) \sech^2 (\|\theta\| W)] \\
=&~\Expect[W^2 \cosh(\delta W) \sech (b W)] + \Expect[W^2 \sinh(\delta W) \sinh(b W) \sech^2 (b W)]\\
\leq & ~ \Expect[W^2 \cosh(\delta W)] + b \Expect[W^3 \sinh(\delta W)] \\
= &~	(1+\delta^2) e^{\delta^2/2} + \|\theta\| \delta (3+\delta^2) e^{\delta^2/2}.
\end{align*}
Therefore $u^\top J(\theta) u \leq e^{ -\|\theta_*\|^2/50} $.

Finally, we combine the two cases. For an arbitrary unit vector $u$, let $u=\cos\phi \eta + \sin\phi v$ for some $v\perp \eta$. 
Then $\Iprod{v}{Y}$ and $\Iprod{\eta}{Y}$ are independent and hence
\begin{align*}
u^\top J(\theta) u 
= & ~ \cos^2\phi\Expect[\Iprod{\eta}{Y}^2 \sech^2 \Iprod{\theta}{Y}] + \sin^2\phi\Expect[\Iprod{v}{Y}^2 \sech^2 \Iprod{\theta}{Y}] \\
& ~ + 2\cos\phi\sin\phi \Expect[\Iprod{v}{Y} \Iprod{\eta}{Y}  \sech^2 \Iprod{\theta}{Y}]  \\
= & ~ \cos^2\phi\Expect[\Iprod{\eta}{Y}^2 \sech^2 \Iprod{\theta}{Y}] + \sin^2\phi\Expect[\Iprod{v}{Y}^2 \sech^2 \Iprod{\theta}{Y}] \leq e^{ -\|\theta_*\|^2/50},
\end{align*}
where the second equality follows from 
\[
\Expect[\Iprod{v}{Y} \Iprod{\eta}{Y}  \sech^2 \Iprod{\theta}{Y}] = \Expect[\Iprod{v}{Y}]\Expect[\Iprod{\eta}{Y}  \sech^2 \Iprod{\theta}{Y}] =  0
\]
 thanks to independence. This yields the desired \prettyref{eq:hess-pop}.

\medskip
\emph{Step 2: Concentration.}
We show that with probability at least $1-2n^{-1}$, 
\begin{equation}
\sup_{\|\theta-\theta_*\|\leq\delta} \opnorm{J_n(\theta)-J(\theta)} \leq \sqrt{\frac{C_0 d \log n}{n}}.
\label{eq:hess-dev}
\end{equation}
Since $\sech^2$ is 1-Lipschitz, we have
\begin{align*}
\opnorm{J_n(\theta) - J_n(\theta')}
\leq & ~  \opnorm{ \Expect_n\qth{YY^\top  \big|\sech^2 \Iprod{\theta}{Y} - \sech^2 \Iprod{\theta'}{Y} \big| }}\\
\leq & ~ \|\theta-\theta'\| \cdot \opnorm{\Expect_n\qth{YY^\top  \cdot \|Y\|}}	\\
\leq & ~ \|\theta-\theta'\| \cdot \Expect_n[\|Y\|^3]. 
\end{align*}
Therefore on the event $F_1$ in \prettyref{eq:H3}, which has probability at least $1-n^{-4}$, 
$\theta \mapsto J_n(\theta)$ is $C_4 d^{3/2}$-Lipschitz with respect to the $\ell_2$-norm and the $\opnorm{\cdot}$-norm, where $C_4$ is a constant depending only on $\thetaub$.
Let 
$\calE$ be an $\epsilon$-net of $B(\theta_*,\delta)$ with $\epsilon=\frac{\delta}{\sqrt{d^3 n}} $
and $|\calE| \leq (1+2\frac{\delta}{\epsilon})^d \leq \exp(C_5 d \log(n))$. 
Let $\calU$ be a $\frac{1}{2}$-net of $S^{d-1}$ with cardinality at most $|\calU| \leq 5^d$. 
Then
\begin{equation}
\sup_{\|\theta-\theta_*\|\leq\delta} \opnorm{J_n(\theta)-J(\theta)} \leq 
2 \sup_{\theta \in \calE} \sup_{u \in \calU} u^\top (J_n(\theta)-J(\theta))u + \frac{2 C_4}{\sqrt{n}}.
\label{eq:hess-dev1}
\end{equation}

Fix $u\in \calU$ and $\theta \in \calE$, put $U = \Iprod{u}{Y}^2 \sech^2 \Iprod{\theta}{Y}$ and $U_i = \Iprod{u}{Y_i}^2 \sech^2 \Iprod{\theta}{Y_i}$. 
Note that 
$\Iprod{u}{Y}^2$ is sub-exponential with $\|\Iprod{u}{Y}^2\|_{\psi_1} \leq C_1=C_1(\thetaub)$. By the moment characterization of sub-exponentiality (cf.~\cite[Proposition 2.7.1]{Vershynin-HDP}), since $|\sech| \leq 1$, we conclude that $\|U\|_{\psi_1} \leq C_2=C_2(\thetaub)$.
By Bernstein's inequality (c.f.~\cite[Theorem 2.8.1]{Vershynin-HDP}), 
\begin{align*}
\prob{|u^\top (J_n(\theta)-J(\theta))u| \geq  \frac{t}{\sqrt{n}}}
= & ~ \prob{|\Expect_n[U] - \Expect[U]| \geq  \frac{t}{\sqrt{n}} } \\
\leq & ~ 	2\exp\pth{-c \min\sth{\frac{t^2}{\|U\|_{\psi_1}^2}, \frac{t \sqrt{n}}{\|U\|_{\psi_1}} }}.
\end{align*}
for some absolute constant $c$.
Choosing $t = \sqrt{C_3d \log n}$ with $C_3=C_3(\thetaub)$ sufficiently large, and in view of the assumption that
$n = \Omega(d \log n)$, we conclude that
\[
\prob{|u^\top (J_n(\theta)-J(\theta))u| \geq  \frac{t}{\sqrt{n}}}
\leq 2 \exp\pth{-2 C_5 d \log n}.
\]
The proof of \prettyref{eq:hess-dev} is completed by applying the union bound over $\theta\in\calE$ and $u\in\calU$ in \prettyref{eq:hess-dev1}.

Finally, since $\thetanorm^2  = \Omega(\sqrt{\frac{d \log n}{n}}) $, combining \prettyref{eq:hess-dev} with \prettyref{eq:hess-pop} yields the lemma.		
	\end{proof}

\appendix

\section{Auxiliary results}
	\label{app:lemma}
	
\begin{lemma}[{\cite[Lemma 1]{LM00}}]
\label{lmm:LM}	
For any $x \geq 0$,
\begin{align}
\prob{\chi^2_n \geq 2n+3x} \leq \prob{\chi^2_n - n \geq 2 \sqrt{nx}+2x} \leq &~\exp(-x), \label{eq:LM} \\
\prob{\chi_n^2 \leq n - 2 \sqrt{ n x}} \leq & ~ \exp(-x).\label{eq:LM-lower}
\end{align}
\end{lemma}

\begin{lemma}
\label{lmm:Kt}	
Let $\epsilon,\delta>0$. 
Assume that the sequence $\{K_t\}$ satisfies $K_0=0$ and $K_{t+1} \leq (1+\epsilon) K_t + \delta$. Then for all $t\geq 0$,
\[
K_t \leq \frac{\delta}{\epsilon} \sth{(1+\epsilon)^t-1}.
\]	
\end{lemma}
\begin{proof}
	This follows simply from induction on $t$.
\end{proof}

The following lemma is useful for analyzing the rate of convergence:
	\begin{lemma}[{\cite[Appendix A]{PW14a}}]
	\label{lmm:rate}
Let
\[
x_{t+1} \leq x_{t} - h(x_t), \quad x_0>0
\]
where $h:\reals_+ \to \reals_+$ is a continuous increasing function with $h(0)=0$ and $h(x) < x$ for all $x\in(0,x_0)$. 
Then $\{x_t\}\subset \reals_+$ is a monotonically decreasing sequence converging to the unique fixed point at zero as $n\diverge$. 
Furthermore,
\begin{equation}
	x_t \leq G^{-1}(t), \quad t \geq 1
	\label{eq:tn-rate}
\end{equation}	
where $G: [0,1] \to \reals_+$ by $G(x) = \int_x^{x_0} \frac{1}{h(\tau)} \diff \tau$. 
\end{lemma}

The proof of \prettyref{lmm:fprop} and \prettyref{lmm:FG} on the properties of the population EM map relies on the following auxiliary results. 
\begin{lemma}
\label{lmm:estimator}	
	Let $Y = \alpha V + \beta W$, where $\alpha,\beta \geq 0$ and $W\sim N(0,1)$. 
	Let $\hat V(y) = \Expect[V|Y=y]$. Then
	\begin{enumerate}
		
		\item $\hat V$ is an increasing function.
		\item If $V$ has a symmetric distribution in the sense that $V \eqlaw -V$, then $\hat V$ is an odd function.
	
	\end{enumerate}
\end{lemma}
\begin{proof}
By scaling, it suffices to consider $\alpha=\beta=1$.
	The first item follows from the well-known fact that $\frac{d}{dy}\hat V(y) = \Var(V|Y=y) \geq 0$ (see, \eg, \cite[Eq.~(131)]{mmse.functional.IT}), while the second is due to the fact that $W$ has a symmetric distribution.
\end{proof}

We also need the following bound on the Mill's ratio due to  Ito and McKean \cite[Exercise 1, p.~851]{SW-empirical}
\begin{lemma}
\label{lmm:ito}	
	Let $\varphi(x) \triangleq \frac{1}{\sqrt{2\pi}} \exp(-\frac{x^2}{2})$ denote the standard normal density and $\bar\Phi(x) = \int_{x}^\infty \varphi(t)dt$ the normal tail probability. Then	
	\begin{equation}
	\frac{\bar{\Phi}(x)}{\varphi(x)} \leq \frac{2}{\sqrt{2+x^2}+x},
	\label{eq:ito}
	\end{equation}
\end{lemma}

We will invoke Stein's lemma repeatedly, which is included below for completeness:
\begin{lemma}
\label{lmm:stein}	
	Let $W\sim N(0,1)$ and $f$ be a differentiable function such that $\expect{|f'(W)|}<\infty$. Then
	\begin{equation}
	\expect{Wf(W)} = \Expect[f'(W)].
	\label{eq:stein}
	\end{equation}
\end{lemma}
The following useful result is simply a change of measure from the symmetric 2-GM to the standard normal:
\begin{lemma}
\label{lmm:com}	
Let $V\sim P_s = \frac{1}{2} N(\pm s,1)$ as in \prettyref{eq:Ptheta} and let $Z\sim N(0,1)$. Then
\[
\Expect[f(V)] = \Expect[f(Z)\cosh(sZ)] e^{-s^2/2}.
\]
\end{lemma}
\begin{proof}
	This follows from $\frac{p_s(z)}{\varphi(z)} = \cosh(s z) e^{-s^2/2}$.
\end{proof}

\section{Minimax rates}
\label{app:minimax}


\begin{theorem}
\label{thm:minimax}
For any $d \geq 2$ and $n\in\naturals$ and $s\geq 0$,
		\begin{equation}
	\inf_{\hat\theta} \sup_{\|\theta_*\| = s} \Expect_{\theta^*}[\ell(\hat\theta,\theta_*)] \asymp \min\sth{
	\frac{1}{s} \pth{\frac{d}{n} + \sqrt{\frac{d}{n}}}  + \sqrt{\frac{d}{n}}, s}.
	\label{eq:minimax1}
	\end{equation}	
Furthermore, for any $d,n\in\naturals$ and $r\geq 0$,
		\begin{equation}
	\inf_{\hat\theta} \sup_{\|\theta_*\| \leq r} \Expect_{\theta^*}[\ell(\hat\theta,\theta_*)] \asymp \min\sth{\pth{\frac{d}{n}}^{\frac{1}{4}} + \sqrt{\frac{d}{n}}, r}.
	\label{eq:minimax2}
	\end{equation}	
\end{theorem}
Before proving \prettyref{thm:minimax}, we note that the rate in \prettyref{eq:minimax1} behaves as 
\begin{equation}
\inf_{\hat\theta} \sup_{\|\theta_*\| = s} \Expect_{\theta^*}[\ell(\hat\theta,\theta_*)]  \asymp \begin{cases}
  s & s \leq \pth{\frac{d}{n}}^{\frac{1}{4}} \\
 \frac{1}{s} \sqrt{\frac{d}{n}} & \pth{\frac{d}{n}}^{\frac{1}{4}} \leq s \leq 1 \\
\sqrt{\frac{d}{n}} & s \geq 1 \\
\end{cases}
\label{eq:minimax11}
\end{equation}
for $d\leq n$ and 
\begin{equation}
\inf_{\hat\theta} \sup_{\|\theta_*\| = s} \Expect_{\theta^*}[\ell(\hat\theta,\theta_*)]  \asymp \begin{cases}
  s & s \leq \sqrt{\frac{d}{n}}  \\
 \sqrt{\frac{d}{n}} & s \geq \sqrt{\frac{d}{n}} \\
\end{cases}
\label{eq:minimax12}
\end{equation}
for $d\geq n$. The latter case coincides with the $\ell_2$-rate of the Gaussian location model.

\paragraph{Upper bound.}
As before, denote $s = \thetanorm$ and $\eta_*=\theta_*/s$.
Let $\epsilon \triangleq \max\{\sqrt{\frac{d}{n}}, \frac{d}{n}\}$. 
Since the trivial estimator $\hat\theta=0$ achieves $\ell(\hat\theta,\theta_*) = s$, it remains to show the upper bound $C_0 \sqrt{\epsilon}$ under the assumption that $\thetanorm \geq C_1 \sqrt{\epsilon}$, for some universal constants $C_0,C_1$.
Let $\hat{\lambda}$ and $\hat\eta$ denote the top eigenvalue and the associated eigenvector (of unit norm) of the sample covariance matrix $\hat\Sigma \triangleq \Expect_n[YY^\top]$. 
Let $\Sigma=\Expect[YY^\top] = I_d + \theta_*\theta_*^\top$. 
Consider the estimator:
\begin{equation}
\hat\theta = \hat s \hat\eta, \quad \hat s = \sqrt{(\hat{\lambda}-1)_+},
\label{eq:spectral}
\end{equation}
where $(x)_+ \triangleq \max\{0,x\}$ 
for any $x\in\reals$.
To analyze its loss, recall that $Y=X\theta^*+Z$, where $X$ is Rademacher and independent of $Z\sim N(0,I_d)$.
Since
$\Expect_n[YY^\top] = \theta_*\theta_*^\top+\Expect_n[ZZ^\top] + \theta_*(\Expect_n[XZ] )^\top + (\Expect_n[XZ] ) \theta_*^\top$, we have
$\hat\Sigma-\Sigma \overset{\text{law}}{=} \Delta + \frac{1}{\sqrt{n}} (\theta_* w^\top + w \theta_*^\top) $, where $w\sim N(0,I_d)$ and $\Delta \triangleq \Expect_n[ZZ^\top] -I_d$.
Consequently,
$\Opnorm{\hat\Sigma-\Sigma} \leq \opnorm{\Delta} +  \frac{2}{\sqrt{n}} \|w\| \thetanorm$.
By Davis-Kahan's perturbation bound, we have
\[
\ell(\hat\eta,\eta_*) \leq 4 \frac{\Opnorm{\hat\Sigma-\Sigma}}{s^2}
\]
Furthermore, by Weyl's inequality, $|\hat \lambda -1 - s^2| \leq \Opnorm{\hat\Sigma-\Sigma}$ and thus
\[
|\hat s - s| = \frac{|(\hat \lambda-1)_+^2 - s^2|}{(\hat \lambda-1)_+ + s} \leq 
 \frac{|\hat \lambda-1 - s^2|}{s} \leq 
\frac{\Opnorm{\hat\Sigma-\Sigma}}{s}.
\]
 Applying the triangle inequality and combining the last two displays, we obtain
\[
\ell(\hat\theta,\theta_*) \leq |\hat s -s| + s \ell(\hat\eta,\eta_*) \leq 
5 \frac{\Opnorm{\hat\Sigma-\Sigma}}{s^2}.
\]
Finally, since $\Expect[\opnorm{\Delta}] \leq  C\epsilon$ \cite[Theorem 4.7.1]{Vershynin-HDP} for some universal constant $C$ and $\Expect[\|w\|] \leq \sqrt{d}$, taking expectation on both sides, we have
\[
\Expect \ell(\hat\theta,\theta_*) \leq 5 \frac{\Expect\Opnorm{\hat\Sigma-\Sigma}}{s} \leq C'\pth{\frac{\epsilon}{s}    + \sqrt{\frac{d}{n}}}
\]
for some universal constant $C'$.
This proves the upper bound part of \prettyref{eq:minimax1}, and, upon taking the supremum over $s \leq r$, that of \prettyref{eq:minimax2} (since the estimator \prettyref{eq:spectral} does not depend on $\thetanorm$).

\paragraph{Lower bound.}
Recall that 
$P_\theta = \frac{1}{2} N(-\theta,I_d) + \frac{1}{2} N(\theta,I_d)$; in particular, $P_0=N(0,I_d)$.
Then straightforward calculation shows that the $\chi^2$-divergence is $\chi^2(P_\theta\|P_0) = \cosh(\|\theta\|^2) - 1$.
Since $D(P\|Q) \leq \log (1+\chi^2(P\|Q))$, the KL divergence is upper bounded by
\begin{equation}
D(P_\theta\|P_0) \leq \log \cosh(\|\theta\|^2).
\label{eq:KL}
\end{equation}
Note that $\log \cosh(x) \asymp \min\{x,x^2\}$ for $x \geq 0$.
Applying Le Cam's method (two-point argument) to $\theta_*=0$ versus $\theta_*=\epsilon$, with $\epsilon = c_0 \min\{r,n^{-1/4}\}$ for some sufficiently small constant $c_0$, we obtain the desired lower bound in \prettyref{eq:minimax2} for $d=1$.

Next we consider $d\geq 2$. It suffices to prove the lower bound part of \prettyref{eq:minimax1}, which yields that of \prettyref{eq:minimax2} by taking the supremum over $s \leq r$.
Furthermore, since the rate for the Gaussian location model (which is $s \wedge \sqrt{\frac{d}{n}}$) constitutes a lower bound for the Gaussian mixture model, this proves \prettyref{eq:minimax12} as well as the last case of \prettyref{eq:minimax11}. So next we focus on $2 \leq d \leq n$ and $s \leq 1$.

Let $c_0$ be some small absolute constant.
Let $v_1,\ldots,v_M$ be a $c_0$-net for the unit sphere $S^{d-2} \cap \reals_+^{d-1}$, such that 
(a) $\|v_i\|=1$; 
(b) $\ell(v_i,v_j)=\|v_i-v_j\| \geq c_0$ for any $i\neq j$;
(c) $M \geq \exp(C_0d)$ for some absolute constant $C_0$.
Now define 
$u_0,\ldots,u_M \in \reals^d$ by $u_0 = e_1=[1,0,\ldots,0]$ and $u_i= [1-\epsilon^2, \epsilon v_i]$ for $i\in[M]$, where $\epsilon = c_1 \min\{1,\frac{1}{s^2} \sqrt{\frac{d}{n}}\}$ for some small constant $c_1$. 
Then $\ell(u_i,u_j)  \geq c_0 \epsilon$ for any distinct $i,j\in[M]$ and 
$\ell(u_i,u_0)  \leq 2 c_0 \epsilon$ for any $i\in [M]$. 
Finally, let $\theta_i=s u_i$ for $i=0,\ldots,M$.
By the key result \prettyref{lmm:localKL} below, the KL radius of $\{P_{\theta_i}:i\in[M]\}$ is at most
\[
\max_{i\in[M]} D(P_{\theta_i}\|P_{\theta_0}) \leq C_1 s^4 \epsilon^2
\]
for some absolute constant $C_1$.
Applying Fano's method \cite{YB99} yields a lower bound that is a constant factor of $\epsilon s \asymp \min\{s,\frac{1}{s} \sqrt{\frac{d}{n}}\}$.



It remains to prove the following result on the local behavior of KL divergence in the 2-GM model.
\begin{lemma}
\label{lmm:localKL}	
	Let $0\leq s\leq 1$. Then there exists a universal constant $C$, such that for any $d\geq 1$ and $u,v\in S^{d-1}$,
	\begin{equation}
	D(P_{su} \| P_{sv}) \leq C \ell(u,v)^2 s^4.
	\label{eq:local}
	\end{equation}
\end{lemma}
\begin{remark}
	The result \prettyref{eq:local} can be interpreted in two ways. 
First, by the local expansion of the KL divergence, we have $D(P_{\theta'} \| P_{\theta}) = O(\|\theta-\theta'\|^2 I(\theta))$, where 
$I(\theta)$ is the Fisher information at $\theta$, which, in the 2-GM model, behaves as $\|\theta\|^2$ for small $\theta$ (see \prettyref{rmk:rate});
however, this intuition does not directly lead to the desired dimension-free bound.
Additionally, \prettyref{eq:local} can be ``anticipated'' by drawing analogy to the covariance model: Suppose the latent variable in the mixture model is standard normal instead of Rademacher. Then $D(P_{su} \| P_{sv}) = D(N(0,I+s^2uu^\top) \| N(0,I+s^2 vv^\top)) = \frac{s^4}{2(1+s^2)} \Fnorm{uu^\top-vv^\top}^2\asymp s^4 \ell(u,v)^2$,
where the second identity is from \cite[Eqn.~(52)]{CMW12}.
\end{remark}

\begin{proof}[Proof of \prettyref{lmm:localKL}]
	First of all, by symmetry, it suffices to show 
		\begin{equation}
	D(P_{su} \| P_{sv}) \leq C \|u-v\|^2 s^4.
	\label{eq:local1}
	\end{equation}	
	Next, let $\delta = \|u-v\| \in [0,\sqrt{2}]$. 
	By the rotational invariance of the normal distribution, we can and shall assume
	$v=e_1$ and $u$ satisfies $|u_1-1|\leq \delta$ and $\|u_\perp\|\leq \delta$, where $u_\perp=(u_2,\ldots,u_d)$.
	Let $Q=Q_{Y_1,\ldots,Y_d}=P_{sv}$ and $P=P_{Y_1,\ldots,Y_d}=P_{sv}$. 
	Then 	$Q = P_s \otimes N(0,I_{d-1})$ is a product distribution, while $P$ is not, since under $P$, $Y_1,\ldots,Y_d$ are dependent through the common label; this is where the majority of the technical difficulty of this proof comes from.
	Next we use the chain rule to evaluate the KL divergence:
	\[
	D(P_{Y_1,\ldots,Y_d}\|Q_{Y_1,\ldots,Y_d}) = 
	\underbrace{D(P_{Y_1}\|Q_{Y_1})}_{\termI} + \underbrace{\Expect[D(P_{Y_\perp |Y_1}\|N(0,I_{d-1})]}_{\termII},
	\]
	where we used the fact that $Y_\perp$ is standard normal and independent of $Y_1$ under $Q$, and the expectation in (II) is taken over $P_{Y_1}$.
	In what follows we show that both terms are $O(s^4 \delta^2)$.
	
	\medskip
	\textbf{Bounding (I):}
	Let $u_1=s+\epsilon$, where $|\epsilon| \leq s \delta$.
	Then $\termI = D(P_{s+\epsilon} \|P_s)$.
	Recall $p_\theta(y)$ given in \prettyref{eq:ptheta} denotes the density function of $P_\theta$.
	In one dimension, we have $p_\theta(y)= e^{-\frac{\theta^2}{2}}  \varphi(y) \cosh(\theta y)$.
	Then
	\begin{align*}
	\termI
	\leq & ~ \chi^2(P_{s+\epsilon} \|P_s) \\
	\stepa{\leq} & ~ e^{\frac{s^2}{2}} \int \varphi(y) [e^{-\frac{(s+\epsilon)^2}{2}}  \cosh((s+\epsilon)y) - e^{-\frac{s^2}{2}}  \cosh(s y)]^2	\\
	\stepb{= } & ~ e^{\frac{s^2}{2}} (\cosh \left(s^2\right)-2 \cosh (s (s+\epsilon ))+\cosh \left((s+\epsilon )^2\right)) \stepc{\leq} C_1 s^2 \epsilon^2 \leq C_1 s^4 \delta^2,
	\end{align*}
	where (a) is due to $\cosh \geq 1$; (b) follows from the facts that $\int \varphi(y)  \cosh(sy) = e^{s^2/2}$, $\int \varphi(y)  \cosh(sy)^2 = e^{s^2} \cosh(s^2)$, and $2\cosh(a)\cosh(b) =\cosh(a+b)+\cosh(a-b)$;
	(c) is by Taylor expansion since $0\leq |\epsilon| \leq \sqrt{2} s \leq \sqrt{2}$, where $C_1$ is some universal constant.

	\medskip
	\textbf{Bounding (II):}
	Let $Y=(Y_1,Y_\perp)$ and $Y_\perp=(Y_2,\ldots,Y_d)$. Under $P$, we can write $Y_i = R_i + Z_i$, 
	where $R_i = s u_i \cdot B$, $B$ is Rademacher and independent of $Z_i \iiddistr N(0,1)$.
	Therefore $P_{Y_\perp|Y_1} = P_{R_\perp|Y_1} * N(0,I_{d-1})$
	is a Gaussian location mixture (convolution). 	
	Recall the Ingster-Suslina identity \cite{IS03}: 
	for any distribution $\mu$ on $\reals^d$,
	\[
	\chi^2(\mu * N(0,I_{d})\|N(0,I_{d}) = \Expect[\exp(\Iprod{X}{\tilde X})]-1,
	\]
	where $X,\tilde X \iiddistr \mu$.
	Then we have
	\[
	\termII \leq \Expect[\chi^2(P_{Y_\perp|Y_1}\|N(0,I_{d-1}))] = \Expect[\exp(\Iprod{R_\perp}{\tilde R_\perp})]-1,
	\]
	where $\tilde R_\perp$ is an independent copy of $R_\perp$ conditioned on $Y_1$.
	Note that $\|R_\perp\| \leq s \|u_\perp\| \leq s \delta$ almost surely. 
	Then $|\Iprod{R_\perp}{\tilde R_\perp}| \leq (s \delta)^2 \leq 2$. Therefore by Taylor expansion, we have
	\[
	\Expect[\exp(\Iprod{R_\perp}{\tilde R_\perp})]-1 \leq 
	\Expect[\Iprod{R_\perp}{\tilde R_\perp}] + C_2 (s \delta)^4,
	\]
	where $C_2$ is some universal constant.
	By linearity, we have
	\begin{align*}
	\Expect[\Iprod{R_\perp}{\tilde R_\perp}]
	= & ~  \sum_{i=2}^d \Expect[R_i \tilde R_i] 
	= \sum_{i=2}^d \Expect[\Expect[R_i|Y_1] \Expect[\tilde R_i|Y_1]] \\
	\stepa{=} & ~  \sum_{i=2}^d \Expect[\Expect[R_i|Y_1]^2] 
	\stepb{=} s^2 \sum_{i=2}^d u_i^2 \Expect[\Expect[B|Y_1]^2] \\
	\stepc{=} & ~  s^2 \delta^2 \Expect[\tanh(u_1 Y_1)^2] 
	\stepd{\leq} 4 s^4 (1+4s^2) \delta^2  \leq 40 s^4 \delta^2,
	\end{align*}
	where (a) is because of $\tilde R_i$ is a conditional independent copy of $R_i$;
	(b) is due to $R_i = s u_i B$;
	(c) is by $\|u_\perp\|=\delta$ and the conditional mean is given by \prettyref{eq:condmean};
	(d) is by $|\tanh(x)|\leq |x|$ and $|u_1| \leq s(1+\delta) \leq 2s$.
		
	\medskip
	Finally, combining (I) and (II) completes the proof of \prettyref{eq:local1}.
\end{proof}

\section*{Acknowledgment}
The authors are grateful for Yuxin Chen for helpful discussions on \cite{chen2018gradient} and Natalie Doss for pointing out \cite{ho2016convergence}.


\newcommand{\etalchar}[1]{$^{#1}$}

\end{document}